\documentclass{article}
\usepackage{graphicx} 
\usepackage{amsthm,amsmath,amsfonts,amssymb}

\usepackage[a4paper, margin=1in]{geometry} 

\usepackage{enumitem}

\usepackage{dsfont}
\usepackage{subfig}

\usepackage{hyperref}
\hypersetup{
    colorlinks=true,
    linkcolor=blue,
    filecolor=magenta,      
    urlcolor=cyan,
    pdftitle={Overleaf Example},
    pdfpagemode=FullScreen,
    }

\usepackage{cleveref}
\usepackage{ytableau}
\usepackage{multicol}

\usepackage{matlab-prettifier}

\def\V#1{{\mathbf #1}}
\def\floor#1{{\lfloor #1 \rfloor}}
\def\GUE{\operatorname{GUE}}
\def\GOE{\operatorname{GOE}}
\def\Bern{\operatorname{Bern}}

\def\Unif{\operatorname{Unif}}
\def\Haar{\operatorname{Haar}}
\def\Var{\operatorname{Var}}

\def\P{\mathbb{P}}
\def\E{\mathbb{E}}
\def\B{\operatorname{B}}

\def\SO{\operatorname{SO}}

\def\Binom{\operatorname{Binom}}
\def\pmod{\operatorname{mod}}
\def\ord{\operatorname{ord}}

\DeclareMathOperator*{\argmax}{arg\,max}

\newtheorem{lemma}{Lemma}
\newtheorem{proposition}{Proposition}
\newtheorem{theorem}{Theorem}
\newtheorem{corollary}{Corollary}

\newtheorem{remark}{Remark}

\allowdisplaybreaks

\title{On the longest increasing subsequence and number of cycles of butterfly permutations}
\author{John Peca-Medlin\thanks{Department of Mathematics, University of California, San Diego, \href{mailto:jpecamedlin@ucsd.edu}{jpecamedlin@ucsd.edu}} \and Chenyang Zhong\thanks{Department of Statistics, Columbia University, \href{mailto:cz2755@columbia.edu}{cz2755@columbia.edu}}}
\date{\today}

\begin{document}

\maketitle

\begin{abstract}
    One method to generate random permutations involves using Gaussian elimination with partial pivoting (GEPP) on a random matrix $A$ and storing the permutation matrix factor $P$ from the resulting GEPP factorization $PA=LU$. We are interested in exploring properties of random butterfly permutations, which are generated using GEPP on specific random butterfly matrices. Our paper highlights new connections among random matrix theory, numerical linear algebra, group actions of rooted trees, and random permutations. We address the questions of the longest increasing subsequence (LIS) and number of cycles for particular uniform butterfly permutations, with full distributional descriptions and limit theorems for simple butterfly permutations. We also establish scaling limit results and limit theorems for nonsimple butterfly permutations, which include certain $p$-Sylow subgroups of the symmetric group of $N=p^n$ elements for prime $p$. For the LIS, we establish power law bounds on the expected LIS of the form $N^{\alpha_p}$ and $N^{\beta_p}$ where $\frac12 < \alpha_p < \beta_p < 1$ for each $p$ with $\alpha_p = 1 - o_p(1)$, showing distinction from the typical $O(N^{1/2})$ expected LIS frequently encountered in the study of random permutations (e.g., uniform permutations). For the number of cycles scaled by $(2-1/p)^n$, we establish a full CLT to a new limiting distribution depending on $p$ with positive support we introduce that is uniquely determined by its positive moments that satisfy explicit recursive formulas; this thus determines a CLT for the number of cycles for any uniform $p$-Sylow subgroup of $S_{p^n}$.

\end{abstract}

\tableofcontents

\section{Introduction}

Gaussian elimination (GE) with partial pivoting (GEPP) remains one of the most used linear solvers for the linear system $A \V x = \V b$ for $A \in \mathbb C^{n\times n}$. GEPP iteratively transforms $A$ one column at a time using rank 1 updates into the product $PA = LU$ where $L$ is unit lower triangular, $U$ is upper triangular, and $P$ is a permutation matrix constructed so that each successive pivot entry is maximal in the leading column of the untriangularized lower-right block at the $k^{th}$ GEPP step.  For $S_n$ the group of permutations of $n$ objects, let $\sigma = \sigma(A) \in S_n$ denote the associated permutation generated using GEPP such that $P = P_{\sigma}$, where $P_\sigma\V e_k = \V e_{\sigma(k)}$ (realized as the left-regular action of $S_n$ on $[n]:=\{1,2,\ldots,n\}$). When $A$ is a random matrix, then $\sigma$ is a random permutation. \Cref{fig:permutons} shows diagrams of mapping pairs $(j,\sigma(j))/N \subset [0,1]^2$ of random permutations $\sigma \in S_N$ generated using GEPP with particular random matrix ensembles, including examples of butterfly permutations (see \Cref{sec:butterfly}).  

We are interested in studying probabilistic questions for the associated permutation generated when applying GEPP to certain random matrices, with a focus on random butterfly matrices. We will focus on the classical questions of the  longest increasing subsequence (LIS) and the number of cycles for the associated random permutation (see \Cref{sec: lis,sec: cycles}, respectively, for the relevant definitions for each question). Exploring these particular global statistics yields a concise view into the structural properties underlying butterfly permutations, giving first insight into the overall shape of such a permutation through the LIS as well as the inherent decompositional cyclic structure through the number of cycles.

\begin{figure}[t] 
    \centering
    \subfloat[$\Unif(S_N)$]{%
        \fbox{\includegraphics[width=0.28\textwidth]{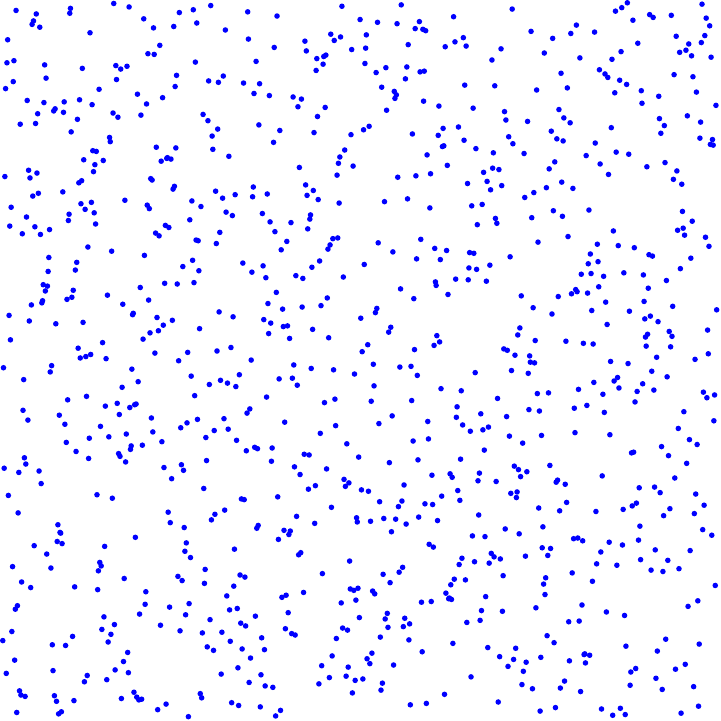}%
        \label{fig:unif permuton}%
        }
        }%
    \hspace{.1pc}
    \subfloat[$B_{s,N}$]{%
        \fbox{\includegraphics[width=0.28\textwidth]{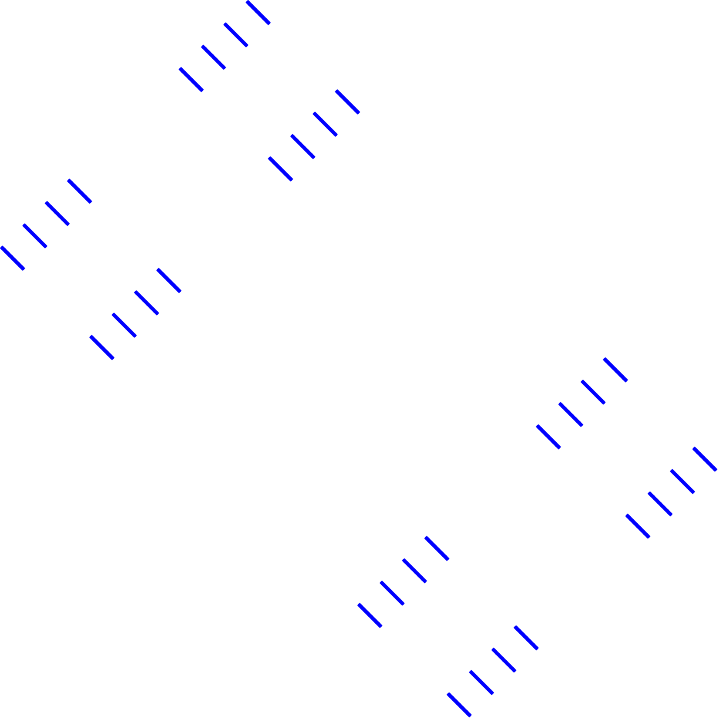}%
        \label{fig:s bin permuton}%
        }
        }%
    \hspace{.1pc}
    \subfloat[$B_N$]{%
        \fbox{\includegraphics[width=0.28\textwidth]{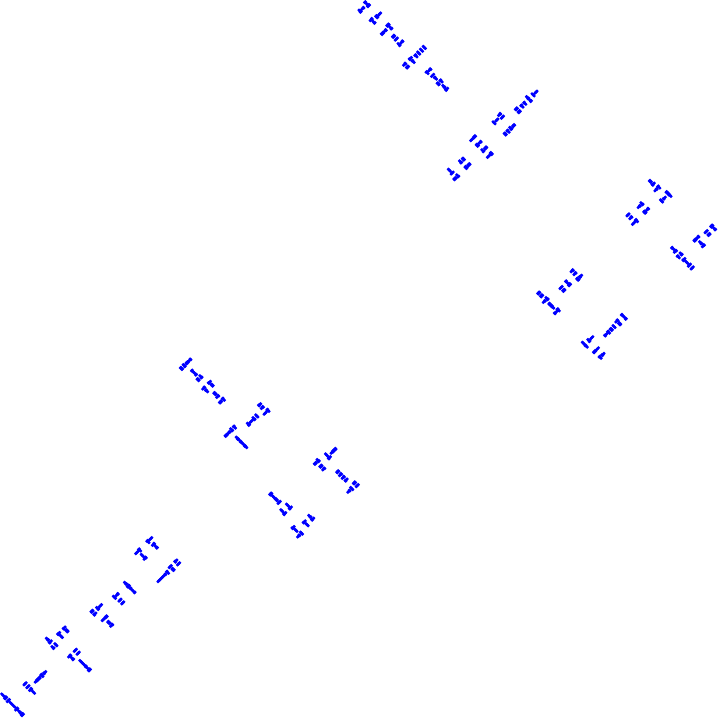}%
        \label{fig:ns bin permuton}%
        }
        }%
    \\
    \subfloat[$B_N^{(D)}$]{%
        \fbox{\includegraphics[width=0.28\textwidth]{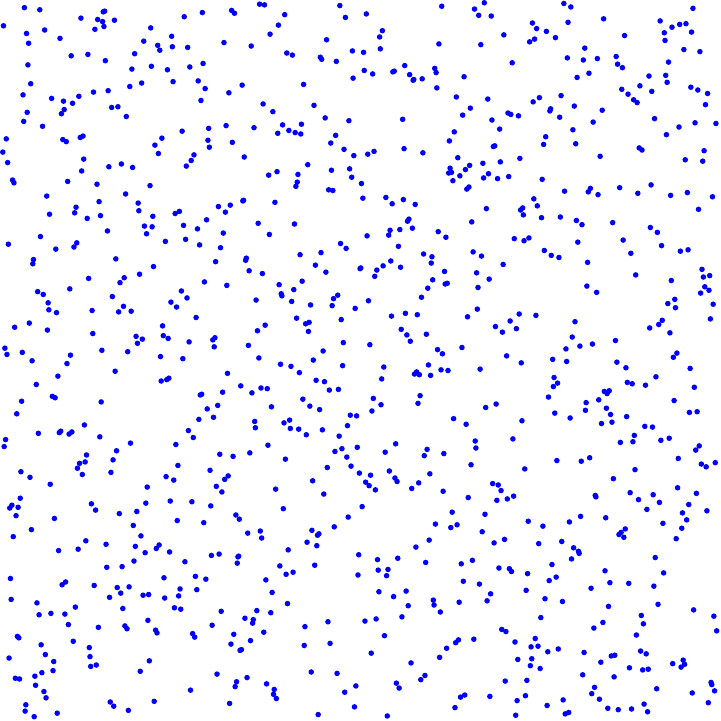}%
        \label{fig:diag permuton}%
        }
        }%
    \hspace{.1pc}
    \subfloat[$B_{s,n}^{(3)}$]{%
        \fbox{\includegraphics[width=0.28\textwidth]{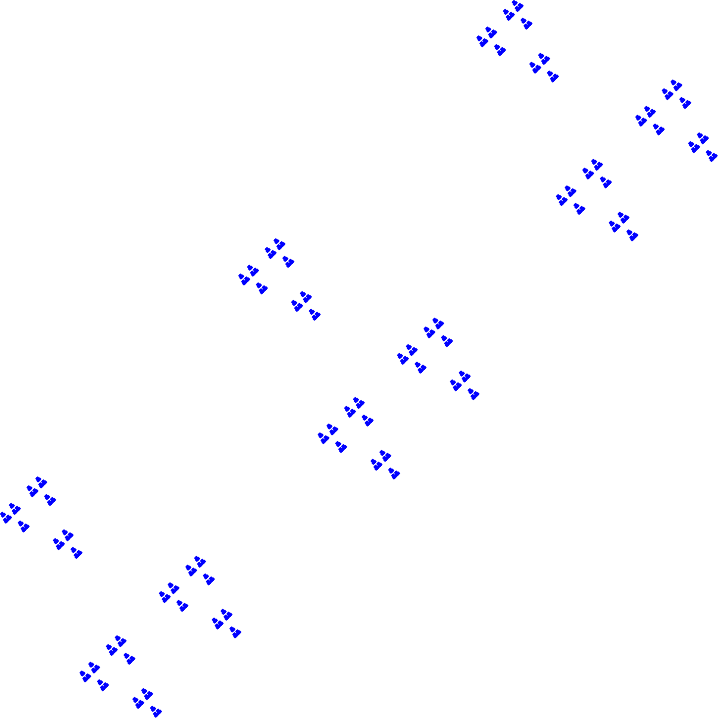}%
        \label{fig:s tern permuton}%
        }
        }%
    \hspace{.1pc}
    \subfloat[$B_{n}^{(3)}$]{%
        \fbox{\includegraphics[width=0.28\textwidth]{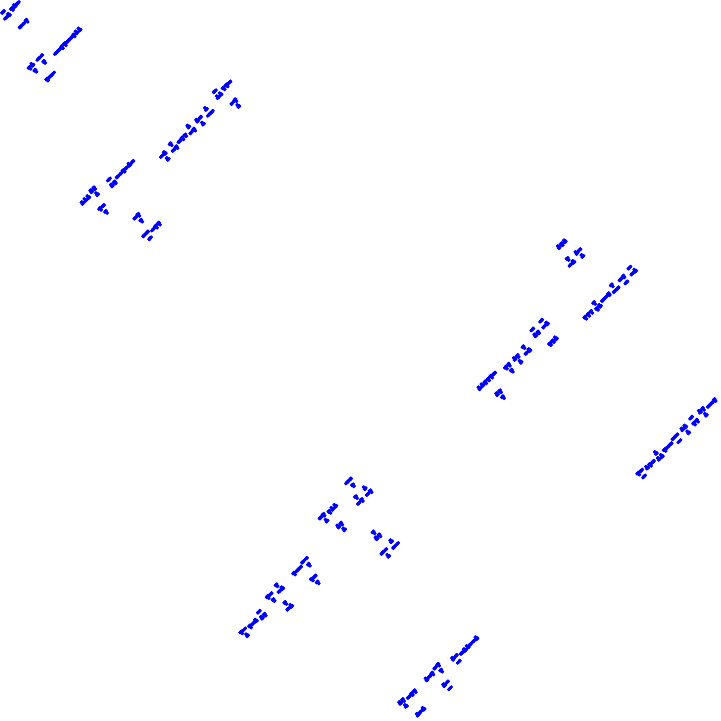}%
        \label{fig:ns tern permuton}%
        }
        }%
    \caption{Diagrams for random permutations of size $N = 2^{10} = 1024$ for (a) uniform, (b) simple binary butterfly, (c) nonsimple binary butterfly, and (d) diagonal butterfly permutations, and $N = 3^7 = 2187$ for (e) simple ternary butterfly and (f) nonsimple ternary butterfly permutations.}
    \label{fig:permutons}
\end{figure}

We now recall how applying GEPP to a matrix $A$ iteratively builds up an associated permutation $\sigma$. At step one, the transposition $(1 \ i_1)$ is used that swaps the corresponding rows of $A$ so that the first entry is maximal in magnitude among the remaining entries in the first column. After going forward with the GE elimination step, this process is then repeated in the untriangularized lower right $(n-1)\times(n-1)$ matrix, where now the transposition $(2 \ i_2)$ for $i_2 \ge 2$ is used to ensure the pivot is maximal in the leading column of the untriangularized system. This continues on, which yields the decomposition
\begin{equation}
\label{eq:GEPP perm}
    \sigma(A) := (n \ i_{n})\cdots(2 \ i_2)(1 \ i_1)
\end{equation}
where $i_k \in \{k,k+1\ldots,n\}$. (Necessarily $i_n = n$.) Unless otherwise specified, we will assume GEPP uses the default tie-breaking convention  that 
\begin{equation*}
i_k = \min\left(\argmax_{j \ge k} |A_{jk}^{(k)}|\right),
\end{equation*}
where $A^{(k)}$ denotes intermediate form of $A$ at GEPP step $k$, with zeros below the first $k - 1$ diagonals, where $A^{(1)} = P_{(1 \ i_1)}A$ and $A^{(n)} = U$. In terms of GEPP, the associated transposition used at the $k^{th}$ GE step then also determines whenever a pivot movement was \textit{not} needed, viz., when $i_k = k$. Note if $i_k > k$, then the $\sigma(A)$-orbit of $k$ contains a larger element. It follows these fixed pivot movement locations at $i_k = k$ correspond then directly to the maximal elements for the associated orbits of $\sigma(A)$, i.e., each cycle in the disjoint cycle decomposition of $\sigma(A)$. Hence, if $\Pi(A)$ returns the number of GEPP pivot movements needed on $A \in \mathbb C^{n\times n}$ and if $C(\sigma)$ denotes the number of cycles in the disjoint cycle decomposition of $\sigma$, then 
\begin{equation*}
    \Pi(A) = n - C(\sigma(A)).
\end{equation*}
Determining the distribution of the number of pivot movements needed when using GEPP on certain random matrices was the focus of \cite{P24}. Our goal in this paper is to study a particular set of random permutations that are induced through applying GEPP to particular classes of random butterfly matrices (see \Cref{sec:butterfly} for relevant definitions).

By the above description, then the lower triangular GEPP factor $L$ is iteratively formed using GE elimination steps, where each entry below the diagonal is zeroed out by subtracting a $L_{ij} = {A_{ij}^{(j)}}/{A_{jj}^{(j)}}$ multiple of row $j$ from row $i > j$ of $A^{(j)}$. Since then $|A_{jj}^{(j)}| \ge |A_{ij}^{(j)}|$ for all $i \ge j$, then $|L_{ij}| \le 1$ for all $i,j$. If no ties are encountered during any GEPP pivot search, then $|L_{ij}| < 1$ for all $i > j$. 

If $PA = LU$ is the GEPP factorization of $A$, then $\sigma(A) = \sigma(P^T)$ since $PP^T = \V I$ is the GEPP factorization of $P^T$. (Note no tie strategy is needed using GEPP on a permutation matrix.) Hence, for $\pi \in S_n$, then $\sigma(P_\pi^T) = \pi$. This enables the study of random permutations generated using GEPP to then expand and enhance the general study of random permutations.



Here we review some other classes of random permutations that have been analyzed in the literature. The standard starting place for the study of random permutations considers uniform permutations $\sigma \sim \Unif(S_n)$. We will discuss particular such properties of uniform permutations in \Cref{sec: lis,sec: cycles}. Recent literature on non-uniform random permutations has touched on, for example, locally uniform permutations \cite{sjostrand2023monotone}, the Mallows permutation model \cite{mallows1957non}, and pattern-avoiding permutations \cite{Borga21}. Locally uniform permutations are random permutations obtained by drawing $n$ independent points on the unit square $[0,1]^2$ from a continuous density ($i\in [n]$ is mapped to $j\in [n]$ if the point with the $i$th smallest $x$-coordinate has the $j$th smallest $y$-coordinate). The Mallows permutation model was introduced by Mallows \cite{mallows1957non} as a statistical model for rankings. Under the Mallows permutation model, the probability of picking a permutation $\sigma\in S_n$ takes the form 
\begin{equation*}
    \mathbb{P}(\sigma)\propto e^{-\beta d(\sigma,\sigma_0)},
\end{equation*}
where $\beta\in\mathbb{R}$ is a scale parameter, $\sigma_0\in S_n$ is a location parameter, and $d(\cdot,\cdot)$ is a distance metric on permutations (see \cite[Chapter 6]{diaconis1988group} for a host of such distance metrics). {In the literature, $\sigma_0$ is often taken to be the identity permutation, Mallows permutation model with Kendall's $\tau$ is often parametrized by $q:=e^{-\beta}$ (see e.g. \cite{bhatnagar2015lengths, Mueller_Starr_2013}), and Mallows permutation model with Cayley distance (also known as the Ewens measure) is often parametrized by $\theta:=e^{-\beta}$ (see e.g. \cite{Crane16, diaconis2022statistical}).} Pattern-avoiding permutations refer to the uniform distribution on permutations that avoid one or multiple patterns (see, for example, \cite{Borga21} for details). The properties of these non-uniform random permutations, including LIS and cycle structure, will also be reviewed in \Cref{sec: lis,sec: cycles}.

\subsection{Overview of results}

This paper is organized to first introduce butterfly permutations (see \Cref{sec:butterfly}), with which we will then present new results in regarding the distributions of their LIS (see \Cref{sec: lis}) and number of cycles (see \Cref{sec: cycles}). Our proof techniques will highly utilize the intrinsic structural properties of these permutations to be able to answer what are typically hard analytical questions using a simpler set of combinatorial tools. For instance, we reduce the LIS and number of cycle questions for these random permutations to studying particular recursive random variables (cf. \Cref{prop: lis 2-Xn,prop: ns cycle dist})

In \Cref{sec:butterfly}, we define butterfly permutations of length $N = p^n$ as being the associated permutations induced through applying GEPP to particular classes of random butterfly matrices. This builds off of previous results in \cite{P24}, which gave a full classification of the simple (binary) butterfly permutations, $B_{s,N}$ (with $p = 2$), induced using GEPP on simple scalar butterfly matrices; in particular, this induces the Haar measure on $B_{s,N}$ when the input butterfly matrix is constructed using uniform input angles. A notable contribution in \Cref{sec: ns butterfly} is in regard to the nonsimple butterfly permutations, $B_N$, which can be described explicitly via a full GEPP factorization of particular block matrices (see \Cref{p:ns factor}). Of note, since $|B_N| = 2^{2^n - 1}$ and $B_N$ is a subgroup of $S_{2^n}$, then necessarily $B_N$ comprises a 2-Sylow subgroup of $S_{2^n}$. (Recall if $G$ is a finite group, and $p$ is a prime such that $p^k \mid |G|$ and $p^{k+1} \nmid |G|$ for some $k$, then there exists a subgroup $H$ of order $p^k$ that is called a $p$-Sylow subgroup of $G$; all $p$-Sylow subgroups are conjugate to one another.) This leads to a new analogous result for $B_N$ that GEPP yields a Haar nonsimple butterfly permutation when the input butterfly matrix is formed using uniform angles (see \Cref{thm: unif ns}). We also establish more general $m$-nary butterfly permutations of order $N = m^n$, $B_{s,n}^{(m)}$ and $B_n^{(m)}$, that are constructed analogously to $B_{s,2^n}$ and $B_{2^n}$; these include particular $p$-Sylow subgroups of $S_{p^n}$ for each prime $p$.

In \Cref{sec: lis}, we address the question of the LIS for butterfly permutations. In \Cref{sec: lis s}, we first address the LIS question for simple butterfly permutations. Unlike in the classical study of the LIS for uniform permutations, that took over 3 decades to more or less fully resolve (see \cite{Romik_2015} for the full story), we are able to utilize the explicit structural properties of the simple butterfly permutations to give full distributional descriptions along with the accompanying Central Limit Theorem (CLT) for the associated LIS when $\sigma_n \sim \Unif(B_{s,n}^{(m)})$ (see \Cref{t: s lis}). Of note, our methods then enable very strong results using straightforward analytic and combinatorial tools, as these permutations can be studied by underlying permutations on blocks. For example, for $p = 2$, then for $L(\sigma)$ the LIS for the permutation $\sigma$, we have $\log_2 L(\sigma_n) \sim \Binom(n,\frac12)$ when $\sigma_n \sim \Unif(B_{s,N})$, which yields then limiting log-normal behavior when properly then scaling $L(\sigma_n)$. This is a clear distinction from the Tracy-Widom behavior shown to govern asymptotics of  uniform permutations \cite{BDJ99}. We can similarly fully classify the analogous question of the longest decreasing subsequence $D(\sigma_n)$ for $\sigma_n \sim \Unif(B_{s,n}^{(m)})$, which again have binomial and hence log-normal limiting behavior (see \Cref{prop: LDS bin}). For $p = 2$, the LIS and longest decreasing subsequence satisfy the additional constraint $L(\sigma_n)D(\sigma_n) = 2^n$.

In \Cref{sec: ns lis}, we address the LIS question for nonsimple butterfly permutations. This direction is less clear than the simple case. For instance, instead of providing full distributional descriptions, we focus primarily on establishing power-law bounds on $\E L(\sigma_n)$ for $\sigma_n \sim \Unif(B_n^{(m)})$. For each $m$, we establish constants  $\frac12 < \alpha_m < \beta_m < 1$ such that $N^{\alpha_m} \le \E L(\sigma_n) \le N^{\beta_m}$ (see \Cref{thm: power law}); in particular, this lower bound matches the exact first moment for the associated simple butterfly permutation, i.e., $N^{\alpha_m} = \E L(\sigma_n)$ when $\sigma_n \sim \Unif(B_{s,n}^{(m)})$. Moreover, we establish $\alpha_m = 1 - o_m(1)$, and also the difference $\beta_m - \alpha_m = o_m(1)$ so that each bound can serve as a sufficient estimator for the overall proper power-law exponent of $\E L(\sigma_n)$. For example, for $m = 2$, we have $\alpha_2 \approx 0.585$ and $\beta_2^* \approx 0.833$; we then use linear regression and exact computations of $\E L(\sigma_n)$ for $\sigma_n \sim \Unif(B_N)$ for fixed $n$, using an explicit recursion formula, to estimate the actual power-law exponent $\E L(\sigma_n) \approx N^{\hat \alpha}$, which yields an estimate $\hat \alpha \approx 0.682 \in (\alpha_2,\beta_2)$. {We note that uniform permutations as well as most of the non-uniform random permutation models have expected LIS of order $O(N^{1\slash 2})$ (see e.g. \cite{Borodin_1999, chatterjee2024vershik, deuschel1995limiting, kammoun2, Mueller_Starr_2013}). Two exceptions are Mallows permutation model (with Kendall's $\tau$) with specific scaling regimes of the parameter $q$ (specifically, $n(1-q)\rightarrow\infty$; see \cite{bhatnagar2015lengths}) and permutations sampled from Brownian separable permutons (\cite{BDG24}). We note that the butterfly permutations considered in this paper are quite different from these two models: under the scaling $n(1-q)\rightarrow\infty$, the permuton limit of Mallows permutation model is concentrated on the diagonal of the unit square and is deterministic; in contrast, we expect the butterfly permutations to have a fractal and random permuton limit (as indicated by Figure \ref{fig:permutons}). Moreover, the permutations considered in \cite{BDG24} are directly sampled from the permuton limit (unlike random separable permutations \cite{bassino2018brownian}, which are pre-limiting objects and defined at the discrete level); in comparison, the butterfly permutations are defined at the discrete level. The permuton limit of the permutations in \cite{BDG24} also appear to be quite different from the permuton limit of the butterfly permutations (for example, compare Figure \ref{fig:permutons} and \cite[Fig. 1]{BDG24}).}


In \Cref{sec: cycles}, we change our focus now to the question of the number of cycles for butterfly permutations, $C(\sigma_n)$. First, for simple butterfly permutations, a Law of Large Numbers (LLN) type result is straightforward: for prime $p$, then $\frac{p}N C(\sigma_n)$ converges to 1 in probability (see \Cref{thm: s cycles LLN}). For the nonsimple butterfly permutations, we provide a full distributional description along with a CLT  result when $m = p$ is prime: we introduce new distributions $W^{(p)}$ uniquely determined by their moments that satisfy an explicit recursion formula, and we then show $C(\sigma_n)/(2-\frac1p)^n$ with $\sigma_n \sim \Unif(B_n^{(p)})$ converges in distribution to $W^{(p)}$ using the moment method (see \Cref{thm: cycles p}). Since the number of cycles of a permutation is invariant under conjugation, and every $p$-Sylow subgroup is conjugate with one another, we thus then provide a full CLT for the number of cycles for any uniform $p$-Sylow subgroup of $S_{p^n}$. We also run numerical experiments to support the descriptions for $p = 2$, and include some select results relating to the number of fixed points for butterfly permutations.

\subsection{Notation and preliminaries}

For convenience, we will reserve $N$ for when $N = m^n$ for a base $m$, with a large focus of the following text using $m = 2$ or $m$ prime; we will then often reserve $p$ for when $p$ is a prime. Let $\V e_i \in \mathbb R^n$ denote the standard basis vector. For $A \in \mathbb R^{n \times m}$, let $A_{ij} = \V e_i^T A \V e_j$ denote the entry of $A$ in row $i$ and column $j$, and write $\V I = \sum_{i=1}^n \V e_i \V e_i^T$ and $\V 0$ for the corresponding identity and zero matrices, where the order is given explicitly if not implicitly obvious. For $A \in \mathbb R^{n_1 \times m_1}$, $B \in \mathbb R^{n_2 \times m_2}$, let $A \oplus B \in \mathbb R^{(n_1 + n_2) \times (m_1 + m_2)}$ denote the block diagonal matrix 
\begin{equation*}
     A\oplus B = \begin{bmatrix}
    A & \V 0 \\ \V 0 & B
\end{bmatrix}
\end{equation*}
and $A \otimes B \in \mathbb R^{n_1n_2 \times m_1m_2}$ denote the Kronecker product
\begin{equation*}
    A \otimes B = \begin{bmatrix}
        A_{11} B & \cdots & A_{1,m_1} B \\
        \vdots & \ddots & \vdots \\
        A_{n_1,1} & \cdots & A_{n_1,m_1} B
    \end{bmatrix}.
\end{equation*}
Both $\oplus$ and $\otimes$ commute with complex transposition and inverses (e.g., $(A \otimes B)^T = A^T \otimes B^T$), and preserve certain matrix structures, such as unitary, lower triangular, and permutation matrices. The Kronecker product also satisfies the mixed-product property, where 
\begin{equation*}
    (A \otimes B)(C \otimes D) = (AC) \otimes (BD)
\end{equation*}
when the respective matrix dimensions are compatible. This can be rephrased (somewhat) succinctly as the product of Kronecker products is the Kronecker product of products.

Let $S_n$ denote the symmetric group of permutations of $[n] = \{1,2,\ldots,n\}$. We will employ the standard notation $\sigma = (\sigma(1) \ \sigma(2) \ \cdots \ \sigma(n))$ (displaying in index $k$ the value $\sigma(k)$), which will be used in the setting of studying the LIS. Otherwise, we will primarily utilize the cycle notation for permutations. 
Let $\mathcal P_n = \{P_\sigma: \sigma \in S_n\}$ denote the corresponding permutation matrices formed as the left-regular representation of the action of $S_n$ on the standard basis vectors, i.e., $P_\sigma \V e_i = \V e_{\sigma(i)}$. 
For $\sigma_1 \in S_n, \sigma_2 \in S_m$, let $\sigma_1 \oplus \sigma_2 \in S_{n + m}$ and $\sigma_1 \otimes \sigma_2 \in S_{nm}$ be determined by the associated permutation matrices $P_{\sigma_1\oplus \sigma_2} = P_{\sigma_1} \oplus P_{\sigma_2}$ and $P_{\sigma_1 \otimes \sigma_2} = P_{\sigma_1} \otimes P_{\sigma_2}$. Let $1_n \in S_n$ denote the identity permutation.

For (real) random variables, $X,Y$, let $X \sim Y$ denote that $X$ and $Y$ are equal in distribution, where the cumulative distribution functions (cdf) for $X$ and $Y$ both align for any $t \in \mathbb R$. 
Additionally, we will utilize left- and right-invariance properties of the Haar measure on locally compact Hausdorff topological groups, first established by Weil \cite{We40}. For a compact group $G$, this measure can be normalized to yield a probability measure $\Haar(G)$ that corresponds to $\Unif(G)$, which inherits the invariance and regularity properties of the original measure.

    
    


    

\section{Random permutations using GEPP}
\label{sec:butterfly}

The study of random permutations started with uniform permutations. In 1938, Fisher and Yates provided an algorithm to produce a uniform permutation by iteratively sampling a uniform number of the remaining unshuffled index to flip starting from a fixed initial permutation. Starting with the identity, then this aligns to produce a permutation in the form \eqref{eq:GEPP perm} to yield:
\begin{theorem}[\cite{FY38}]
    If $i_k \sim \Unif(\{k,k+1,\ldots,n\})$ are independent for $k = 1,\ldots,n$, then $(n \ i_n) \cdots$ $(2 \ i_2)(1 \ i_1) \sim \Unif(S_n)$.
\end{theorem}
A proof for this can be succinctly realized as an iterative application of the Subgroup algorithm of Diaconis and Shashahani in \cite{DiSh87}:
\begin{theorem}[\cite{DiSh87}]\label{thm: subgroup alg}
    Let $G$ be a compact group, $H$ a closed subgroup of $G$, and $G/H$ the set of left-cosets of $H$ in $G$. If $x \sim \Unif(G/H)$ is independent of $y \sim \Unif(H)$,  then $xy \sim \Unif(G)$.
\end{theorem}
\noindent Hence, if we view $S_{n-1} \subset S_n$ by identifying $S_{n-1}$ with the group of permutations of $[n]$ that fix 1, then for $i_1 \sim \Unif([n])$ and $\pi \sim \Unif(S_{n-1})$ we have $(1 \ i_1) \sim \Unif(S_n/ S_{n-1})$ is independent of $\pi$, and so $\pi (1 \ i_1) \sim \Unif(S_n)$. \Cref{fig:unif permuton} shows the diagram of pairs $(j,\sigma(j))/N \subset [0,1]^2$ for $\sigma \sim \Unif(S_n)$; as $n$ grows, this diagram will progressively fill in the entire unit box uniformly (e.g., see \cite{Borga21}).

In \cite{P24}, the author established sufficient (but not necessary) conditions on $A$ such that $\sigma(A) \sim \Unif(S_n)$.

\begin{theorem}[\cite{P24}]
\label{thm:unif_perm}
    If the first $n-1$ columns of $A \in \mathbb C^{n\times n}$ are independent and each column has independent and identically distributed (iid) entries from an absolutely continuous distribution, then $\sigma(A) \sim \Unif(S_n)$.
\end{theorem}

\noindent Additional connections of GEPP and QR factorizations (along with the Subgroup algorithm), yield other random matrix transformations that also yield uniform GEPP permutations.

\begin{corollary}[\cite{P24}]
\label{cor:unif_perm}
    Let $U,V \sim Haar(G_n)$ be iid for $G_n$ one of the (special) orthogonal or unitary groups of order $n$. If $A = U $ or $A = UBV^*$ where $B$ is deterministic (and real when $U,V$ are real), then $\sigma(A) \sim \Unif(S_n)$. 
\end{corollary}

\noindent Of course, this is not comprehensive for all random matrix models that yield uniform GEPP permutation factors: if $\pi \sim \Unif(S_n)$, then $\sigma(P_\pi^T) = \pi$.

For some other random matrices that do not meet the hypotheses of \Cref{thm:unif_perm} or \Cref{cor:unif_perm}, then the resulting associated GEPP permutation is not necessarily uniform. This can be seen for simple examples with $A \in \mathbb C^{2 \times 2}$, where $\sigma(A) = (1 \ i_1)$ and $i_1 - 1$ is necessarily a Bernoulli random variable $\Bern(q)$. Using $A \sim \GOE(2)$, $\GUE(2)$, or $\Bern(p)^{2\times 2}$ (an iid matrix) (cf. \cite[Example 3]{P24}), then $i_1 - 1 \sim \Bern(q)$ for $q$ taking the respective values $\frac2\pi \arctan(\frac1{\sqrt 2}) \approx .39183$, $\frac1{\sqrt 3} \approx 57735$, and $p(1-p) \le \frac14$.\footnote{Using the built-in function \texttt{lu} from MATLAB, however, produces different behavior with complex matrices, as the pivot search uses the $\ell^1$ norm on $z = a+bi \in \mathbb C$, i.e., $|a|+|b|$, rather than the complex modulus ($\ell^2$-)norm, $\sqrt{a^2 + b^2}$. Using this approach, for $A \sim \GUE(2)$, then a pivot movement occurs with probability approximately $\frac23$ (which is equal to $\P(|Z_1|+|Z_2|>\sqrt 2 |Z_3|)$ for $Z_i$ iid $N(0,1)$) using the $\ell^1$-norm instead of $\frac1{\sqrt 3}\approx 0.57735$ with the standard complex modulus norm.} This differs from the uniform permutation behavior guaranteed, say, if $A_{ij} \sim N(0,1)$ iid, where then $i_1 - 1 \sim \Bern(\frac12)$.

A natural question then for these models is what can be determined for the associated GEPP permutation matrix. For fixed $n$ for each model, the induced random permutation is never truly uniform. For example, for $G \sim \GOE(n)$, this can be realized since the leading pivot has higher chance of not needing to move since $G_{11} \sim N(0,2)$ while $G_{k1} \sim N(0,1)$ for $k > 1$, i.e., $\P(i_1 = 1) > \frac1n$. Future work will explore the asymptotic behavior for these random ensembles, which appear to align with uniform permutations. 

Another note could explore further the relationship between the \textit{computed} GEPP permutation matrix factor, using floating-point arithmetic, rather than the \textit{exact} GEPP permutation matrix factor. This was addressed in the case of iid Gaussian matrices by Huang and Tikhomirov in \cite{HT23}. To study why GEPP performs so well in practice despite having potentially exponential growth, they carried out an improved average-case analysis of GEPP using iid Gaussian matrices to better study the behavior of the growth factors and condition numbers. In addition to establishing polynomial growth with high probability (far from the worst-case exponential growth), they also establish that the computed and exact GEPP permutation factors align with very high probability.

Random butterfly permutations are of the form $\sigma(B)$ where $B$ is a random butterfly matrix. The next section includes a short overview and definition of random butterfly matrices, which will then lead into full descriptions of the associated butterfly permutations (under exact arithmetic; future work can explore how floating-point arithmetic impacts the permutation matrix factor for butterfly matrices, although empirically each such factor encountered matches the expected structure of the exact permutation factor). 

\subsection{Random butterfly matrices}
Random butterfly matrices were introduced by Parker to remove the \textit{need} to pivoting altogether when using GEPP \cite{Pa95}. Pivoting with GE can become a bottleneck that significantly slows down computation and often prevents the use of parallelization and other acceleration strategies (cf. \cite{baboulin,Pa95}). Parker showed independent butterfly matrices $U,V$ of order $N = 2^n$ can be used to transform the linear system $A \V x = \V b$ into the equivalent system $UAV^*\V y = U\V b$ and $\V y = V \V x$, which can  then be solved (with high probability in floating-point arithmetic, and almost surely in exact arithmetic) using GE with no pivoting (GENP); the recursive structure of butterfly matrices yield fast matrix-vector multiplication of order $O(N \log N)$ FLOPS, so that this prepocessing $UAV^*$ can be carried out without impacting the $O(N^3)$ leading order complexity of GE. An order $N=2^n$ butterfly matrix, $B$, which was studied extensively in \cite{P24_gecp,P24,PT23,Tr19}, is constructed recursively as
\begin{equation}\label{def: butterfly}
    B = \begin{bmatrix}
        C & S \\ -S & C
    \end{bmatrix} \begin{bmatrix} A_1 & \V 0 \\ \V 0 & A_2 \end{bmatrix} = \begin{bmatrix}
        C A_1 & S A_2 \\ -S A_1 & C A_2
    \end{bmatrix},
\end{equation}
where $A_1,A_2$ are order $N/2$ butterfly matrices and $C,S$ are order $N/2$ matrices that satisfy $CS = SC$ and the Pythagorean matrix equation $C^2 + S^2 = \V I$.  We will call butterfly matrices \textit{simple} when $A_1 = A_2$ holds at each recursive step. When $C,S$ are diagonal matrices, then their corresponding entries necessarily come in $(\cos\theta,\sin\theta)$ pairs for some $\theta$. By construction, then $B \in \SO(N)$. We will focus especially on when $C,S$  are also scalar matrices, so that $(C,S) = (\cos\theta,\sin\theta) \V I$ at each recursive step, and denote $\B(N)$ for the scalar butterfly matrices and $\B_s(N)$ for the simple scalar butterfly matrices. Under this lens, \eqref{def: butterfly} for scalar butterfly matrices $B \in \B(N)$ can be rewritten as 
\begin{equation*}
    B = (R_\theta \otimes \V I_{N/2})(A_1 \oplus A_2), \qquad \mbox{where} \qquad R_ \theta = \begin{bmatrix}
\cos\theta & \sin\theta \\ -\sin\theta & \cos\theta
\end{bmatrix},
\end{equation*}
is the (clockwise) rotation matrix of angle $\theta$. When $B \in \B_s(N)$ then $B$ further takes the form
\begin{equation*}
B = \bigotimes^{n}_{j=1} R_{\theta_j}.
\end{equation*}
In particular, we have for any $j = 0,\ldots,n$, then
\begin{align*}
    \B_s(N) &= \bigotimes_{k=1}^n \SO(2) = \B_s(2^{n-j})\otimes \B_s(2^{j})\\
    \B(N) &= (\B(2^{n-j}) \otimes \V I_{2^j}) \bigoplus_{k=1}^{2^{n-j}} \B(2^j).
\end{align*}
Let $\B^{(D)}_s(N)$ and $\B^{(D)}(N)$ denote the simple and nonsimple diagonal butterfly matrices, whose left-factor from \eqref{def: butterfly} has $(C,S) = \bigoplus_{j=1}^{N/2}(\cos\theta_j,\sin\theta_j)$ for each recursive step. Note then 
\begin{equation}\label{def: diag factor}
    \begin{bmatrix}
        C & S\\ -S & C
    \end{bmatrix} = Q_n \left(\bigoplus_{j=1}^{N/2} R_{\theta_j} \right)Q_n^T
\end{equation}
where $Q_n$ is the perfect shuffle matrix such that $Q_n(A \otimes B)Q_n^T = B \otimes A$ for any $A,B \in \mathbb C^{N/2 \times N/2}$. For instance, $Q_1 = \V I_2$ while $Q_2 = P_{(2 \ 3)} \in \mathcal P_4$. It follows then the scalar butterfly matrices are formed using $n$ input angles in the simple case and and $N-1$ input angles in the nonsimple case, while diagonal butterfly matrices are formed using $N-1$ and $nN/2$ respective input angles for the simple and nonsimple cases.

Random butterfly matrices $\B(N,\Sigma)$ are formed using the above recursion where each $(C,S)$ pair is a random matrix sampled from $\Sigma$. Let 
\begin{align*}
    \Sigma_S &= \{(\cos\theta,\sin\theta)\V I_{2^k}: \theta \sim \Unif([0,2\pi)) \ \mbox{iid}\}\\
    \Sigma_D &= \{\bigoplus_{j=1}^{2^{k-1}}(\cos\theta_j,\sin\theta_j): \theta_j \sim \Unif([0,2\pi)) \ \mbox{iid}\}
\end{align*}
denote, respectively, when $C,S$ are formed using independent uniform angle scalar and diagonal matrices. In particular, then 
\begin{equation*}
    \B_s(N,\Sigma_S) \sim \Haar\left(\bigotimes_{j=1}^n \SO(2)\right).
\end{equation*}

General $N = m^n$ order butterfly matrices can similarly be formed using the block structure 
\begin{equation*}
    B = (A \otimes \V I_{N/m})\bigoplus_{j=1}^m A_j.
\end{equation*}
for $A,A_j \in \mathbb C^{N/m \times N/m}$. Moreover, if $A,A_j$ are permutation matrices themselves, then so is $B$. See \cite{phd} for an overview of numerical and statistical properties of general butterfly matrices.

\subsection{Simple  butterfly permutations}\label{sec: simple butterfly}

When $B \sim \B_s(N,\Sigma_S)$, then $\sigma(B)$ is a random simple (scalar) butterfly permutation. We will briefly outline then how one recovers the permutation factors for rotation matrices, and hence for simple scalar butterfly matrices. First recall the GENP factorization of a $2\times 2$ rotation matrix $R_\theta = L_\theta U_\theta$ where
\begin{equation*}
R_\theta = \begin{bmatrix}
\cos\theta & \sin \theta \\ -\sin\theta & \cos\theta
\end{bmatrix}, \qquad L_\theta = \begin{bmatrix}
1 & 0 \\ -\tan \theta & 1
\end{bmatrix}, \qquad 
U_\theta = \begin{bmatrix}
\cos\theta & \sin \theta \\0 & \sec \theta
\end{bmatrix}.
\end{equation*}
Furthermore, for $D = (-1) \oplus 1$ then $DR_\theta = R_{-\theta}D$, while $P_{(1 \ 2)} = R_{\pi/2}D$ so that $P_{(1 \ 2)} R_\theta = R_{\pi/2}DR_\theta = R_{\pi/2 - \theta}D$. Now note a GEPP pivot movement is needed on $R_\theta$ only when $|\cos\theta| < |\sin \theta|$ or $|\tan \theta| > 1$. Define the pivot movement indicator
\begin{equation*}
e(\theta) =  \mathds 1(|\tan\theta| > 1) = \left\{\begin{array}{ll} 1 & \mbox{if $|\tan \theta| > 1$,} \\ 0 & \mbox{if $|\tan \theta| \le 1$,}
\end{array}\right.
\end{equation*}
and let $P_\theta = P_{(1 \ 2)}^{e(\theta)} = P_{e(\theta)\frac\pi2}$ and $D_\theta = D^{e(\theta)}$. It then follows 
\begin{equation}\label{eq: 2x2 gepp}
P_\theta R_\theta = R_{\hat \theta}D_\theta = L_{\hat \theta} (U_{\hat \theta} D_\theta)
\end{equation}
is the GEPP factorization of $R_\theta$, with then $\hat \theta = \theta$ if $e(\theta) = 0$ (i.e., when no pivot movement is needed on $R_\theta$) and $\hat \theta = \frac\pi2 - \theta$ when $e(\theta) = 1$. This can then be used to establish the GEPP factorization of scalar butterfly matrices. Utilizing  properties of the Kronecker product (viz., its closure under transposition and inverses along with the mixed-product property), then the simple scalar butterfly permutations have the GEPP factorization $PB = LU$ where each factor is a Kronecker product of the associated Kronecker factors, with in particular $P = \bigotimes_j P_{\theta_j}$ when $B = \bigotimes_j R_{\theta_j}$. We can then define the simple scalar butterfly permutations  by
\begin{equation*}
     B_{s,N} = \bigotimes_{j=1}^n  S_2 =  B_{s,2^{n-j}} \otimes  B_{s,2^j} \subset S_N,
\end{equation*}
where $B_{s,2} = S_2$. Since $S_2 = \langle {(1 \ 2)}\rangle \cong C_2$, where $C_k$ denotes the cyclic group of order $k$ (using the convention $\langle x \rangle = \{x^j: j = 0,1,2,\ldots\}$ to denote the subgroup of a group $G$ generated by $x \in G$), then $B_{s,N} \cong C_2^n$ by the mixed-product property. \Cref{fig:s bin permuton} shows the fractal diagram for a uniformly sampled simple binary butterfly permutation.

As seen in \cite{P24}, since $\theta_j$ are iid for $B \sim \B(N,\Sigma_S)$, then the generated GEPP permutation is uniform on $B_{s,N}$:
\begin{theorem}[\cite{P24}]\label{t: simple b}
    Let $B \sim \B_s(N,\Sigma_S)$. Then $\sigma(B) \sim \Unif(B_{s,N})$.
\end{theorem}
The  Kronecker structure of $B_{s,N}$ can be further exploited to determine the explicit action of $\sigma(B)$ for each $k \in [N]$. Let $\sigma \in B_{s,N}$ and let $\sigma_j \in   S_2$ such that $\sigma = \bigotimes_{j=1}^n {\sigma_{j}}$. We can write 
\begin{equation*}
    \V e_k = \bigotimes_{j=1}^n \V e_{a_{j}+1}, 
\end{equation*}
where $a_{j} \in \{0,1\}$ are the coefficients of the binary expansion of $k-1 = (a_1a_2\cdots a_n)_2$, so
    \begin{equation*}
        k = 1 + \sum_{j=1}^{n} a_j2^{n-j}.
    \end{equation*}
For example, $\V e_5 = \V e_2 \otimes \V e_1 \otimes \V e_1$ where $5-1=4 = (100)_2$. By the mixed-product property,
    \begin{equation*}
        \V e_{\sigma(k)} = P_\sigma\V e_k = \bigotimes_{j = 1}^n P_{\sigma_{j}} \V e_{a_{j}+1} = \bigotimes_{j = 1}^n \V e_{\sigma_{j}(a_{j}+1)}
    \end{equation*}
so that
    \begin{equation}
    \label{eq: bperm rule}
        \sigma(k) = 1 + \sum_{j=1}^{n}(\sigma_{j}(a_j+1)-1)2^{n-j} = 1 + (b_1b_2\cdots b_n)_2
    \end{equation}
where $b_j = \sigma_j(a_j+1)-1$. Hence, the action of $\sigma$ can be fully realized as the action on the binary coefficients of the input. 

\subsection{Nonsimple  butterfly permutations}\label{sec: ns butterfly}

The random nonsimple (scalar) butterfly permutations are of the form $\sigma(B)$ where $B \sim \B(N,\Sigma_S)$. We will now establish that these random permutations align necessarily with uniform permutations from the group of nonsimple butterfly permutations $B_{N} \subset S_{N}$, which is defined recursively as a semidirect product group
\begin{equation}\label{eq: Bn form}
 B_{N} = ( B_{N/2} \oplus  B_{N/2}) \rtimes_\varphi \langle (1 \ 2) \otimes 1_{N/2}\rangle
\end{equation}
where $\varphi(\sigma_1 \oplus \sigma_2) = \sigma_2 \oplus \sigma_1$ for $\sigma_i \in B_{N/2}$, with $B_2 = \langle {(1 \ 2)}\rangle$. This can be rephrased as $ B_{N}$ is an iterated $n$-fold wreath product of $\langle (1 \ 2)\rangle \cong C_2$ (for $C_m$ denoting the cyclic group of order $m$), i.e., $B_N \cong B_{N/2} \wr C_2 \cong C_2 \wr \cdots \wr C_2$. Moreover, since $|B_{{2^{n+1}}}| = 2|B_{2^n}|^2$, then $|B_{2^n}| = 2^{2^n - 1}$, and since $\ord_2 2^n ! = 2^n - 1$ (using $d = \ord_r(s)$ for integers $r,s$ if $r^d \mid s$ and $r^{d+1} \nmid s$), then necessarily $B_{2^n}$ comprises a 2-Sylow subgroup of $S_{2^n}$ (which aligns with the $n$-fold wreath presentation; see \cite{Abert_Virag_2005,Kaloujnine_1948}). For example, we have for $n = 1$, then $B_2 = S_2 \cong C_2$, the cyclic group of order 2, and for $n = 2$, then $B_4 \cong D_8 \cong C_2 \wr C_2$, the dihedral group of order 8. Moreover, $B_{s,N}$ is a normal subgroup of $B_N$ (which follows from another straightforward induction argument). Note then $ B_N$ can be further factored as
\begin{equation*}
    B_N = (B_{2^{n-j}} \otimes 1_{2^j}) \bigoplus_{j=1}^{2^{n-j}}  B_{2^j}
\end{equation*}
for $j = 1,\ldots,n-1$. \Cref{fig:ns bin permuton} shows the iid fractal structure of uniform nonsimple butterfly permutations.

To determine $\sigma(B)$ when $B \sim \B(N,\Sigma_S)$, we will determine the necessary form of the GEPP permutation matrix factor of $B \in \B(N)$. We first establish a more general result, which includes the nonsimple (and simple) butterfly matrices. This is an expansion of \cite[Lemma SM2.2]{PT23}, which provided a general GEPP form for Kronecker products of order $2^n$ involving a rotation matrix.

\begin{proposition}
    \label{p:ns factor}
Let $B = B(\theta,A_1,A_2) = (R_\theta \otimes \V I_m)(A_1 \oplus A_2) \in \mathbb R^{2m \times 2m}$, where $A_1,A_2 \in \mathbb R^{m \times m}$ and $|\tan \theta| \ne 1$. Let $P_kA_k = L_kU_k$ be the GEPP factorization for each $A_k$ (where $|(L_k)_{ij}| < 1$ for $i > j$). Then $PB = LU$ is the GEPP factorization of $B$ where 
\begin{align}
P &= (P_1 \oplus P_2)(P_\theta \otimes \V I_m) \label{eq: P ns butterfly}\\ 
L &= (P_1 \oplus P_2)(L_{\hat \theta} \otimes \V I_m)(P_1^T \oplus P_2^T)(L_1 \oplus L_2) \label{eq: L 1st}\\
&= \begin{bmatrix}
L_1 & \V 0 \\ -\tan \hat \theta P_2 P_1^T L_1 & L_2
\end{bmatrix} \label{eq: L 2nd}\\
U &= (L_1^{-1} \oplus L_2^{-1})(P_1\oplus P_2)(U_{\hat \theta}D_\theta \otimes \V I_m)(A_1 \oplus A_2) \label{eq: U 1st} \\
&= \begin{bmatrix}
(-1)^{e(\theta)}\cos \hat \theta U_1 & \sin \hat \theta U_1 A_1^{-1} A_2 \\ \V 0 & \sec \hat \theta U_2
\end{bmatrix}. \label{eq: U 2nd}
\end{align}
\end{proposition}

\begin{proof}
Note the condition $|\tan \theta| \ne 1$ then eliminates the cases when there are pivot search ties  when using GEPP on $R_\theta$, and hence eliminates such ties between separate blocks on $B$. A straightforward computation verifies $PB = LU$ when using the forms given on \eqref{eq: L 1st} and \eqref{eq: U 1st}:
\begin{align*}
LU &= (P_1 \oplus P_2)(L_{\hat \theta} \otimes \V I_m)(U_{\hat \theta}D_\theta \otimes \V I_m)(A_1 \oplus A_2)\\
&= (P_1 \oplus P_2)(R_{\hat \theta}D_\theta \otimes \V I_m)(A_1 \oplus A_2)\\
&= (P_1 \oplus P_2)(P_\theta R_\theta \otimes  \V I_m)(A_1 \oplus A_2)\\
&=(P_1 \oplus P_2)(P_\theta \otimes \V I_m)(R_\theta \otimes \V I_m)(A_1 \oplus A_2)\\
&= PB,
\end{align*}
using the mixed-product property and \eqref{eq: 2x2 gepp}. Additional  straightforward checks  verify equality between forms \eqref{eq: L 1st} and \eqref{eq: L 2nd} as well as \eqref{eq: U 1st} and \eqref{eq: U 2nd}. To establish this is necessarily the (unique) GEPP factorization of $B$, we note this follows directly from  the fact $|L_{ij}| < 1$ for all $i > j$ (cf. \cite[Theorem SM1.4]{PT23}), which necessarily holds since $|(L_{k})_{ij}| < 1$ for $i > j$ for each $k$ and $|\tan \hat \theta| = |\min(\tan\theta,\cot\theta)| < 1$.
\end{proof}

Recall for $B \sim \B(N,\Sigma_S)$, then $B  = (R_\theta \otimes \V I_{N/2})(A_1 \oplus A_2) = B(\boldsymbol \theta)$ for $\boldsymbol \theta = (\theta,\boldsymbol \theta^{(1)},\boldsymbol \theta^{(2)})$ where $A_i = B(\boldsymbol \theta_i) \sim \B(N/2,\Sigma_S)$ and $\theta_i\sim \Unif([0,2\pi))$ iid. Since $\P(|\tan \theta| = 1) = 0$, then $\sigma(B) \in B_N$ almost surely for $B \sim \B(N,\Sigma_S)$ using GEPP by \Cref{p:ns factor}. In particular, then $\sigma(B(N,\Sigma_S)) = B_N$ (almost surely, with complete equality when using the standard tie strategy). Unlike simple scalar butterfly matrices, nonsimple butterfly scalar matrices are not closed under multiplication, and in particular, do not comprise a subgroup of $\SO(N)$. Nevertheless, these both induce subgroups of  $S_N$ through GEPP.

Using \Cref{p:ns factor} and straightforward induction, it follows no GEPP pivot movements are needed if $|\tan \theta_k| < 1$ for all $k$. This further lines up with a similar result for simple butterfly matrices applied to the simple butterfly permutations, since ${\sigma_k} = 1$ only if $|\tan \theta_k| < 1$ for $\bigotimes_{j=1}^n \sigma_j \sim B_{s,N}$.
\begin{corollary}\label{cor: no pivot}
Let $\Sigma_S' = \Sigma_S \cap \{(\cos\theta,\sin\theta)\V I_{2^k}: |\tan \theta| < 1\}$. Fix $\Sigma$. If $B \sim \B(N, \Sigma \cap \Sigma_S')$ or $B \sim \B_s(N,\Sigma \cap \Sigma_S')$, then the GENP and GEPP factorizations of $B$ align.
\end{corollary}
\noindent Again, this extends to also include $|\tan \theta_k| \le 1$ for all $k$ if using the default GEPP tie breaking scheme. 

As seen in \cite{PT23}, the GEPP factorization for the simple case extends to also include alignment with the GE with rook pivoting (GERP) factorizations, which follows a pivot candidate search update path using successive column and then row searches until a candidate has maximal magnitude in both its row and column. This is not the case for nonsimple butterfly matrices, as can be seen by considering the first pivot search when $\max_{ij}|B_{ij}| = \|(A_2)_{1,:}\|_\infty>\|(A_1)_{:,1}\|_\infty$ which would then result in a GERP column pivot movement. For example, for $B = (R_{\pi/4} \otimes \V I_2)(R_{\pi/4} \oplus R_{\pi/3}) \in \B(4)$, then 
\begin{equation*}
    \|B_{:,1}\|_\infty = |B_{11}| = \frac12 = \sqrt{\frac28} < \|B_{1,:}\|_\infty = |B_{14}| = \sqrt{\frac38} = \max_{ij} |B_{ij}|. 
\end{equation*}
Another future area of study can consider induced permutations using GE with other pivoting strategies, such as GERP or GE with complete pivoting (GECP).

We can now establish:

\begin{theorem}\label{thm: unif ns}
Let $B \sim \B(N,\Sigma_S)$. Then $\sigma(B) \sim \Unif(B_{N})$.
\end{theorem}

\begin{proof}
    We will use induction on $n = \log_2 N$. For $n = 1$,  the result holds from \Cref{t: simple b} since $B(2,\Sigma_S) = B_s(2,\Sigma_S) = \Haar(\SO(2))$ and $S_2 = B_2$. Assume the result holds for $n$. Let $B \sim \B(2^{n+1},\Sigma_S)$ and write $B = B(\theta,A_1,A_2) = (R_\theta \otimes \V I_{2^n})(A_1 \oplus A_2)$ for independent $\theta,A_1,A_2$ where $\theta \sim \Unif([0,2\pi))$ and $A_i \sim \B(2^n,\Sigma_S)$. Since  $\P(|\tan \theta| \ne 1) = 1$, then by \Cref{p:ns factor} we have ${\sigma(B)} = (\sigma_1 \oplus \sigma_2)(\sigma_\theta \otimes  1_{2^n}) \in  B_{2^n}$ (almost surely using GEPP with any tie strategy) where $\sigma_i \in  B_{2^{n}}$ and $\sigma_\theta$ comprise the permutations from the corresponding GEPP permutation matrix factors, respectively, for $A_i$ and $R_\theta$. By the inductive hypothesis, $\sigma_i \sim \Unif( B_{2^n}) = \Haar( B_{2^n})$ and so $\sigma_1 \oplus \sigma_2 \sim \Unif( B_{2^n} \oplus  B_{2^n})$. Moreover, $\sigma_\theta \otimes 1 \sim \Haar(\langle (1\ 2) \otimes 1_{2^n}) \sim \Unif( B_{2^{n+1}}\backslash( B_{2^n} \oplus  B_{2^n}))$ (i.e., it comprises a uniform representative of the right cosets of (the closed subgroup) $ B_{2^n} \oplus  B_{2^n}$ inside the group $ B_{2^{n+1}}$ using \eqref{eq: Bn form}), which is independent of $\sigma_1 \oplus \sigma_2$. It follows then from the Subgroup algorithm (\Cref{thm: subgroup alg}) that ${\sigma(B)} \sim \Haar( B_{2^{n+1}}) = \Unif(B_{2^{n+1}})$.
\end{proof}

Considering \Cref{cor: no pivot}, then we have  $B \sim \B(N,\Sigma_S)$ if and only if $B = P B'D$ where $P = P_\sigma$ for $\sigma \sim \Unif( B_N)$ and $B' \sim B(N,\Sigma_S')$ and $D=D(P)$ is a diagonal sign matrix. Future directions can consider numerical properties of butterfly matrices formed where $P$ is sampled with a non-uniform distribution on $ B_N$ with still $B' \sim \B(N,\Sigma_S')$.

Moreover, we can use the explicit structure on ${\sigma} \in  B_N$ from \eqref{eq: P ns butterfly} to determine a recursive formula to determine $\sigma(k)$ for any $k \in [N]$. Let $\sigma_i \in  B_{N/2}$ and $\sigma_\theta \in B_2$ where ${\sigma} = ({\sigma_1} \oplus {\sigma_2})(\sigma_\theta \otimes 1)$. It follows then
\begin{equation*}
    \sigma(k) = \left\{
\begin{array}{ll}
\begin{array}{ll}
\sigma_1(k) & \mbox{if $k \le N/2$}\\
\sigma_2(k-N/2)+N/2 & \mbox{if $k > N/2$}\\
\end{array} & \mbox{if $\sigma_\theta = 1$}\\ \vspace{-.5pc} \\
\begin{array}{ll}
\sigma_2(k)+N/2 & \mbox{if $k \le N/2$}\\
\sigma_1(k-N/2) \hspace{2.5pc} & \mbox{if $k > N/2$}
\end{array} & \mbox{if $\sigma_\theta = {(1 \ 2)}$}

\end{array}    \right.
\end{equation*}


\subsection{\texorpdfstring{{\boldmath$m$}}{n}-nary butterfly permutations}\label{sec: m-nary}

Let $\tau_m = (1 \ 2 \ \cdots \ m) \in S_m$ denote the standard $m$-cycle, where we note $P_{\tau_m}$ is a circulant permutation matrix whose respective powers form a basis for the circulant matrices of order $m$. Let $N = m^n$. We will now similarly define the $m$-nary simple butterfly permutations $B_{s,n}^{(m)}$ and the $m$-nary nonsimple butterfly permutations $B_{n}^{(m)}$ by
\begin{equation*}
     B_{s,n}^{(m)} = \bigotimes_{j=1}^n \langle {\tau_m}\rangle \qquad \mbox{and} \qquad
    B_{n}^{(m)} = \bigoplus_{j=1}^m  B_{{n-1}}^{(m)} \rtimes_{\varphi_m} \langle {\tau_m} \otimes 1_{N/m}\rangle
\end{equation*}
such that $ B_{s,1}^{(m)} =  B_1^{(m)} = \langle {\tau_m}\rangle \cong C_m$, and $\varphi_m(\sigma_1 \oplus \cdots \oplus \sigma_m) = \sigma_{\tau_m(1)} \oplus \cdots \oplus \sigma_{\tau_m(m)} = \sigma_2 \oplus \cdots \oplus \sigma_m \oplus \sigma_1$. Note the previous binary butterfly permutations satisfy $ B_{s,2^n} =  B_{s,n}^{(2)}$ and $ B_{2^n} =  B_{n}^{(2)}$. Moreover, we see $B_{s,n}^{(m)} \cong C_m^n$, while $|B_{n}^{(m)}| = m^{\frac{m^n - 1}{m - 1}}$. Hence, when when $m=p$  is prime, then $B_{n}^{(p)}$ again necessarily comprises a $p$-Sylow subgroup of $S_{p^n}$ since $\ord_p p^n! = \frac{p^n-1}{p-1}$. We then let $\sigma \sim \Unif(B_{s,n}^{(m)})$ and $\sigma \sim \Unif(B_{n}^{(m)})$ denote the uniform random $m$-nary butterfly permutations. \Cref{fig:s tern permuton,fig:ns tern permuton} show the fractal diagrams for sampled uniform ternary butterfly permutations. Future work will explore additional properties for these types of diagrams for $m$-nary butterfly permutations in the context of permuton theory for a geometric approach to studying random permutations.

Unlike in the binary case, these permutation groups do not naturally arise as GEPP permutation matrix factors for general butterfly matrices generated using $\SO(m)$ (e.g., $\sigma(A) \sim \Unif(S_m)$ if $A \sim \Haar(\SO(m))$ by \Cref{cor:unif_perm}). They do trivially occur in the case of applying GEPP directly to $P_\sigma$ for $\sigma$ general butterfly permutations, again since $\sigma(P_\pi^T) = \pi$. $B_n^{(p)}$ has also been studied in the context of the group of $p$-adic automorphisms acting on the $p$-nary rooted tree of depth $n$ \cite{Abert_Virag_2005}. 

\subsection{Diagonal butterfly permutations}\label{sec: diagonal}

Unlike in the scalar case, the GEPP factorization does not go through as directly for diagonal butterfly matrices. When no ties are encountered (i.e., when $|L_{ij}| < 1$ when $i > j$), then GEPP yields unique $L$ and $U$ factors for any row pivot permutation multiples of $A = LU$ (i.e., if $Q$ is a permutation matrix, then $B = QA$ has GEPP factorization $Q^T B = LU$). However, each GEPP matrix factor is not necessarily invariant if the initial system uses column permutations. (For example, GEPP pivoting is agnostic of any columns after the first $n-1$ of $A \in \mathbb C^{n \times m}$ for $m \ge n$, so permuting these columns has no impact on the GEPP permutation matrix factor.) The diagonal butterfly factors that involve conjugation by perfect shuffle permutation matrices (see \eqref{def: diag factor}) thus obstructs a straightforward GEPP factorization result like \Cref{p:ns factor}. 

All of this said, we can still define the diagonal butterfly permutations to again be the permutations corresponding to the GEPP permutation factors of random diagonal butterfly matrices, which we will denote by $B_{s,N}^{(D)}$ and $B_{N}^{(D)}$, respectively, for the simple and nonsimple diagonal butterfly permutations. Since $\B_s(N) \subset \B_s^{(D)}(N) \subset \B^{(D)}(N)$ and $\B_s(N) \subset \B(N) \subset \B^{(D)}(N)$, then $B_{s,N} \subset B_{s,N}^{(D)} \subset B_{s,N} \subset S_N$ and $B_{s,N} \subset B_N \subset B_N^{(D)} \subset S_N$. Note one can see $B_{s,N} \ne B_{s,N}^{(D)}$ and $B_N \ne B_N^{(D)}$ by considering first $N = 4$ (since $N = 2$ has equality everywhere since $B_{s,2} = S_2$). 
We note $B \in \B_s^{(D)}(4)$ of the form
\begin{align*}
    B = Q_2(R_{\pi/3}\oplus \V I_2)Q_2(\V I_2 \otimes R_{\pi/4}) 
    = \frac1{\sqrt 8} \begin{bmatrix}
        1 & 1 & \sqrt 3 & \sqrt 3\\
        -2 & 2 & 0 & 0\\
        -\sqrt 3 & -\sqrt 3 & 1 & 1\\
        0&0&-2 & 2
    \end{bmatrix}
\end{align*}
 has GEPP factorization $PB = LU$ for $P = P_{(1 \ 3 \ 2)} = P_{(2 \ 3)(1 \ 2)}$ and
\begin{equation*}
    L = \begin{bmatrix}
        1 & 0& 0& 0 \\ \frac{\sqrt{3}}2 & 1 & 0& 0\\
        -\frac12 & -\frac1{\sqrt 3} & 1 & 0\\ 0 & 0 & -\frac{\sqrt 3}2 & 1
    \end{bmatrix} \quad \mbox{and} \quad U = \frac1{\sqrt 8}\begin{bmatrix}
        -2 & 2 & 0 & 0\\0 & -2\sqrt 3 & 1 & 1\\ 0 & 0 & \frac{4}{\sqrt 3} & \frac{4}{\sqrt 3}\\0&0&0&4
    \end{bmatrix}.
\end{equation*}
Hence, $(1 \ 3 \ 2) \in B_{s,4}^{(D)}$. By considering $\sigma_k = \sigma(\V I_{2^k} \otimes B) \in B^{(D)}_{s,2^{k+2}}$ such that $P_{\sigma_k} = \V I_{2^k} \otimes P_{(1 \ 3 \ 2)}$ so $\sigma_k^3 = 1$, 
then it follows $B^{(D)}_{s,N} \subset B^{(D)}_N$ always contains an element of order $3$ for all $n \ge 2$. Since all elements of $B_{s,N} \subset B_N$ must have order a power of 2 (since they belong to a 2-Sylow subgroup of $S_{2^n}$), then the inclusions $B_{s,N} \subset B^{(D)}_{s,N}$ and $\B_N \subset \B^{(D)}_N$ are necessarily strict. 

Recall if one applies GEPP to a diagonal multiple of a permutation matrix $DP_\pi$ for $D$ a diagonal matrix and $\pi \in S_n$, then one again necessarily recovers the inverse of the associated permutation matrix as the permutation factor, i.e, $\sigma(DP_\pi) = \pi^{-1}$. If $\boldsymbol \theta$ is a vector with $\theta_j \in \{0,\frac\pi2\}$ for each $j$, then $B(\boldsymbol \theta)$ is a diagonal multiple of a permutation matrix (e.g., $R_{\pi/2} = P_{(1 \ 2)}((-1) \oplus 1)$), and so $\tilde B = \{\sigma(B(\boldsymbol{\theta})): \theta_j \in \{\frac\pi4,\frac\pi2\}\}$ is a subset of the butterfly permutations. Since $|\tilde B| = 2^n$ when $B(\boldsymbol{\theta}) \in \B_s(N)$ and $|\tilde B| = 2^{2^n - 1}$ when $B(\boldsymbol{\theta}) \in \B(N)$, then necessarily $\tilde B$ aligns with $B_{s,N}$ and $B_N$; so no more butterfly permutations are possible than those generated using input angles $\{0,\frac\pi2\}$ for the scalar butterfly matrices. 

We see this is not the case then for simple diagonal butterfly matrices, since these are also formed using $2^n - 1$ input angles; in fact, we have actually $\tilde B = q_nB_Nq_n^{-1}$ comprises another $2$-Sylow subgroup of $S_{2^n}$ when using diagonal butterfly matrices (each butterfly factor layer between $\B(N)$ and $\B_s(N)^{(D)}$ correspond by sums of conjugations by perfect shuffles (permutation $q_k$ that corresponds to the permutation matrix $Q_k$ such that $Q_k(A \otimes B)Q_k^T = B \otimes A$; cf. \eqref{def: diag factor}) and so inductively $\tilde B$ contains both $q_n(B_{N/2} \oplus B_{N/2})q_n^{-1}$ as a left factor and also $q_n({(1 \ 2)} \otimes 1)q_n^{-1}$ as a right factor, so then contains $q_n B_N q_n^{-1}$). Necessarily $|B_N| < |B_{s,N}^{(D)}|$ since $B_{s,N}^{(D)}$ additionally contains an element of order 3. Future work can explore additional properties of diagonal butterfly permutations. Our focus in the remainder of the current text will stay restricted to the scalar butterfly models.

\section{Longest increasing subsequence} \label{sec: lis}

\begin{figure}[t]
    \centering
    \subfloat[Permutation Matrix]{%
        \begin{minipage}{0.28\textwidth}
            \centering
            \[
            \begin{bmatrix}
             \!0&\!0&\!0&\! 1&\!0&\!0&\!0&\!0 \\
             \!0&\!0&\!0&\!0&\!0&\!0&\!0&\! 1\\
             \!0&\!0&\!0&\!0&\! 1&\!0&\!0&\!0\\
             \! \fbox{\textcolor{red}{1}}&\!0&\!0&\!0&\!0&\!0&\!0&\!0\\
             \!0&\!0&\! \fbox{\textcolor{red}{1}}&\!0&\!0&\!0&\!0&\!0\\
             \!0&\!0&\!0&\!0&\!0&\! \fbox{\textcolor{red}{1}}&\!0&\!0\\
             \!0&\!0&\!0&\!0&\!0&\!0&\! \fbox{\textcolor{red}{1}}&\!0\\
             \!0&\! 1&\!0&\!0&\!0&\!0&\!0&\!0
            \end{bmatrix}
            \] \vspace{1pc}
        \end{minipage}
    }%
    \hspace{1pc}
    \subfloat[Diagram $(j,{\sigma(j)})/8$]{%
        \begin{minipage}{0.28\textwidth}
            \fbox{\includegraphics[width=1\textwidth]{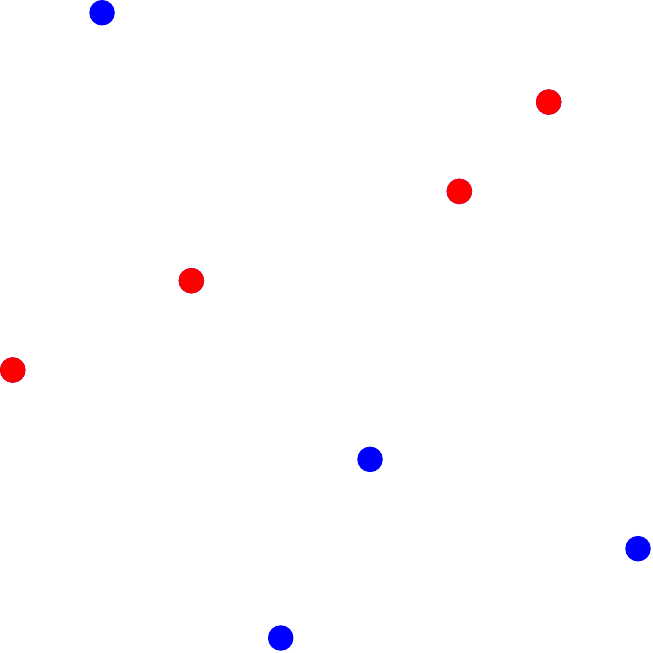}}%
        \end{minipage}
    }%
    \hspace{.1pc}
    \subfloat[Young Tableau]{%
        \begin{minipage}{0.24\textwidth}
            \centering
            \vspace{3pc}
            \begin{ytableau}
                \none & & & & \\
                \none & & & \none \\
                \none & & \none & \none \\
                \none & & \none & \none \\
            \end{ytableau}
            \vspace{3pc}
        \end{minipage}
    }%
    \caption{Three views of the LIS $(4 \ 5 \ 6 \ 7)$ within $(4 \ 8 \ 5 \ 1 \ 3 \ 6 \ 7 \ 2)$, as seen in the (a) associated permutation matrix, (b) diagram of plotted pairs $(j,\sigma(j))/8 \subset [0,1]^2$, and (c) corresponding Young tableau where the length of the top row yields the LIS.}
    \label{fig:LIS ex}
\end{figure}

A classical well-studied statistic for permutations includes the longest increasing subsequence (LIS). We recall the definition for $\sigma \in S_m$, then $L(\sigma)$, the LIS of $\sigma$, gives the length of the longest subsequence of indices $i_1 < i_2 < \ldots < i_k$ such that $\sigma$ maintains the monotonicity property on this sequence, i.e., $\sigma(i_j) < \sigma(i_{j+1})$. Hence, we can write
\begin{equation*}
    L(\sigma) = \max \{k \in [n]: i_1 < \ldots < i_k, \sigma(i_1) < \ldots < \sigma(i_k)\}.
\end{equation*}
For example, if $\sigma = (4 \ 8 \ 5 \ 1 \ 3 \ 6 \ 7 \ 2) \in S_8$ (using the notation that puts $\sigma(k)$ in index $k$ for the associated permutation array),  then $L(\sigma) = 4$, as seen with both increasing subsequences $(4\ 5 \ 6 \ 7)$ and $(1\ 3 \ 6 \ 7)$. \Cref{fig:LIS ex} shows three equivalent ways of visualizing the particular LIS $(4\ 5 \ 6 \ 7)$, including: (1) using a down-right path connecting 1s in the associated permutation matrix; (2) up-right path in the associated diagram mapping $({j},{\sigma(j)})/8$; and (3) the length of the top row of the associated Young tableau.

The study of the LIS question for random permutations first focused on the uniform case. See \cite{Romik_2015} for a thorough overview of the total research path for uniform permutations. See also \cite{aldous1999longest, corwin2018comments} for further discussions and connections to the Kardar-Parisi-Zhang universality phenomenon. The investigation started with Ulam's first efforts to establish the proper power law scaling $n^{1/2}$ for $\E L(\sigma)$ in the 1960s for $\sigma \sim \Unif(S_n)$ (see \cite{Ulam61}). This was followed by the establishment of the LLN results in \cite{logan1977variational, vershik1977asymptotics, aldous1995hammersley}. This line of research culminated in \cite{BDJ99}, when Baik, Deift, and Johansson established
\begin{equation}\label{eq: TW}
    \lim_{n \to \infty} \P((L(\sigma) - 2n^{1/2})n^{-1/6} \le t) = \exp \left(-\int_t^\infty (x-t)^2 u(x) \ dx \right)
\end{equation}
where $u(x)$ is a solution to the Painlev\'e II equation $u_{xx} = 2u^3+xu$. The right-hand side term in \eqref{eq: TW} is the cdf for the Tracy-Widom (TW$_2$) distribution, which was originally introduced in the context of level spacing distributions of the Airy kernel and models the rescaled largest eigenvalue of $A \sim \GUE(n)$   \cite{Tracy_Widom_1993,Tracy_Widom_1994,Tracy_Widom_1996}. 
The convergence to the Tracy-Widom distribution, along with further results, were also shown in \cite{BOO, Okounkov}. Recently, Dauvergne and Vir\'ag \cite{dauvergne2021scaling} identified the scaling limit of the longest increasing subsequence in a uniform permutation to the directed geodesic of the directed landscape, which is a universal object in the Kardar-Parisi-Zhang universality class. This intimately connects problems from the study of random permutations to those in random matrix theory. Our goal is to further bolster these connections.


The LIS question has also been studied for certain non-uniform random permutation models in the literature. For example, the LIS for locally uniform permutations was investigated in \cite{deuschel1995limiting, deuschel1999increasing, sjostrand2023monotone}; see also \cite{dubach2024increasing} for generalizations. The LIS for the Mallows permutation model with various distance metrics (including Kendall's $\tau$, Cayley distance, Spearman's $\rho$, and Spearman's footrule) has been analyzed in \cite{Basu_Bhatnagar_2017, bhatnagar2015lengths,Mueller_Starr_2013,kammoun, Zhong_2023}. The LIS for pattern-avoiding permutations was studied in \cite{bassino2022linear, deutsch2003longest, madras2017longest, mansour2020permutations}. Power-law bounds on the LIS for permutations sampled from Brownian separable permutons (the scaling limit of separable permutations---a class of pattern-avoiding permutations) have recently been obtained in \cite{BDG24}. The LIS for random colored permutations has been analyzed by Borodin \cite{Borodin_1999}, and the LIS for  wreath products of the form $\Gamma^n \rtimes S_n$ for $\Gamma \subset S_k$, using an $S_n$ action on certain block groups (e.g., the hyperoctohedral group), has recently been analyzed by Chatterjee and Diaconis \cite{chatterjee2024vershik}.

\color{black}

\begin{figure}[t!] 
    \centering
    \subfloat[$\Unif(S_N)$]{%
        \fbox{\includegraphics[width=0.28\textwidth]{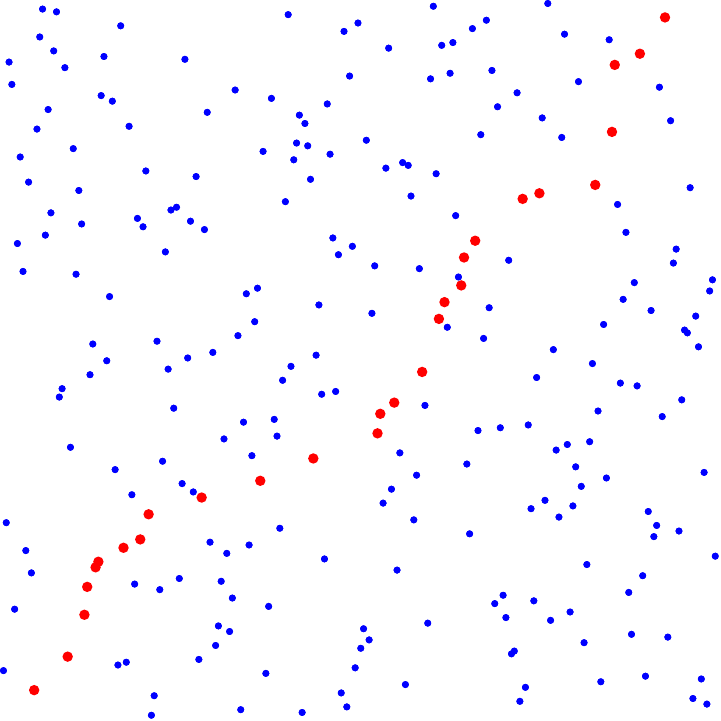}%
        \label{fig:unif permuton lis}%
        }
        }%
    \hspace{.1pc}
    \subfloat[$B_{s,N}$]{%
        \fbox{\includegraphics[width=0.28\textwidth]{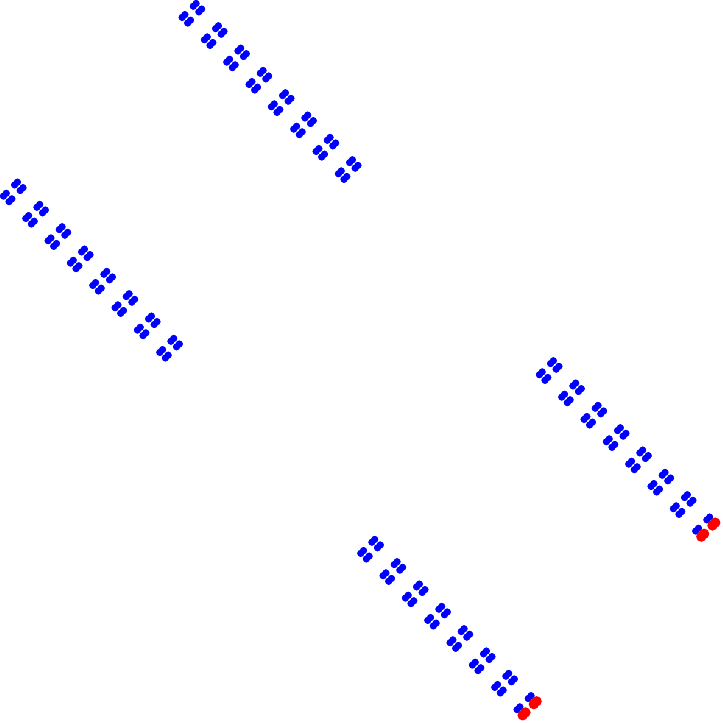}%
        \label{fig:s bin permuton lis}%
        }
        }%
    \hspace{.1pc}
    \subfloat[$B_N$]{%
        \fbox{\includegraphics[width=0.28\textwidth]{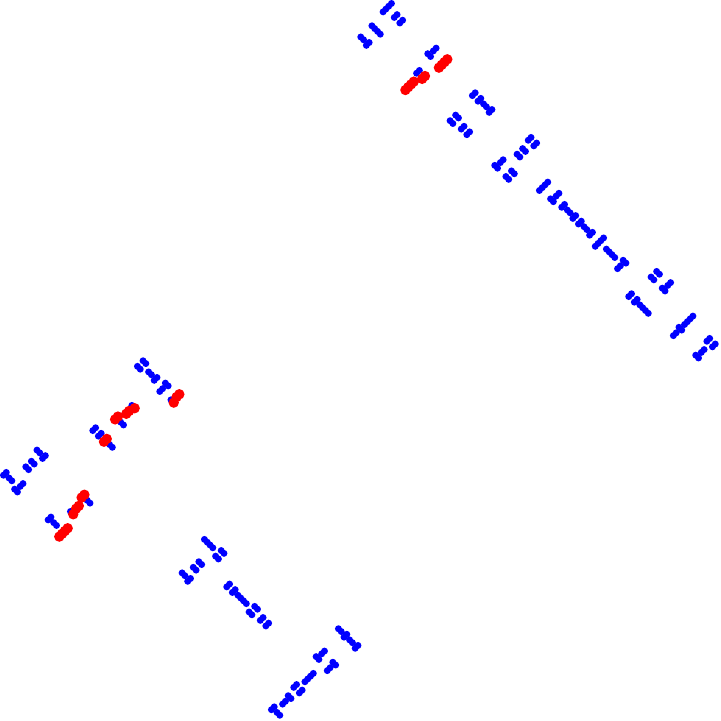}%
        \label{fig:ns bin permuton lis}%
        }
        }%
    \\
    \subfloat[$B_N^{(D)}$]{%
        \fbox{\includegraphics[width=0.28\textwidth]{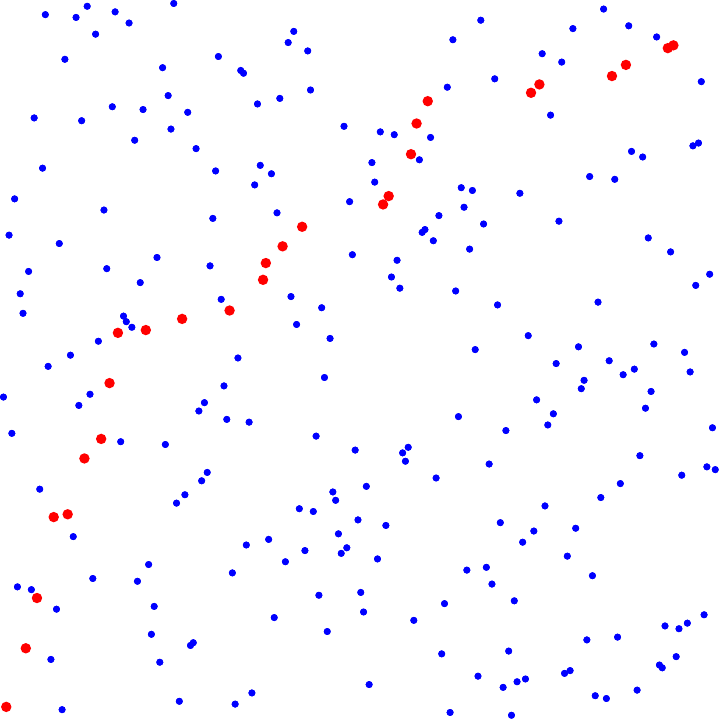}%
        \label{fig:diag permuton lis}%
        }
        }%
    \hspace{.1pc}
    \subfloat[$B_{s,n}^{(3)}$]{%
        \fbox{\includegraphics[width=0.28\textwidth]{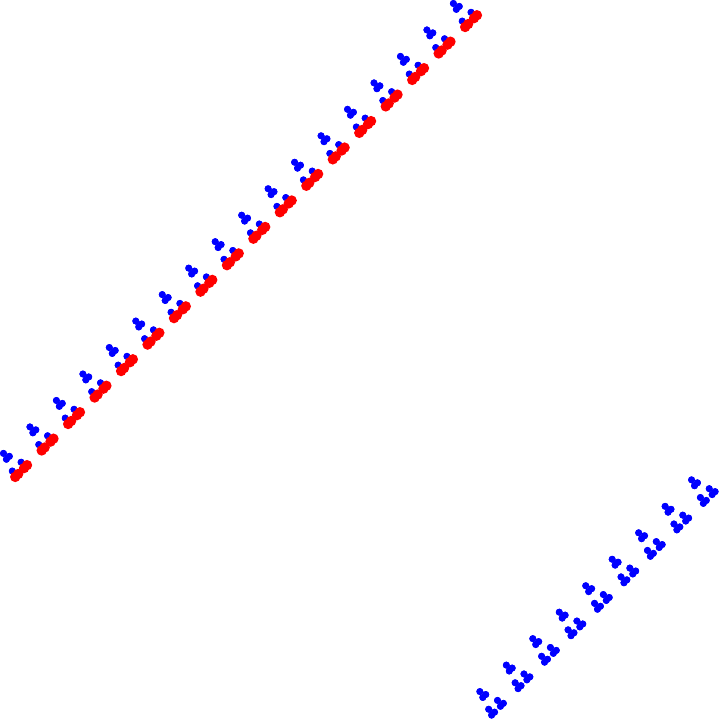}%
        \label{fig:s tern permuton lis}%
        }
        }%
    \hspace{.1pc}
    \subfloat[$B_{n}^{(3)}$]{%
        \fbox{\includegraphics[width=0.28\textwidth]{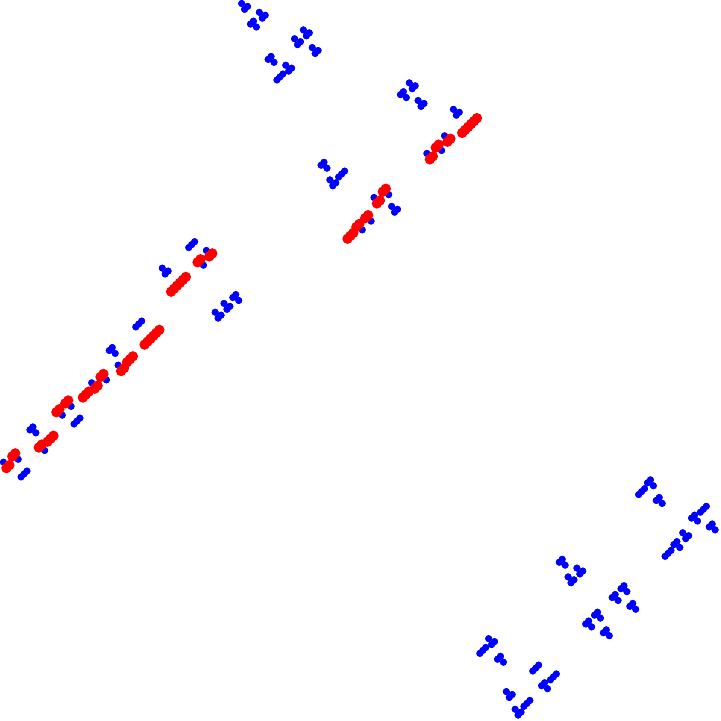}%
        \label{fig:ns tern permuton lis}%
        }
        }%
    \caption{Diagrams for a LIS (in red) for random permutations of size $N = 2^{8} = 256$ for (a) uniform, (b) simple binary butterfly, (c) nonsimple binary butterfly, and (d) diagonal butterfly permutations, and $N = 3^5 = 243$ for (e) simple ternary butterfly and (f) nonsimple ternary butterfly permutations.}
    \label{fig:permutons lis}
\end{figure}

We are now interested in addressing the LIS question for random butterfly permutations. \Cref{fig:permutons lis} shows sample diagrams of random permutations, including binary and ternary butterfly permutations, along with the plot of a particular LIS in red. We further note the explicitly differing behavior for the LIS from the simple and nonsimple butterfly permutations, as compared to the uniform (and diagonal butterfly) permutations. We will now outline explicit descriptions for the corresponding behavior exhibited in (b), (c), (e) and (f) from \Cref{fig:permutons lis}.

\subsection{Simple butterfly permutations} \label{sec: lis s}

The extreme structured nature of simple butterfly permutations allows simple computations for some statistical questions that are traditionally computationally had. We will now show that this includes the LIS question. We will first address the binary case, $\sigma_n \sim \Unif(B_{s,N})$. 

\subsubsection{\texorpdfstring{{\boldmath$m = 2$}}{n}}
If $N = 2$, then it's clear $L(\sigma) \sim \Unif(\{1,2\})$ since then $B_{s,2} = S_2$. For $n \ge 1$, consider $\sigma_{n+1} \in \Unif( B_{s,2N})$, where then $\sigma_{n+1} = \sigma_1 \otimes  \sigma_{n}$ for $\sigma_1 \sim \Unif(B_{s,2}) = \Unif(S_2)$ independent of $ \sigma_n \sim \Unif(B_{s,N})$. It follows 
\begin{equation}\label{eq: s bin}
    P_{\sigma_{n+1}} = \left\{\begin{array}{cc}\begin{bmatrix}
        P_{ \sigma_n} & \V 0\\ \V 0 & P_{ \sigma_n}
    \end{bmatrix} & \mbox{if $\sigma_1 = 1$,}\\ \vspace{-5pt}\\
    \begin{bmatrix}
        \V 0 & P_{ \sigma_n}\\ P_{ \sigma_n} & \V 0
    \end{bmatrix} & \mbox{if $\sigma_1 \ne 1$.}
    \end{array}\right.
\end{equation}
Recall $L(\sigma_{n+1})$ can be computed as a down-right path connecting 1s in the corresponding permutation matrix $P_{\sigma_{n+1}}$ (cf. \Cref{fig:LIS ex}). Hence, we see $L(\sigma_{n+1})$ is formed by combining 2 copies of an LIS for $\sigma_n$ for the first $N/2$ and last $N/2$ indices when $\sigma_1 = 1$, while if $\sigma_1 \ne 1$ then an LIS for $\sigma_{n+1}$ would be formed completely within either of the corresponding $P_{\sigma_n}$ blocks using \eqref{eq: s bin}. Hence, we can summarize this as
\begin{equation*}
    L({\sigma_{n+1}}) = \left\{\begin{array}{cc}2 L(\sigma_n) & \mbox{if $\sigma_1 = 1$,}\\ \vspace{-5pt}\\
    L(\sigma_n) & \mbox{if $\sigma_1 \ne 1$.}
    \end{array}\right.
\end{equation*}
Since $\P(\sigma_1 = 1) = \frac12$, then it follows $L(\sigma_{n+1}) = 2^{\eta_{n+1}} L(\sigma_n)$ where $\eta_{n + 1} \sim \Bern(\frac12)$ is independent of $L(\sigma_n)$. Inductively, it follows then $\log_2 L(\sigma_n) = \eta_{n} + \log_2 L(\sigma_{n-1}) \sim \Binom(n,\frac12)$. We can combine this with a standard application of the CLT to yield:
\begin{proposition}\label{prop:long inc sub}
    Let $\sigma_n \sim \Unif(B_{s,N})$. Then $\log_2 L(\sigma_n) \sim \Binom(n,\frac12)$, while $(\log_2 L(\sigma_n) - n/2)2n^{-1/2}$ converges in distribution to $Z \sim N(0,1)$.
\end{proposition}

In particular, then $(L(\sigma)N^{-1/2})^{(2/\sqrt{\ln 2})/ \sqrt{\ln N}}$ converges in distribution to a standard log-normal random variable. This is notably different than the the limiting TW$_2$ behavior of $L(\sigma)$ when $\sigma\sim\Unif(S_n)$. Ulam's initial foray into the LIS question focused on finding the correct scaling for $\E L(\sigma)$ when $\sigma \sim \Unif(S_n)$. This is a straightforward calculation for simple butterfly permutations, since 
\begin{equation*}
    \E L(\sigma_n) = \sum_{k=0}^n \binom{n}k \left(\frac12\right)^{n-k} = \left(\frac32\right)^n = N^{\log_2\frac32} \approx N^{0.585}
\end{equation*}
when $\sigma_n \sim \Unif(B_{s,N})$. A similar computation yields
\begin{equation*}
    \E L(\sigma_n)^m = N^{\log_2(1+2^m)-1} = N^{m-1+o_m(1)}.
\end{equation*}
In particular, then
\begin{equation*}
    \operatorname{Var}(L(\sigma_n)) = N^{\log_2(5/2)}(1 - N^{\log_2(9/10)}) \approx N^{1.322}(1 - o(1))
\end{equation*}
and so the standard deviation has exponent $\frac12 \log_2 \frac52 = \log_2 \frac{\sqrt 10}2 > \log_2 \frac32$ is of a higher order than that for the mean. 

\begin{remark}\label{rmk: ELIS > sqrt n}
    This again differs from the uniform case, where the standard deviation being of a smaller order than the mean yields a deterministic LLN result for $L(\sigma)N^{-1/2}$ converging almost surely to 2. {It also differs from most of other non-uniform random permutation models (see e.g. \cite{bhatnagar2015lengths, Borodin_1999, chatterjee2024vershik, deuschel1995limiting, kammoun2, Mueller_Starr_2013}), where the standard deviation of LIS is of smaller order than the mean and the LIS satisfies an LLN.}
\end{remark}

Let $D(\sigma)$ denote the longest decreasing subsequence of $\sigma \in S_m$, which can be defined analogously to the LIS; also analogously, $D(\sigma)$ can be realized by computing the longest up-right path connecting 1s in the associated permutation matrix (and similarly, the height of the first column of the associated Young tableau). Hence, $D(\sigma_n)$ for $\sigma_n \sim \Unif(B_{s,N})$ has the inverse relationship with the prior level using \eqref{eq: s bin}, as $D(\sigma_n)$ doubles whenever $L(\sigma_n)$ matches the prior level and vice versa. Hence, their product doubles every round. This leads to: 
\begin{proposition}\label{prop: LDS bin}
    Let $\sigma_n \sim \Unif(B_{s,N})$. Then $L(\sigma_n)D(\sigma_n) = 2^n$.
\end{proposition}


\begin{remark}
    In terms of the associated Young tableau diagrams for these permutations, they then necessarily generate rectangular diagrams with $L(\sigma)$ columns and now $D(\sigma)$ rows by \Cref{prop: LDS bin};  so these rectangles have fixed area. Moreover, each Young tableau corner lies on the line $y = \frac{N}x$ for $x > 0$. A simple hook length formula follows also from \Cref{prop:long inc sub}, since the only associated partitions of $N$ encountered are of the form $(2^k,2^k,\ldots,2^k) \vdash N$, with then $\binom{n}k$ simple butterfly permutations then corresponding to each such partition.
\end{remark} 

Using \Cref{prop:long inc sub,prop: LDS bin}, since then $\log_2 D(\sigma_n) = n - \log_2 L(\sigma_n)$ for $\sigma_n \sim \Unif(B_{s,N})$, we have immediately:
\begin{corollary}
    Let $\sigma_n \sim \Unif(B_{s,N})$. Then $\log_2 D(\sigma_n) \sim \Binom(n,\frac12)$, and $(\log_2 D(\sigma_n) - n/2)2n^{-1/2}$ converges in distribution to $Z \sim N(0,1)$.
\end{corollary}

\subsubsection{General \texorpdfstring{{\boldmath$m$}}{n}}
To now explore the general $m$-nary simple butterfly permutations, we will use $m = 3$ for illustrative purposes. Recall $B_{s,1}^{(3)} = \langle (1 \ 2 \ 3) \rangle \cong C_3$. If $\sigma_1 \in B_{s,1}^{(3)}$, then $L(\sigma_1) = 3$ if $\sigma_1 = 1$ and $L(\sigma_1) = 2$ if $\sigma_1 \ne 1$; in particular, $L(\sigma_1) = \max(j, 3 -j)$ if $\sigma_1 = (1 \ 2 \ 3)^j$ for $j = 1,2,3$. Now for $\sigma_{n+1} \in \B_{s,n+1}^{(3)}$, then $\sigma_{n+1} = \sigma_1 \otimes \sigma_n$ for again $\sigma_1 \in B_{s,1}^{(3)}$ and $\sigma_n \in B_{s,n}^{(3)}$, where now
\begin{equation}\label{eq: tern block}
    P_{\sigma_{n+1}} = \left\{\begin{array}{cl}\begin{bmatrix}
        P_{ \sigma_n} & \V 0 & \V 0\\ \V 0 & P_{ \sigma_n} & \V 0\\\V 0 & \V0 & P_{ \sigma_n} 
    \end{bmatrix} & \mbox{if $\sigma_1 = 1$,}\\ \vspace{-5pt}\\
    \begin{bmatrix}
        \V0& \V 0 & P_{ \sigma_n}\\ P_{ \sigma_n} & \V 0 & \V0\\\V0& P_{ \sigma_n} & \V 0
    \end{bmatrix} & \mbox{if $\sigma_1 = (1 \ 2 \ 3)$.}\\ \vspace{-5pt}\\
    \begin{bmatrix}
        \V0&P_{ \sigma_n} & \V 0\\\V0& \V 0 & P_{ \sigma_n}\\P_{ \sigma_n} & \V 0 & \V0
    \end{bmatrix} & \mbox{if $\sigma_1  = (1 \ 3 \ 2)$.}
    \end{array}\right.
\end{equation}
Again considering down-right paths connecting 1s in $P_{\sigma_{n+1}}$, we similarly now have $L(\sigma_{n+1}) = \max(j, 3-j) L(\sigma_{n})$ for $\sigma_1 = (1 \ 2 \ 3)^j$. Now note if $X \sim \Unif([3])$, then 
\begin{equation*}
    \E \max(X, 3 - X) = \frac13 \sum_{j = 1}^3 \max(j, 3 - j) = \frac73.
\end{equation*}
It now inductively follows then if $\sigma_n \sim \Unif(B_{s,n}^{(3)})$, then \begin{equation*}
    L(\sigma_n) \sim \prod_{j=1}^n \max(X_j, 3 - X_j)
\end{equation*}
where $X_j \sim \Unif([3])$ iid. In particular, then
\begin{equation*}
    \E L(\sigma_n) = \left(\frac73\right)^n = N^{\log_3(7/3)} \approx N^{0.7712}.
\end{equation*}
Note this matches the form from the $m = 2$ case since if $X \sim \Unif(\{1,2\})$, then $\log_2\left(\max(X,2-X)\right) \sim \Bern(\frac12)$. This exact reasoning carries over directly to the general $m$-nary case. In particular, we note then $\log_m L(\sigma_n)$ is the sum of iid terms (again, as seen also in the binary case), so we immediately can derive a CLT result now for this as well as the general $m$-nary case, which has the exact same reasoning carry through directly. 

First, for $X \sim \Unif([m])$, we  define
\begin{align*}
    \mu_m &= \E \log_m \max(X,m-X) \\
    &= \frac1m \sum_{j = 1}^m \log_m(\max(j,m-j)) \\
    &= \frac1m \left( \log_m \left(\frac{m!}{\floor{m/2}!}\right)^2 - \log_m(2 + r_m \cdot (m-2))\right)\\
    &= 1 - o_m(1),
\end{align*}
where $r_m = m \pmod 2$ and using Stirling's approximation $\ln n! = n \ln n - n + o_n(1)$ for the last line. For example, $\mu_2 = \frac12$ aligns with the binary case before; $\mu_3 \approx 0.7970$; $\mu_5 \approx 0.817584$; $\mu_7 \approx 0.845795$; while $\mu_{999983} \approx 0.977789$ (using the largest 6-digit prime number). A similar, but less satisfactory form, can be computed for
\begin{align*}
    \nu_m &= \operatorname{Var} \log_m \max(X,m-X) \\
    &= \frac1m\left(1 - \frac1m\right) \sum_{j=1}^m \log_m^2 \max(j,m-j) \\
    &\hspace{2pc}- \frac2{m^2} \sum_{1\le i < j \le m} \log_m \max(i,m-i) \cdot \log_m \max(j,m-j).
\end{align*}
For example, we have $\nu_2 = \frac14$ (as seen earlier), $\nu_3 \approx 0.0302695$ while $\nu_5 \approx 0.014709$. Moreover, we can define a sequence $\alpha_m$ where 
\begin{equation}\label{eq: alpha_m}
    m^{\alpha_m} := \E \max(X,m-X) = \frac1m \sum_{j = 1}^m \max(j,m-j) = \frac{3 m^2 + r_m}{4m},
\end{equation}
again using $r_m = m \pmod 2$. Note $2^{\alpha_m} = \frac32$ matches what was seen in the binary case, while $3^{\alpha_3} = \frac73$ matches the ternary example above. Additionally, it follows
\begin{equation}\label{eq: asympt am}
    \alpha_m = 1 - \frac{\ln(4/3) + o_m(1)}{\ln m} = 1 - O_m\left(\frac1{\ln m}\right) = 1-o_m(1).
\end{equation}

We can now summarize the general $m$-nary simple butterfly permutation LIS question:
\begin{theorem}\label{t: s lis}
    Let $\sigma_n \sim \Unif(B_{s,n}^{(m)})$. Then $L(\sigma_n) \sim \prod_{j=1}^n \max(X_j, m -X_j)$ where $X_j \sim \Unif([m])$ iid, with $\E L(\sigma_n) = N^{\alpha_m}$. Moreover, $(\log_m L(\sigma_n) - \mu_m n) (\nu_m n)^{-1/2}$ converges in distribution to $Z \sim N(0,1)$.
\end{theorem}

For future use, we additionally note $\alpha_m$ is an increasing sequence:
\begin{proposition}\label{p: am incr}
    $\alpha_m \le \alpha_{m+1}$.
\end{proposition}
\begin{proof}
    If $m$ is even, then $\alpha_m = u(m) := 1 - \frac{\ln (4/3)}{\ln m}$, while if $m$ is odd, then $\alpha_m = v(m)$, where $v(x) := u(x) + \frac{\ln(1 + \frac1{3 x^2})}{\ln x}$. 
    If $m$ is even, then ${\alpha_{m+1}} = v(m+1) > u(m+1) > u(m) = \alpha_m$. Now to show necessarily then $\alpha_m < \alpha_{m+1}$, we only need to consider the case when $m$ is odd. This is now equivalent to showing $\ln(1 + \frac1{3m^2}) < \ln \frac43\cdot (1 - \frac{\ln m}{\ln (m+1)})$ for $x \ge 2$; let $w(x) = \ln(1 + \frac1{3x^2})$ and $s(x) = \ln\frac43\cdot (1 - \frac{\ln x}{\ln (x + 1)})$. Since $w''(x) = \frac{2(9x^2+1)}{x^2(3x^2 + 1)^2} > 0$ for $x \ge 2$, then $w$ is strictly convex; moreover, while $s'(x) <0$ for all $x > 1$ and $s''(1) > 0$, then $s(x)$ is also strictly convex, and so then $w(1) = s(1)$ while $w'(1) < s'(1)$, it follows $w(x) \le s(x)$ for all $x \ge 1$.
\end{proof}

We can similarly consider the length of the longest decreasing subsequence $D(\sigma_n)$ for $\sigma_n \sim \Unif(B_{s,n}^{(m)})$. Again, using the block structure, we can consider only the up-right paths connecting 1s in the corresponding permutation matrix. However, as seen in \eqref{eq: tern block} now considering $\sigma_{n+1} = \sigma_1 \otimes \sigma_n \in B_{s,n+1}^{(m)}$, then we see $D(\sigma_{n+1}) = D(\sigma_n)$ when $\sigma_1 = 1$ while now only 2 blocks can be traversed when $\sigma_1 \ne 1$ to form a longest decreasing subsequence due to the circulant form of $P_{\tau_m}$, i.e., $D(\sigma_{n + 1}) = 2D(\sigma_n)$ when $\sigma_1 \ne 1$ for all $m$. Since now $\P(\sigma_1 = 1) = \frac1m$, it  follows directly:
\begin{theorem}
    Let $\sigma_n \sim \Unif(B_{s,n}^{(m)})$. Then $\log_2 D(\sigma_n) \sim \Binom(n,\frac1m)$. Moreover, $(\log_2 D(\sigma_n) - n/m)\frac{m}{\sqrt{m-1}}n^{-1/2}$ converges in distribution to $Z \sim N(0,1)$.
\end{theorem}
Hence, we only have $L(\sigma_n)D(\sigma_n) = N$ for $m = 2$. In general, if $\sigma_{n+1} = \sigma_1 \otimes \sigma_n \in B_{s,n+1}^{(m)}$, then $L(\sigma_{n+1})D(\sigma_{n+1}) = m L(\sigma_n)D(\sigma_n)$ when $\sigma_1 = 1$ and $L(\sigma_{n+1})D(\sigma_{n+1}) = 2 \max(j,m-j) L(\sigma_n)D(\sigma_n)$ when $\sigma_1 = \tau_m^j$ for $j = 1,\ldots,m-1$.

\subsection{Nonsimple butterfly permutations}\label{sec: ns lis}
Now we will explore the LIS question for nonsimple butterfly permutations. We will again first focus on the binary case before the general $m$-nary butterfly permutations.

For $\sigma_{n+1} \in B_{n+1}^{(m)}$, we can write $\sigma_{n+1} = (\sigma_\theta \otimes 1_{N/m})(\sum_{j =1}^m \sigma_n^{(i)})$ for $\sigma_n^{(i)} \in B_{n}^{(m)}$,  $\sigma_\theta \in B_1^{(m)} = \langle \tau_m\rangle$ when $n \ge 1$. Again, considering down-right paths connecting 1s in the associated permutation matrix, as previously done with \eqref{eq: s bin} and \eqref{eq: tern block}, we have necessarily exactly two possible block matrix paths possible for an LIS to reside due to the circulant form of $P_{\tau_m}$. Hence, we have
\begin{equation*}
    L(\sigma_n) = \left\{\begin{array}{cl}
         \sum_{j=1}^m L( \sigma_n^{(j)}) & \mbox{if $\sigma_\theta = 1$}\\
    \max\left(\sum_{j=1}^k L( \sigma_n^{(j)}),\sum_{j = k+1}^m L( \sigma_n^{(j)})\right) & \mbox{if $\sigma_\theta = \tau_m^j$, $1\le k\le m$}.
    \end{array}
    \right.
\end{equation*}
If $\sigma_n^{(j)} = \sigma_n^{(1)}$ for all $j$, then this form aligns exactly with the simple case. Since $\P(\sigma_\theta = 1) = \frac1m$ if $\sigma_\theta \sim \Unif(B_{1}^{(m)})$, it now immediately follows:
\begin{proposition}\label{prop: lis 2-Xn}
    Let $\sigma_n \sim \Unif(B_n^{(m)})$. Then $L(\sigma_n) \sim X_n$, where $X_0 = 1$ and $$X_{n+1} = \max\left(\sum_{j=1}^k X_n^{(j)},\sum_{j=k+1}^m X_n^{(j)}\right)$$ with probability $\frac1m$ for $k = 1,2,\ldots,m$, where $X_n^{(j)} \sim X_n$ iid, with full support on $[m^n]$.
\end{proposition}

For notational convenience, we will  shift our focus exclusively now to using $X_n$ in place of $L(\sigma_n)$. Although $X_n$ has a straightforward recursive structure, its underlying statistical properties are not as immediate as was seen in the simple case. In particular, finding the proper scaling for $\E X_n$ is a nontrivial task once again (akin to Ulam's original undertaking in the uniform case). Establishing power law bounds on $\E X_n$ will  be our main target for the remainder of this section. We will first consider the case when $m = 2$.

\subsubsection{\texorpdfstring{{\boldmath$m = 2$}}{n}}

From \Cref{prop: lis 2-Xn} for $m = 2$, we have $L(\sigma_n) \sim X_n$ where $X_0 = 1$ and
\begin{equation*}
    X_{n+1} = \left\{\begin{array}{ll}
    X_n + X_n', & \mbox{with probability $1/2$,}\\
    \max(X_n,X_n'), & \mbox{with probability $1/2$}
    \end{array}\right.
\end{equation*}
for independent $X_n \sim X_n'$. Using also $\max(x,y) = \frac12(x + y + |x-y|)$, it follows
\begin{equation*}
    \E X_{n+1} = \E X_n + \frac12 \E \max(X_n^{(1)},X_n^{(2)}) = \frac32 \E X_n + \frac14 \E|X_n^{(1)} - X_n^{(2)}|.
\end{equation*}
Since $0 \le |X_n^{(1)} - X_n^{(2)}| \le N$, then 
\begin{align}\label{eq: ns exp ineq}
    \left(\frac32\right)^{n} &\le \frac32 \E X_{n-1} < \E X_n< \frac32 \E X_{n-1} + \frac18 N \le \frac12 N + \left(\frac32\right)^n - \left(\frac34\right)^{n-1}.
\end{align}
Note in particular then we have a lower bound on $\E X_n$ using the exact first moment from the corresponding simple case, $(\frac32)^n = N^{\alpha_2}$. {As mentioned in \Cref{rmk: ELIS > sqrt n}, we thus then also have nonsimple butterfly permutations lie outside the $O(N^{1/2})$ expected LIS class.}

Using \eqref{eq: ns exp ineq}, we though still have a linear upper bound on $\E X_n$, matching the trivially upper bound scaling since $X_n \le N$. This smaller linear upper bound resulted from the crudest trivial bound on $|X_n^{(1)} - X_n^{(2)}|$. We can reduce this upper bound to be sublinear, as we will now do.

\begin{proposition}\label{prop: bin ns 0}
    Let $\sigma_n \sim \Unif(B_N)$. Then $N^{\alpha} \le \E L(\sigma_n) \le N^{\beta}$ for constants $\frac12 < \alpha < \beta < 1$.
\end{proposition}


\begin{proof}
By above, we can take $\alpha = \alpha_2 \approx 0.585 > \frac12$ for the lower bound. So now we will focus on establishing the sublinear upper bound. Using Cauchy-Schwarz, we have the bounds
\begin{eqnarray*}
    \mathbb{E} X_{n+1} =  \frac{3}{2}\mathbb{E}X_n+\frac{1}{4}\mathbb{E}|X_n^{(1)}-X_n^{(2)}|\leq\frac{3}{2}\mathbb{E}X_n+\frac{\sqrt{2}}{4}\sqrt{\mathrm{Var}(X_n)},
\end{eqnarray*}
and also
\begin{eqnarray*}
    \mathrm{Var}(X_{n+1})\leq \frac{3}{2}\mathrm{Var}(X_n)+\frac{1}{4}\mathbb{E}X_n^2+\frac{\sqrt{2}}{2}\sqrt{\mathrm{Var}(X_n)(\mathrm{Var}(X_n)+(\mathbb{E}X_n)^2)}.
\end{eqnarray*}
Note that $\mathbb{E}X_1=1$ and $\mathrm{Var}(X_1)=0$. Define $a_n,b_n$ such that $a_1=1$, $b_1=0$, and for any $n\geq 1$,
\begin{equation*}
    a_{n+1}=\frac{3}{2}a_n+\frac{\sqrt{2}}{4}\sqrt{b_n},
\end{equation*}
\begin{equation*}
    b_{n+1}=\frac{3}{2}b_n+\frac{1}{4}a_n^2+\frac{\sqrt{2}}{2}\sqrt{b_n(b_n+a_n^2)}.
\end{equation*}
By induction, we have $\mathbb{E}X_n\leq a_n$, $\mathrm{Var}(X_n)\leq b_n$. Now we show by induction that $b_n\leq a_n^2$: when $n=1$, this holds; for any $n\geq 1$, assuming that $b_n\leq a_n^2$, we have
\begin{equation*}
    b_{n+1}\leq \frac{3+\sqrt{2}}{2}b_n+\frac{1}{4}a_n^2+\frac{\sqrt{2}}{2}\sqrt{b_n} a_n\leq \frac{9}{4}a_n^2+\frac{3\sqrt{2}}{4}\sqrt{b_n} a_n+\frac{1}{8}b_n =a_{n+1}^2. 
\end{equation*}
Hence $a_{n+1}\leq \left(\frac32 + \frac1{2^{3/2}}\right)a_n$, and so
\begin{equation*}
    \E X_n \le a_n \le \left(\frac32 + \frac1{2^{3/2}}\right)^n = N^{\log_2(\frac32 + \frac1{2^{3/2}})} = N^{\log_2(6+\sqrt 2) - 2} \approx N^{0.89029}. 
\end{equation*}
Now take $\beta = \log_2(6 + \sqrt 2) - 2 < 1$.
\end{proof}

This upper bound exponent $\beta_2 := \beta$ can be sharpened a bit more, as follows. Utilizing the proof of \Cref{prop: bin ns 0}, we can define a sequence $c_k$ such that $b_n \le c_k a_n^2$ for $c_{k+1} = h(c_k)$ and $c_0 = 1$, with
\begin{equation*}
    h(x) = \frac19 + \frac{107x-6\sqrt{2x}}{9(18 + 6\sqrt{2x} + x)} + 2\sqrt{2} x \sqrt{\frac4{(18+6\sqrt{2x} + x)^2} + \frac1{9(3+\sqrt{2x})^2}}.
\end{equation*}
This follows from using the iterated bounds
\begin{align*}
    b_n  &= \frac{8 c_k}{18 + 6 \sqrt{2c_k} + c_k} \left(\frac 94 \frac{b_n}{c_k} + \frac{3 \sqrt 2}4 \sqrt{b_n \cdot \frac{b_n}{c_k}} + \frac18 b_n\right) \\
    &\le \frac{8 c_k}{18 + 6 \sqrt{2c_k} + c_k} a_{n+1}^2 
\end{align*}
and similarly
\begin{equation*}
    \frac1{\sqrt{c_k}}b_n \le \sqrt{b_n} a_n \le \frac{4 \sqrt{c_k}}{9 + 3\sqrt{2c_k}} a_{n+1}^2
\end{equation*}
inductively applied to
\begin{equation*}
b_{n+1} = \frac19 a_{n+1}^2 + \frac{107}{72}b_n - \frac{\sqrt{2}}{12}\sqrt{b_n}a_n + \sqrt{b_n^2 + b_n a_n} \le h(c_k)a_{n+1}^2.    
\end{equation*}
The above argument yields the starting point $c_0 = 1$, with then $c_1 = h(c_0)  \approx 0.80597$ and $c_2 = h(c_1) \approx 0.71783$. In particular, we have the bounds 
\begin{equation}
a_n \le \left(\frac32 + \frac{\sqrt{c_k}}{2^{3/2}}\right)^n =  N^{\log_2(6+\sqrt{2c_k}) - 2}     
\end{equation}
for each $k$. Moreover, since a straightforward verification shows $|h'(x)| \le 0.9$ for all $x \in [0,1]$, then $h(x)$ is a contraction map on the interval $[0,1]$. This yields the sequence $c_k$ then necessarily converges to the fixed point $h(c^*) = c^* \approx 0.63092$ inside $[0,1]$ by the Banach fixed-point theorem. Hence,
\begin{equation}
    \E X_n \le a_n \le \lim_{k\to\infty} N^{\log_2(6+\sqrt{2c_k})-2}=N^{\log_2(6+\sqrt{2c^*})-2} \approx N^{0.83255}.
\end{equation}
This can be summarized as:
\begin{proposition}\label{prop: sublinear bd}
    If $\sigma_n \sim \Unif(B_N)$, then $\frac1n \log_2\E L(\sigma_n) \in (\alpha_2,\beta_2^*)$ for $\alpha_2 = \log_2(3/2) \approx 0.58496$ and $\beta_2^* = \log_2(6+ \sqrt{2 c^*}) - 2 \approx 0.83255$.
\end{proposition}
Future work can explore methods to improve both the lower and upper bounds. This is still far from the numerically computed value of $\hat \alpha \approx 0.6784$ that we find in \Cref{sec: lis experiments}, although it is comforting that this lies inside the interval from \Cref{prop: sublinear bd}. 

\begin{remark}\label{rmk: Var scaling}
    The above analysis yields  $\sqrt{\Var L(\sigma)} \le c^* a_n$, but does not determine, say, whether $\E L(\sigma)$ and $\sqrt{\Var L(\sigma)}$ are also of the same order,  nor even if $\Var L(\sigma)$ has a (non-constant) polynomial lower bound.
\end{remark}


Although the exact asymptotics for $\E X_n$ may be elusive, we can still compute exact values for fixed $n$. For $k = 1,2,\ldots,2^n$, let 
\begin{equation*}
    b(n,k) = \#\{\sigma \in B_N: L(\sigma) = k\},
\end{equation*}
so that $b(n,k)$ for $k = 1,\ldots,2^n$ comprise a partition of $|B_N| = 2^{2^n-1}$, where $b(n,1) = b(n,2^n) = 1$. Since $b(n,k) = |B_N|\cdot\P(L(\sigma_n) = k)$ for $\sigma_n \sim \Unif(B_N)$, then
\begin{equation}\label{eq: ELIS ns}
    \E L(\sigma) = 2^{1-2^n}\sum_{k=1}^{2^n} k b(n,k).
\end{equation}
Using \Cref{prop: lis 2-Xn}, $b(n,k)$ satisfies the recursion
\begin{align}
    b(n+1,k) &=
    \sum_{j=1}^{k-1} b(n,j)b(n,k-j) 
    + \left(b(n,k)^2 + 2 \sum_{j=1}^{k-1} b(n,k)b(n,j)\right)\nonumber\\
    &= b(n,k)^2 + \sum_{j=1}^{k-1} b(n,j)\left[b(n,k-j) + 2 b(n,k)\right].\label{eq: ELIS ns rec}
\end{align}
From this, quick results include 
\begin{align*}
    b(n+1,2) = (b(n,2)+1)^2 \quad \mbox{and} \quad b(n,2^n-1) = 2^{n-1}.
\end{align*}
\Cref{t:bnk} shows the first few computed values of $b(n,k)$ for small $n$. Note $b(4,4) > b(4,6) > b(4,5)$, while in general for all other computed $n$ we considered, we see $b(n,k)$ is unimodal as in it is monotonically increasing and then decreasing about the mode $n^* = \argmax_k b(n,k)$, 

\begin{table}[t]
\centering
{
\begin{tabular}{r|ccccccccc}
&\multicolumn{8}{c}{$k$}\\
$n$ & 1 & 2 & 3 & 4 & 5 & 6 &7& 8 & \ldots \\ \hline 
1 & 1& 1\\
2 & 1 & 4 &  2 & 1\\
3 & 1 & 25 & 32 & 35&  18&  12& 4& 1\\
4 & 1 & 676 & 2738&  5974&  5342&  5618&  4164&  3240 &\ldots
\end{tabular}
}
\caption{Triangle of $b(n,k)$ values for $n = 1$ to 4.}
\label{t:bnk}
\end{table}
The recursion \eqref{eq: ELIS ns rec} is used to compute exact values of $\E L(\sigma)$ for $\sigma \sim \Unif(B_N)$ for $N = 2^n$ with $n \le 15$ in \Cref{sec: lis experiments}. Similarly, the cdfs $F_n(t) = \P(X_n \le t)$ then also satisfy a recursion, where
\begin{equation}
    F_{n+1}(t) = \frac12 F_n(t)^2 + \sum_{j = 1}^{2^n} F_n(t-j)(F_n(j) - F_n(j-1)).
\end{equation}
These formulas can be utilized in future work to compute other statistics of $\sigma \sim \Unif(B_N)$, such as the mode or median, as well as potentially to sharpen the bounds in \Cref{prop: sublinear bd}.

Moreover, it would be interesting to compare the behavior of this particular 2-Sylow subgroup of $S_{2^n}$ to another. In general, $S_{p^n}$ has $p^n!/[p^{(p^n-1)/(p-1)}(p-1)^n]$ distinct $p$-Sylow subgroups (see, e.g., \cite{Kaloujnine_1948}). Taking $p = 2$ and $n = 10$, this means there are about $6.03\cdot 10^{2331}$ 2-Sylow subgroups of $S_{1024}$, each of which are conjugate to $B_{N}$. $L(\sigma)$ is not invariant under conjugation, as is evident by just comparing the transpositions $(1 \ 3)$ and $(1 \ 2)$ in $S_3$. As is explored in general below, considering a $\E L(\sigma)$ where $\sigma \sim \Unif( \langle (a_1 \ a_2 \ \cdots \ a_m)\rangle)$, then $\E L(\sigma) \le \frac1m \sum_{j =1}^m \max(j,m-j)$ for $X \sim \Unif([m])$. So among the simple $m$-nary butterfly permutations with the explicit recursive block structure, then $B_{s,n}^{(m)}$ will maximize $\E L(\sigma_n)$. Future work can explore how this behaves with regard to conjugation of $B_{s,n}^{(m)}$ or $B_n^{(m)}$ by general $\sigma \in S_{2^m}$ rather than just by $\sigma \in S_{2^m}$ that also preserve the group block structure of each group.

\subsubsection{Numerical experiments: Power-law exponent estimate and variance scaling}\label{sec: lis experiments}

\begin{figure}[t]
    \centering
    \subfloat[Sample means to estimate $\E  {L(\sigma)}$]{
    \includegraphics[width=.9\linewidth]{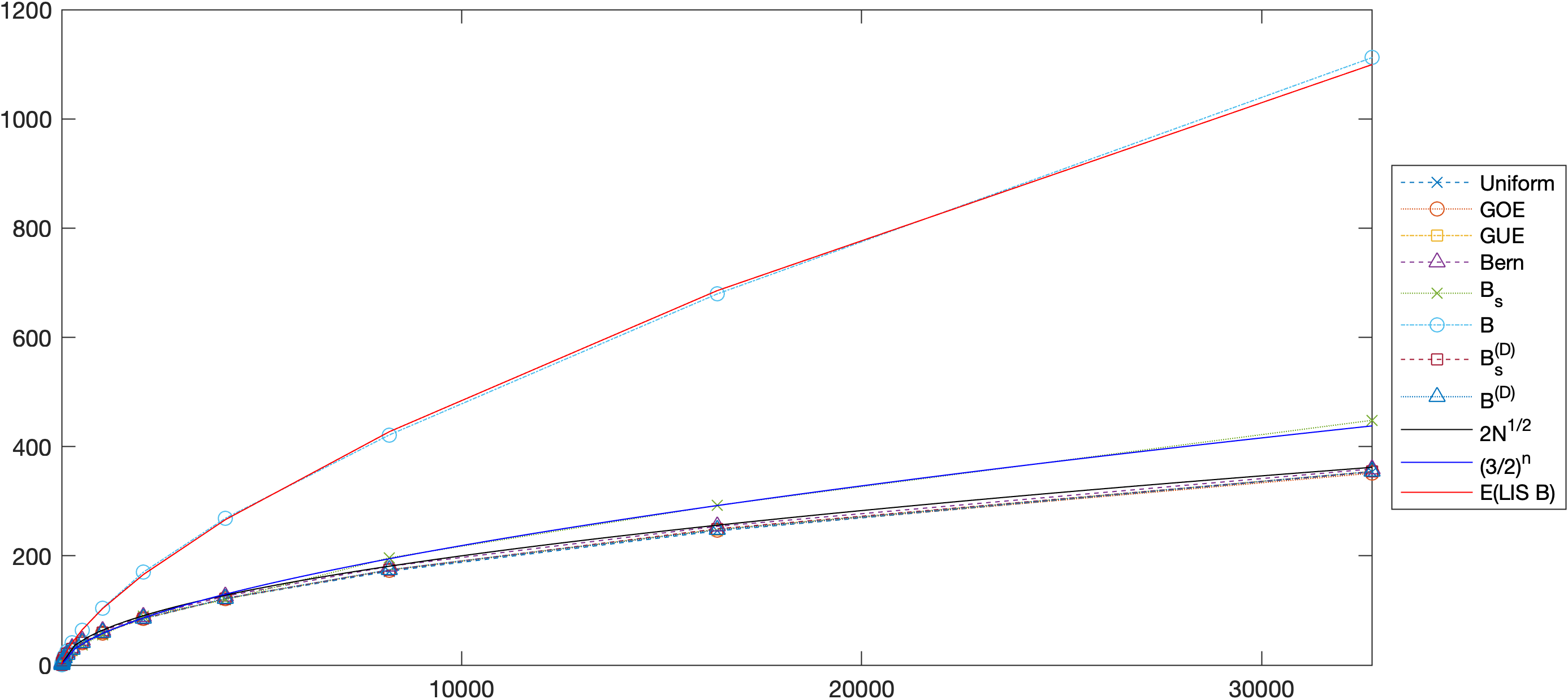}
    }\\
    \subfloat[log-log plot]{
    \includegraphics[width=.9\linewidth]{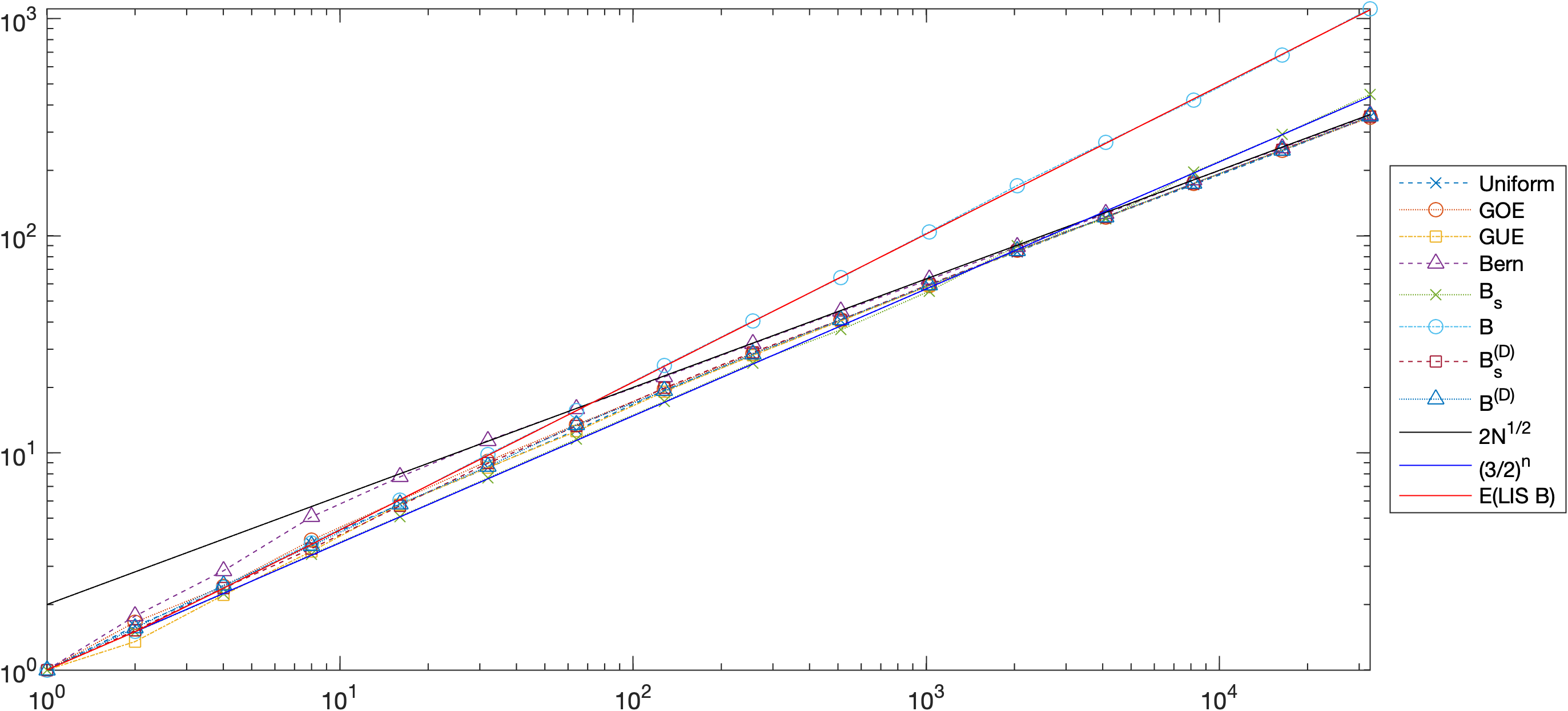}
    \label{fig: log log}}
    \caption{Standard and log-log plot of sample means to estimate $\E L(\sigma)$ for random permutations using GEPP, along with comparisons of $2 N^{1/2}$ for uniform permutations, $(3/2)^n$ for simple butterfly permutations, and exact computations of $\E L(\sigma)$ for nonsimple butterfly permutations, for $N = 2^n$ with $n = 1,2,\ldots,15$.}
    \label{fig:LIS experiments}
\end{figure}

\Cref{prop: sublinear bd} establishes polynomial bounds on $\E L(\sigma)$ for $\sigma \sim \Unif(B_N)$. In this section, we compute sample statistics to estimate the actual polynomial exponent for the expected LIS for the nonsimple butterfly permutations. For comparison, we include other random permutations induced using GEPP on the following random ensembles:
\begin{multicols}{2}
\begin{itemize}
    \item $\B_s(N,\Sigma_S)$
    \item $\B(N,\Sigma_S)$
    \item $\B_s(N,\Sigma_D)$
    \item $\B(N,\Sigma_D)$
    \item $\Unif(S_N)$
    \item $\GOE(N)$
    \item $\GUE(N)$
    \item $\Bern(\frac12)^{N\times N}$
\end{itemize}
\end{multicols}
\noindent For each of the simple scalar butterfly models as well as for the uniform permutation model, we sample directly using the native MATLAB function \texttt{randperm} for $\Unif(S_N)$ along with custom recursive MATLAB functions to match $\Unif(B_{s,N})$ and $\Unif(B_N)$ directly using the constructions from \Cref{sec: simple butterfly,sec: ns butterfly}. The uniform permutations are mainly included as a target comparison for the induced random permutations other than the scalar butterfly matrices.

For comparison, we compute sample statistics for the above other random ensembles, where we sample random permutations using the built-in \texttt{lu} MATLAB function along with standard sampling techniques for each random matrix ensemble (e.g., see \cite{Mezz}). We then compute the LIS directly for each permutation by adapting the standard RSK algorithm to retain only the length of the first row of the associated Young tableau. We use 100 samples for $N = 2^n$ where $n = 2,\ldots,12$ and only 10 samples for $n = 13,14,15$ for each of the non-scalar butterfly models. In addition, we run 1000 samples each directly from $\Unif(B_{s,N})$, $\Unif(B_N)$, and $\Unif(S_N)$ for each of the above $N$ values, as these are sampled without directly calling \texttt{lu}. For comparison, we include lines that correspond to  $2N^{1/2}$, which asymptotically matches $\E L(\sigma)$ for $\sigma \sim \Unif(S_N)$), and $(\frac32)^n$, which exactly matches $\E L(\sigma)$ for $\sigma \sim \Unif(B_{s,N})$; additionally, we also include the exact values $\E L(\sigma_n)$ for $\sigma_n \sim \Unif(B_N)$, computed using \eqref{eq: ELIS ns} along with the recursion formula \eqref{eq: ELIS ns rec}. \Cref{fig:LIS experiments} shows the output for these experiments.

Of note, each of the non-scalar butterfly models had sample means LIS essentially align with the asymptotic uniform permutation behavior controlled by $2N^{1/2}$ as $N$ increases. In fact, both presentations in \Cref{fig:LIS experiments} have non-scalar butterfly sample means essentially indistinguishable from one another for $N \ge 2^{12}$. The Bernoulli matrices had larger LIS than the other sampled matrices until $N = 128$, when the LIS for $B_N$ first eclipsed the corresponding value of $2N^{1/2}$, resp.,  25.0587 versus 22.6274; this is seen in \Cref{fig: log log} as the first point after the red and black lines intersect starting near $10^2$. After this point, the LIS sample means for $B_N$ stayed above those for each of the other ensembles. Similarly, starting at $N = 2^{12} = 4096$, the LIS sample means for $B_{s,N}$ now also eclipse the remaining random ensembles, which aligns with the fact $(\frac32)^n > 2N^{1/2}$ for $N > N_0(2) = \lceil 2^{\log 4/\log(9/8)}\rceil = 3493$ (e.g., $(\frac32)^{12} \approx 129.7463 > 2\sqrt{4096} = 128$). This can also be seen in \Cref{fig: log log} as these occur after the intersection of the blue and black lines between $10^3$ and $10^4$. For $m  > 2$, then $N^{\alpha_m} > 2 N^{1/2}$ for much smaller values of $N_0(m)$ (compared to 3493 for $m = 2$), with $N_0(3) = 13$ (cf. \Cref{fig:s tern permuton lis}) and $N_0(5) = 9$; in general, $N_0(m) = 1 + o_m(1)$.



As is clearly evident in particular in \Cref{fig: log log}, we have a very strong linear fit already for $\E L(\sigma)$ when $\sigma \sim \Unif(B_N)$. To estimate this power-law exponent $\hat \alpha$, we performed a linear regression on the log-log plot, using the exact values of the expected LIS for $N = 2^n$ for $n = 3,\ldots,15$. The fitted model yielded an exponent $\hat \alpha = 0.681831042171098$ (with a negligible constant coefficient $-0.086736463938470$), suggesting that $\E L(\sigma_n) \approx N^{\hat \alpha}$. The goodness of fit is very strongly supported by an $R^2$ value of 
0.999999999998797 for the log-log linear fit (with a computed $p$-value of 0 in double precision), as well as a value of 0.999999939381698 for the original data against the predicted model.

\begin{table}[t]
\centering
\begin{tabular}{c|cc|c}
\( n \) & \( \mathbb{E}[X_n] \) & \( \mathbb{E}[X_n^2] \) & \({{\mathbb{E}[X_n^2]}}/{\mathbb{E}[X_n]^2} \) \\
\hline
1  & 1.5      & 2.5        & 1.1111111  \\
2  & 2.375    & 6.375      & 1.1301939  \\
3  & 3.78906  & 16.3359    & 1.1378382  \\
4  & 6.07187  & 41.9741    & 1.1385091  \\
5  & 9.73715  & 107.955    & 1.1386162  \\
6  & 15.6201  & 277.762    & 1.1384322  \\
7  & 25.0588  & 714.792    & 1.1383076  \\
8  & 40.2016  & 1839.57    & 1.1382278  \\
9  & 64.495   & 4734.39    & 1.1381830  \\
10 & 103.468  & 12184.8    & 1.1381592  \\
11 & 165.992  & 31359.9    & 1.1381469  \\
12 & 266.298  & 80710.8    & 1.1381408  \\
13 & 427.216  & 207725.2    & 1.1381377  \\
14 & 685.372  & 534622.1    & 1.1381362  \\
15 & 1099.526  & 1375956.6 & 1.1381355  \\
\end{tabular}
\caption{Values of \( \mathbb{E}[X_n] \), \( \mathbb{E}[X_n^2] \), and \( {{\mathbb{E}[X_n^2]}}/{\mathbb{E}[X_n]^2} \) for different values of \( n \).}
\label{t: var scaling}
\end{table}

A similar approach can address the question of the scaling of $\E[L(\sigma_n)^2]$ and $\operatorname{Var} L(\sigma_n)$, as mentioned in \Cref{rmk: Var scaling}. In particular, we are interested in how this compares to the scaling needed for $\E L(\sigma_n)$. Using the same setup for the previous experiments to compute exact values of $\E L(\sigma_n)$ for $n \le 15$ we also can use $b(n,k)$ to yield exact values of $\E[L(\sigma_n)^2]$. We then compute the ratio ${E[L(\sigma_n)^2]}/\E[ L(\sigma_n)]^2$ for each value. \Cref{t: var scaling} shows the output for the values. This empirically suggests the second moment then would have a scaling on the same order as the square of the first moment, with a leading constant $c \in [1,1.14]$  (and so the standard deviation and first moment would each have scaling of the same order). Future work can further explore  properties for this scaling.


\subsubsection{\texorpdfstring{{\boldmath$m$}}{n}-nary butterfly permutations}



Now moving to the LIS question for the general $m$ case, a lot of the same approaches used in the binary case carry over directly. In terms of our main focus, we can similarly now derive nontrivial sublinear bounds on $\E L(\sigma_n)$ for $\sigma_n \sim \Unif(B_n^{(m)})$.

\begin{theorem}\label{thm: power law}
    Let $\sigma_n \sim \Unif(B_{n}^{(m)})$. Then there exist constants $\alpha_m,\beta_m$ depending only on $m$ such that $\frac12 < \alpha_m < \beta_m < 1$ and $N^{\alpha_m} \le \E L(\sigma_n) \le N^{\beta_m}$, while $\alpha_m = 1 - o_m(1)$.
\end{theorem}


\begin{proof}
Note first $X_0 = 1$ and
\begin{align*}
    \E X_{n+1} = \E X_n + \frac1m \sum_{j=1}^{m-1} \E \max\left(\sum_{i = 1}^j X^{(i)}_n,\sum_{i = j+1}^m X^{(i)}_n\right).
\end{align*}
We first establish a lower bound on $\E X_n$, since now
\begin{align*}
    \E X_{n+1} &\ge \E X_n + \frac1m \sum_{j=1}^{m-1} \max\left(\sum_{i = 1}^j \E X^{(i)}_n,\sum_{i = j+1}^m \E X^{(i)}_n\right)\\
    &= \left(1 + \frac1m \sum_{j = 1}^{m - 1} \max(j,m-j)\right) \E X_n\\
    &= \left(\frac{3m^2 + r_m}{4m} \right) \E X_n
\end{align*}
where $r_m = m \pmod 2$. This now yields
\begin{equation}
    \E X_n \ge \left(\frac{3m^2 + r_m}{4m}\right)^n = N^{\alpha_m},
\end{equation}
showing the exact first moment in the simple butterfly permutation case (cf. \Cref{t: s lis}) serves as a lower bound for the corresponding nonsimple first moment. We already established then $\alpha_m$ is increasing by \Cref{p: am incr}, so in particular, $\frac12 < \log_2(3/2) = \alpha_2 \le \alpha_m$ for all $m \ge 2$.

Next, we want to establish a sub-linear upper bound. We already established this holds if $m = 2$ (cf. \Cref{prop: bin ns 0}), so we only need to consider $m \ge 3$. This will use the following lemma:
\begin{lemma}\label{l: var bd m}
    $\Var X_n \le (\E X_n)^2$ for $m \ge 3$.
\end{lemma}
\begin{proof}
    We will again show this by induction on $n$ for $m$ fixed. This holds for $n = 0$ since $\Var X_0 = 0 < \E X_0 = 1$. Now assume $\Var X_n \le (\E X_n)^2$, and equivalently $\E X_n^2 \le 2 (\E X_n)^2$. It suffices now to show $\E X_{n+1}^2 \le 2 (\E X_{n+1}^2)$. 
    
    First note
\begin{align*}
    \E\left( \sum_{i = 1}^j X_n^{(i)} - \sum_{i = j+1}^m X_n^{(i)}\right)^2 
    &= m \E X_n^2 + 2\left( \binom{j}2 + \binom{m-j}2 - j(m-j)\right)(\E X_n)^2\\
    &= m \E X_n^2 + ((m-2j)^2 - m) (\E X_n)^2\\
    &= m \Var X_n + (m-2j)^2 (\E X_n)^2.
\end{align*}
We now compute
\begin{align*}
    \E X_{n+1}^2 &= \frac1m \E\left( \sum_{j=1}^m X_n^{(j)}\right)^2 + \frac1m \sum_{j = 1}^{m-1} \E \max\left( \sum_{i = 1}^j X_n^{(i)},\sum_{i = j+1}^m X_n^{(i)}\right)^2\\
    &=\E X_n^2 + (m-1) (\E X_n)^2\\
    &\hspace{2pc} + \frac1{2m} \sum_{j=1}^{m - 1} \left( m \E X_n^2 + 2\left( \binom{j}2 + \binom{m-j}2\right)(\E X_n)^2\right.\\
    &\hspace{2pc} \left. + \sum_{\ell = 1}^m \E\left[X_n^{(\ell)} \cdot \left| \sum_{i =  1}^j X_n^{(i)} - \sum_{i = j+1}^{m} X_n^{(i)} \right|\right] \right)\\
    & \le \frac12(m+1) \E X_n^2 + \frac13(m^2-1)(\E X_n)^2\\
    &\hspace{2pc}+\frac1{2} \sum_{j=1}^{m - 1} \left(\E X_n^2 \cdot \left(m \Var X_n + (m - 2j)^2(\E X_n)^2\right)\right)^{1/2}.
\end{align*}
By the inductive hypothesis, $\Var X_n \le (\E X_n)^2$ and $\E X_n^2 \le 2 (\E X_n)^2$, so
\begin{align*}
    \E X_{n+1}^2 &\le G(m) \cdot (\E X_{n+1})^2
\end{align*}
for
\begin{equation*}
    G(m) = \left(\frac13(m+1)(m + 2)+\frac1{\sqrt 2} \sum_{j = 1}^{m - 1}  (m + (m-2j)^2)^{1/2}\right) \cdot \left(\frac{4m}{3m^2 + r}\right)^2.
\end{equation*}
Now we compute $G(3) = \frac6{49}(10 + 3\sqrt 2) \approx 1.743997$, $G(4) = \frac19(14 + \sqrt 2) \approx 1.71269$, and $G(5) = \frac{50}{361}(7 + \sqrt 7 + \sqrt 3) \approx 1.57587$. (However, $G(2) = \frac{20}9 > 2$, but we already have the $m = 2$ case handled; furthermore, the computations in \Cref{sec: lis experiments} suggest $\E X_n^2 \le 1.2 (\E X_n)^2$ for $m = 2$.) Note since $m > \frac43$ (say), then
\begin{equation}\label{eq: sub m(m-1)}
    m + (m-2j)^2 \le m + (m-2)^2 = m^2 -(3m-4) < m^2 
\end{equation}
for $j = 1,\ldots,m-1$. Hence, for $m \ge 6$, then
\begin{align*}
    G(m) &\le \left(\frac13(m+1)(m+2) + \frac1{\sqrt 2}m(m-1)\right) \cdot  \left(\frac{4}{3m}\right)^2\\
    &= \frac{16}9 \left(\frac23 \cdot \frac1{m^2} + \left(1 - \frac1{\sqrt 2}\right) \cdot \frac1m + \left(\frac13 + \frac1{\sqrt2}\right) \right)\\
    &\le \frac{16}9 \left( \frac23 \cdot \frac1{6^2} + \left(1 - \frac1{\sqrt 2}\right) \cdot \frac16 + \left(\frac13 + \frac1{\sqrt 2}\right)\right) \approx 1.9693763.
\end{align*}
It follows then $\E X^2_{n+1} \le G(m) (\E X_{n+1})^2 \le 2(\E X_{n+1})^2$ also for $m \ge 3$, as desired.
\end{proof}

We now evaluate
\begin{align*}
    \E X_{n+1} &= \E X_n + \frac1m \sum_{j=1}^{m-1} \E \max\left(\sum_{i = 1}^j X^{(i)}_n,\sum_{i = j+1}^m X^{(i)}_n\right)\\
    &= \E X_n + \frac1m \sum_{j = 1}^{m-1} \frac12 \E \left(\sum_{j=1}^m X_n^{(j)} + \left| \sum_{i = 1}^j X_n^{(i)} - \sum_{i = j+1}^m X_n^{(i)}\right| \right)\\
    &= \left(1 + \frac{m-1}2\right) \E X_n + \frac1{2m} \sum_{j=1}^{m-1} \E \left| \sum_{i = 1}^j X_n^{(i)} - \sum_{i = j+1}^m X_n^{(i)}\right|\\
    &\le \frac12(m+1)\E X_n + \frac1{2m} \sum_{j = 1}^{m-1} \left(m \Var X_n + (m-2j)^2 (\E X_n)^2\right)^{1/2}\\
    &\le \left(\frac12(m+1) + \frac1{2m} \sum_{j = 1}^{m-1}(m + (m-2j)^2)^{1/2} \right) \E X_n,
\end{align*}
using Cauchy-Schwarz and then \Cref{l: var bd m} for the final two lines. This  establishes $\E X_n \le  m^{n \beta_m} =  N^{\beta_m}$, where
\begin{align*}
    m^{\beta_m} &:= \frac12(m+1) + \frac1{2m} \sum_{j = 1}^{m-1}(m + (m-2j)^2)^{1/2} \\
    &< \frac12(m+1) + \frac1{2m} (m-1) m = m, 
\end{align*}
using \eqref{eq: sub m(m-1)}, yielding   $\frac12  < \alpha_m < \beta_m < 1$.
\end{proof}

\begin{figure}
    \centering
    \includegraphics[width=1\linewidth]{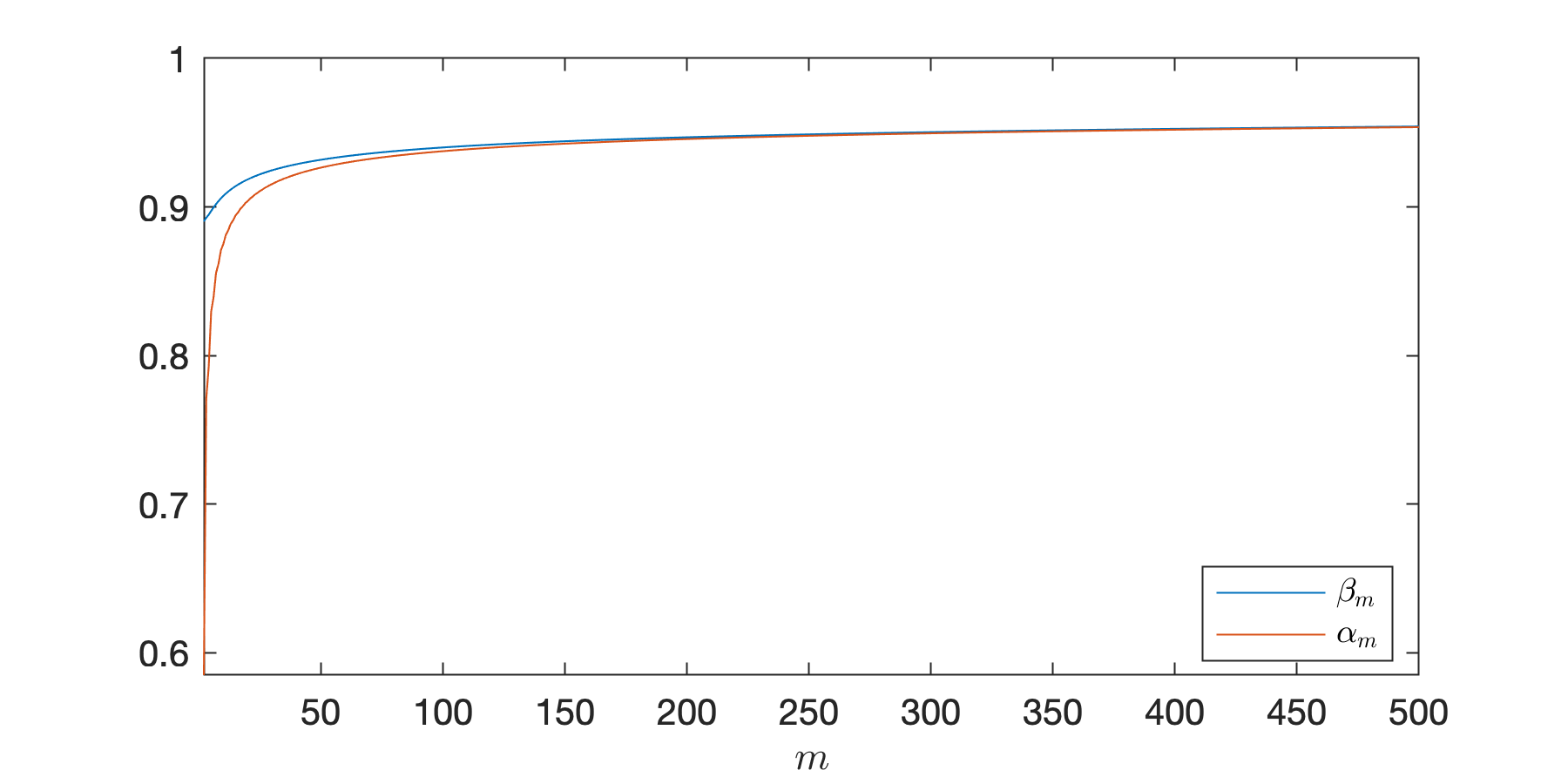}
    \caption{Comparing $\alpha_m$ to $\beta_m$.}
    \label{fig:LIS bds}
\end{figure}

\Cref{fig:LIS bds} shows a plot of $\alpha_m,\beta_m$ for $m \le 500$. 
\begin{remark}
    Note $\alpha_m = 1 - O(\frac1{\ln m})$ by \eqref{eq: asympt am}, which yields then also $\beta_m = 1 - O(\frac1{\ln m})$ and so $\beta_m - \alpha_m = O(\frac1{\ln m})$. Hence, the bounds for $\E X_n$ significantly sharpen as $m$ increases, so that both $\alpha_m$ and $\beta_m$ (and their average value) can serve as statistical estimators for the exponent for $\E X_n$ for $m$ sufficiently large.
\end{remark} 
\begin{remark}
    A more involved argument could yield also $\beta_m$ is strictly increasing, while an iterative fixed-point argument similar to that used in the binary case can be used to sharpen $\beta_m$ for each $m$.
\end{remark}

\Cref{t:LIS bds} shows specific examples comparing $\alpha_p$ and $\beta_p$ for some select prime numbers $p$, including the largest prime numbers for certain fixed numbers of digits. In particular, we note how much faster $\alpha_m$ and $\beta_m$ align compared to how $\alpha_m$ (and hence also $\beta_m$) converge to 1; both have convergence rates $O(\frac1{\ln m})$, with thus notably different leading constants. This table further suggests $\alpha_m$ and $\beta_m$ typically match the first $\log_{10} m$ digits (i.e., the number of digits of $m$ itself, in line with the $O(\frac1{\ln m})$ convergence rate). For an additional comparison point, using the currently largest known prime number $p = 2^{136279841}-1$ (with 41,024,320 digits) \cite{GIMPS2024}, 
then $\alpha_p  \approx 0.99999999695452$ that should thus match $\beta_p$ to over 41 million digits.

\begin{table}[ht]
\centering
\begin{tabular}{|c|c|c|l||c|c|c|l|}
\hline
$p$ & $\alpha_p$ & $\beta_p$ &\!$\beta_p - \alpha_p$\!& $p$ & $\alpha_p$ & $\beta_p$ &\!$\beta_p - \alpha_p$\!\\ 
\hline
\!2 \!  & 0.58496   &0.89029 &\  0.305\!  & \!97 \!  & {0.93}712   & {0.93}971 & $2.6 \cdot 10^{-3}$\!  \\ 
\!3  \! & 0.77124   & 0.89279 &\  0.122 \! &\! 997  \! & {0.958}33   & {0.958}56 & $2.3 \cdot 10^{-4}$\!  \\ 
\!5 \!  & 0.82948   & 0.89730 &\  0.068\!  & \!9971 \!  & {0.9687}5   & {0.9687}8 &$2.1 \cdot 10^{-5}$\!  \\ 
\!7 \!  & 0.85564   & 0.90191 &\  0.046 \! & \!99991 \!  & {0.97501}   & {0.97501} & $2.0 \cdot 10^{-6}$\!  \\ 
\!11 \!  & 0.88117   & 0.90898 &\  0.028\!  & \!999983\!  & {0.97917} & {0.97917} & $2.0 \cdot 10^{-7}$\! \\ 
\hline
\end{tabular}
\caption{Values of $\alpha_p$ and $\beta_p$ for particular primes $p$}
\label{t:LIS bds}
\end{table}

    Again, it would be interesting to consider the LIS question with other $p$-Sylow subgroups of $S_{p^n}$. For other $p$-Sylow subgroups generated using an analogous block structure to $B_n^{(p)}$, we would conjecture $B_n^{(p)}$ \textit{maximizes} the LIS in expectation. For example, for $p = 5$, consider using instead the generating $5$-cycle $\tilde \tau_5 = (1 \ 3 \ 5 \ 4 \ 2)$ (in cyclic notation) to similarly construct the $5$-Sylow subgroup of $S_{5^n}$ of the form $C_5 \wr \cdots \wr C_5$ for $C_5 \cong \langle \tilde \tau_5\rangle$; for $\sigma \sim \Unif(\langle \tilde \tau_5\rangle)$, then $\E L(\sigma) = \frac25 \cdot 2 + \frac25 \cdot 3 + \frac15 \cdot 1  = \frac{15}5 = 3  <  5^{\alpha_5} = \frac{3 \cdot 5^2 + 1}{4 \cdot 5} = 3.8$. For all other subgroups of $S_5$ isomorphic to $C_5$ (of the $4! = 24$ distinct 5-cycles, these break up into 6 distinct groups, with example generators $\tau_5 = (1 \ 2 \ 3 \ 4 \ 5)$ as well as $(1 \ 2 \ 3 \ 5 \ 4), (1 \ 2 \ 4 \ 3 \ 5), (1 \ 2 \ 4 \ 5 \ 3), (1 \ 2 \ 5 \ 3 \ 4), (1 \ 2 \ 5 \ 4 \ 3)$), the expected LIS for each uniform subgroup matches the $\tilde \tau_5$ statistics of $\E L(\sigma) = 3$ with the exception of the circulant $5$-cycle group $\tau_5$, that had $\E L(\sigma) = 3.8$. Additionally, future work can consider the LIS questions for conjugations of $B_n^{(p)}$ that do not preserve the block structure of $B_n^{(p)}$.


\section{Number of  cycles}\label{sec: cycles}

We now shift our focus to the number of cycles. Recall the number of cycles of a permutation $\sigma \in S_n$, $C(\sigma)$, gives the number of cycles in the disjoint cycle decomposition of $\sigma$. Equivalently, this returns the number of $\sigma$-orbits in $[n]$. As noted in \cite{P24}, this relates directly to $\Pi(A)$,  the number of GEPP pivot movements needed on $A$, via the equality 
\begin{equation*}
    \Pi(A) = n - C(\sigma(A)).
\end{equation*} 
Unlike the LIS, the number of cycles for a permutation is invariant under conjugation. Hence, for $p$-Sylow subgroups, which are all conjugate to one another, it suffices to establish the number of cycles for a particular $p$-Sylow subgroup to then satisfy the question for all $p$-Sylow subgroups. This is the motivating idea underlying our current aim.

The number of cycles is a classic statistic of study for random permutations. 
Early work studied uniform permutations, where the problem of the number of cycles is well-understood (see standard textbooks, such as, e.g., \cite[Example 3.4.6]{Durrett_2010}). If $\sigma \sim \Unif(S_n)$, then $C(\sigma) \sim \Upsilon_n$ where $\P(\Upsilon_n = k) = \frac{|s(n,k)|}{n!}$ for $k = 1,2,\ldots,n$ and $s(n,k)$ denotes the Stirling numbers of the first kind. This is natural as the signed Stirling numbers of the first kind, $|s(n,k)|$, are often defined as the number of permutations in $S_n$ with $k$ cycles, so that 
\begin{equation*}
    |s(n,k)| = \#\{ \pi \in S_n: C(\pi) = k\} = n!\cdot \P(C(\sigma) = k)
\end{equation*}
for $\sigma \sim \Unif(S_n)$. Similarly, we have then also $\Pi(A) \sim n - \Upsilon_n$ when $\sigma(A) \sim \Unif(S_n)$. Moreover, Goncharov's CLT (see, e.g., \cite[§5.1.1]{Sachkov}) established the asymptotics for $\Upsilon_n$, showing this  satisfies Gaussian limiting behavior, where
\begin{equation}\label{eq: Goncharov}
    \lim_{n\to\infty} \P((\Upsilon_n - \log n)(\log n)^{-1/2} \le t) = \frac1{\sqrt{2 \pi}} \int_{-\infty}^t e^{-t^2/2} dt.
\end{equation}
The centering and scaling in \eqref{eq: Goncharov} use the standard results
\begin{align*}
    \E \Upsilon_n &= H_n^{(1)} = \log n + \gamma + o(1)\\
    \Var \Upsilon_n &= H_n^{(1)} - H_n^{(2)} = \log n + \gamma - \frac{\pi^2}6 + o(1)
\end{align*}
where $H_n^{(k)} = \sum_{j = 1}^n \frac1{j^k}$ denote the generalized Harmonic numbers and $\gamma \approx 0.57721$ is the Euler-Mascheroni constant.

Additionally, we can define $C_k(\sigma)$ that counts the number of cycles of $\sigma$ of length $k$ for a permutation $\sigma\in S_M$. Recall then $C(\sigma) = \sum_{j = 1}^{M} C_j(\sigma)$ while $M = \sum_{j=1}^{M} j C_j(\sigma)$, where we further note $C_1(\sigma)$ returns the number of fixed points of $\sigma$. For example, Goncharov in \cite{Goncharov} also showed that the $C_k(\sigma)$ for $\sigma \sim \Unif(S_n)$ jointly converge to independent Poisson random variables of mean $k^{-1}$. Moreover, the sorted cycle lengths in $\sigma$ (in descending order), normalized by $n$, converge in distribution to the Poisson-Dirichlet distribution. 




The cycle statistics have also been analyzed for non-uniform random permutation models. For a class of exponential families on permutations that include the Mallows permutation model, Mukherjee \cite{mukherjee} showed a Poisson distributional limit for the number of cycles of fixed length. For the Mallows permutation model with Kendall's $\tau$ and parameter $q\in (0,1)$, Gladkich and Peled in \cite{gladkich2018cycle} showed that when $(1-q)^{-2}\gg n$, the sorted and normalized cycle lengths (in descending order) converge in distribution to the Poisson-Dirichlet distribution (as in the uniform case); when $(1-q)^{-2}\ll n$, all cycles have size $o(n)$. The sorted and normalized cycle lengths of permutations drawn from the Ewens measure with parameter $\theta>0$ converge to the Poisson-Dirichlet distribution with parameter $\theta$ (see e.g. \cite{arratia2003logarithmic, Crane16, feng1, kingman1978representation}). Diaconis and Tung additionally studied cycle statistics for certain wreath products of the form $\Gamma^n \rtimes S_n$ for fixed $\Gamma \subset S_k$, including normal limiting distribution for the number of cycles when $k$ is fixed and $n$ grows \cite{Diaconis_Tung_2024}. Further results can be found in e.g. \cite{he2023cycles, hoffman2017pattern, slim2024small, zhong2023cycle}.


We will now focus on addressing the question of the number of cycles for butterfly permutations.

\subsection{Simple butterfly permutations}

The number of cycles question has already been answered in the case of uniform simple binary butterfly permutations. As part of \cite[Theorem 1]{P24},  the author established 
\begin{equation*}
    \Pi(B) \sim \frac{N}2 \Bern\left(1 - \frac1N\right) 
\end{equation*} 
for $B \sim \B_s(N,\Sigma_S)$. In particular, since then $\sigma(B) \sim \Unif(B_{s,N})$ via \Cref{t: simple b}, we have
\begin{equation*}
    C(\sigma) \sim \frac N2\left(1 + \Bern\left(\frac1N\right)\right)
\end{equation*}
for $\sigma \sim \Unif(B_{s,N})$; equivalently, $\log_2 C(\sigma) \sim n-1 + \Bern(\frac1N)$. The author established this by cataloging the permutations in $B_{s,N}$ and their cycle types. This can alternatively be derived directly from \eqref{eq: bperm rule}: for $\sigma \in \B_{s,N}$, unless $\sigma = 1$, then no points are fixed by $\sigma$ (since $k$ and $\sigma(k)$ have at least one differing binary coefficient) while $\sigma^2 = 1$, which then yields necessarily $\sigma$ is the disjoint product of $N/2$ transpositions; it follows $C(\sigma) = \frac N2$ whenever $\sigma \ne 1$ and $C(\sigma) = N$ when $\sigma = 1$. 

This can be rephrased in terms of now computing $C_k(\sigma)$, the number of cycles of length $k$ for a permutation $\sigma$. For $\sigma_n \in B_{s,N}$, then $C(\sigma_n) = C_1(\sigma_n) + 2 C_2(\sigma_n)$, with then $C_1(\sigma_n) = 2^n \mathds 1(\sigma_n = 1) \sim N \Bern(\frac1N)$ and hence $C_2(\sigma_n) = \frac{N}2 \mathds 1(\sigma_n \ne 1) \sim \frac{N}2 \Bern(1 - \frac1N)$ (using $C_1(\sigma_n) + 2 C_2(\sigma_n) = N$), while also $C_1(\sigma) C_2(\sigma) = 0$. Since $C_1(\sigma_n)$ converges to 0 in probability, this yields then $\frac{2}N C(\sigma_n)$ converges to 1 in probability. 

This method generalizes to the simple $p$-nary butterfly permutations, when $p$ is a prime: then $C(\sigma_n) = C_1(\sigma_n) + C_p(\sigma_n)$, where $C_1(\sigma_n) \sim p^n \Bern(p^{-n})$ (i.e., the number of fixed points, where now $\E C_1(\sigma_n) = 1$) and so $C_p(\sigma_n) \sim \frac1p\left( p^n - 1 + \Bern(1 - p^{-n})\right)$. Again, we have the number of fixed points of $\sigma_n$ converging to 0 in probability and hence $\frac{p}N C_p(\sigma_n)$ and so also $\frac{p}N C(\sigma_n)$ converging to 1 in probability.

\begin{theorem}\label{thm: s cycles LLN}
    If $\sigma_n \sim \Unif(B_{s,n}^{(p)})$ for $p$ prime, then $C(\sigma_n) \sim \frac{N}p(1+  (p-1)\Bern(\frac1N))$ and $\frac{p}N C(\sigma_n)$ converges in probability to 1. 
\end{theorem}
\noindent Again, this satisfies $\log_p C(\sigma) \sim n-1 + \Bern(\frac1N)$ in the prime $p$ case. A direct application applies to the order of a random element from $B_{s,n}^{(p)}$, where $\ord(x) = k$ for $k \ge 0$ minimal such that $x^k = 1$ denotes the order of $x$ in a group $G$. If $\sigma \in B_{s,n}^{(p)}$, then $\ord(\sigma) = p^{K_p(\sigma)}$. By \Cref{thm: s cycles LLN}, then necessarily $K(\sigma) = 1$ if $\sigma \ne 1$ and $K(\sigma) = 0$ if $\sigma = 1$. This can summarized as:
\begin{corollary}\label{cor: s order}
    Let $\sigma_n \sim \Unif(B_{s,n}^{(p)})$. Then $K_p(\sigma_n) \sim \Bern(1-p^{-n})$ and $K_p(\sigma_n)$ converges in probability to 1.
\end{corollary}

For general $m$, we have $C(\sigma_n) = \sum_{d \mid m} C_d(\sigma_n)$ when $\sigma_n \in B_{s,n}^{(m)}$, where $C_d(\sigma_n) C_\ell(\sigma_n) = 0$ for $d \ne \ell$; this follows from the mixed-product property, since each Kronecker factor has order that divides $m$ and so their Kronecker product has order that divides the least common multiple of each order. For example, if $m = 4$, then for $\sigma \sim \Unif(B_{s,1}^{(4)}) = \Unif(\langle (1 \ 2 \ 3 \ 4)\rangle)$, then $C_1(\sigma) \sim 4 \Bern(\frac14)$, $C_2(\sigma) \sim 2 \Bern(\frac14)$ and $C_4(\sigma) \sim \Bern(\frac12)$, while for $m = 6$, then $\sigma \sim \Unif(B_{s,1}^{(6)})$ satisfies $C_1(\sigma) \sim 6 \Bern(\frac16)$, $C_2(\sigma) \sim 3 \Bern(\frac16)$, $C_3(\sigma) \sim 2 \Bern(\frac13)$ and $C_6(\sigma) \sim \Bern(\frac13)$. 

In general, for $d \mid m$ and $d < m$, we have 
\[
C_d(\sigma_n) \sim \frac{N}d \Bern\left(\left(\frac{\#\{\ell \le d: \ell \mid d\}}m\right)^n - \left(\frac{\#\{\ell < d: \ell \mid d\}}m\right)^n\right),
\]
with then
\[
  C_m(\sigma_n) = \frac{N}m\left(1 - \sum_{d \mid m; d < m}  \frac{d}N C_d(\sigma_n)\right). 
\] 
This general formula follows from the method we will use in the following section; for now, we note if $\sigma_1 \in S_d$ is a $d$-cycle and $\sigma_2 \in S_\ell$ an $\ell$-cycle where $\ell \mid d$, then $\sigma_1 \otimes \sigma_2 \in S_{d\ell}$ is a product of $\ell$ disjoint $d$-cycles. In particular, we always have the number of fixed points $C_1(\sigma_n) \sim N \Bern(\frac1N)$. For example, for $m = 4$ and $\sigma_n \sim \Unif(B_{s,n}^{(4)})$ then $C_1(\sigma_n) \sim N \Bern(\frac1{4^n})$, $C_2(\sigma_n) \sim \frac{N}2 \Bern(\frac1{2^n} - \frac1{4^n})$ and $C_4(\sigma_n) \sim \frac{N}4 \Bern(1 - \frac1{2^n})$. In particular, we have $\frac{d}N C_d(\sigma_n)$ converges to 0 in probability for all $d\mid m$ and $d < m$ (e.g., the number of fixed points then satisfies this), so that $\frac{m}NC_m(\sigma_n)$ and hence $\frac{m}N C(\sigma_n)$ converge to 1 in probability.

\subsection{Nonsimple butterfly permutations}\label{sec: ns cycles 2}

The number of cycles for nonsimple butterfly permutations are not as immediate. In \cite{P24}, the author noted through numerical simulations that $B_N$ had number of cycle sample statistics (equivalently presented through GEPP pivot movement sample statistics) distinct from both the simple butterfly model as well as the uniform permutation model. We can now resolve exactly what this behavior is in what follows. In particular, this will establish the number of cycles for any $p$-Sylow subgroup of $S_{p^n}$, first focusing on $p = 2$.

We can first approach establishing the number of cycles directly by equivalently determining the number of GEPP steps needed for the corresponding permutation matrix. Let $P^{(k)}$ denote the intermediate GEPP form of the corresponding permutation matrix at step $k$ of $A$ has GEPP factorization $PA = LU$, with $P^{(1)} = \V I$ and $P^{(N-1)} = P$. If $P = P_\sigma$ with $\sigma = (n \ i_n)\cdots (2 \ i_2)(1 \ i_1) \in S_n$ with $i_k \ge k$, then $P^{(k+1)} = P_{(k \ i_k) \cdots (1 \ i_1)}$. Hence, $\Pi(P^{(k+1)}) = \Pi(P^{(k)}) + \mathds 1(i_k > k)$. First, we will use this in establishing a result for the intermediate GEPP forms of a Kronecker product of permutation matrices:
\begin{lemma}\label{l: int gepp}
    Let $C = P^T \otimes Q$ where $P,Q$ are permutation matrices, with $P = P_\sigma$ for $\sigma = (n \ i_n) \ldots (2 \ i_2)(1 \ i_1) \in S_n$ with $i_k \ge k$ and $Q \in \mathcal P_m$. After $n+1$ GEPP steps, then $C^{(n+1)} = \V I_m \oplus (P_\rho \otimes Q)$ where $\rho \in S_{n-1}$ is such that $1 \oplus \rho = (1 \ i_1) \sigma^{-1} = (2 \ i_2) \cdots (n \ i_n) \in S_n$.
\end{lemma}
\begin{proof}
    Note $C = P^T \otimes Q$ has GEPP factorization $\tilde P C = \tilde L \tilde U$ with $\tilde P = P \otimes Q^T$ and $\tilde L = \tilde U = \V I$. Then each intermediate GEPP form of $C$  necessarily is formed by acting on the first set of columns of $C$ to match the top $m$ columns of $\V I = \tilde U$; in particular, this only impacts the embedded $Q$ corresponding to the block location of $\sigma(1)$ in $C$ while the remaining $Q$ blocks remain untouched. Hence, after $n$ GEPP steps, then $(\tilde P^{(n+1)})^T = (Q^T \oplus \V I_{(n-1)m})(P_{(1 \ i_1)} \otimes Q)$ so that 
    \begin{align*}
        C^{(n+1)} &= (\tilde P^{(n+1)})^T(P \otimes Q) = (Q^T \oplus \V I_{(n-1)m})(P_{(2 \ i_2) \cdots (n \ i_n)} \otimes Q) \\
        &= \V I_m \oplus (P_\rho \otimes Q),
    \end{align*}
    as desired.
\end{proof}

For the nonsimple case, we consider 
\begin{equation*}
    \sigma_n = (\sigma_\theta \otimes 1)(\sigma_1 \oplus \sigma_2)=(\sigma_\theta \otimes \sigma_1)(1 \oplus \sigma_1^{-1} \sigma_2) \in B_N, 
\end{equation*}
where $\sigma_i \in B_{N/2}$ and $\sigma_\theta \in B_2$. If $\sigma_\theta = 1$, then $\sigma_n = \sigma_1 \oplus \sigma_2$ and so $C(\sigma_n) = C(\sigma_1) + C(\sigma_2)$. If $\sigma_\theta \ne 1$, then
\[
P_{\sigma_n} = \begin{bmatrix}
    \V 0 & P_{\sigma_2} \\ P_{\sigma_1} & \V 0
\end{bmatrix}.
\]
In particular, we see if $j \le N/2$, then $j$ is in the same $\sigma_n$-orbit as $\sigma_n(j) = N/2 + \sigma_1(j) > j$. Hence, if $\sigma_n = (N \ i_N) \cdots (2 \ i_2)(1 \ i_1)$, then necessarily $i_j > j$ for all $j \le N/2$. Moreover, the first $N/2$ GEPP steps are agnostic of the remaining $N/2$ columns of $P$, and so the intermediate permutation factor necessarily aligns with the factor needed for $P_{\sigma_\theta} \otimes P_{\sigma_1}$. Hence, using \Cref{l: int gepp}, we have
\begin{equation*}
    P^{(N/2 + 1)} = (P_{\sigma_\theta} \otimes P_{\sigma_1})^{(n+1)}(\V I \oplus P_{\sigma_1^{-1}\sigma_2}) = (\V I \oplus P_{\sigma_1})(\V I \oplus P_{\sigma_1^{-1}\sigma_2}) = \V I \oplus P_{\sigma_2}.
\end{equation*}
It follows then $\Pi(P_{\sigma_n}) = N/2 + \Pi(P_{\sigma_2})$. Putting this together, we have
\begin{align*}
\Pi(P_{\sigma_n}) 
&= \Pi(P_{\sigma_2}) + \left\{\begin{array}{ll}
\Pi(P_{\sigma_1}), & \sigma_\theta = 1\\
N/2, & \sigma_\theta \ne 1.
\end{array} 
\right.
\end{align*}
Since 
\begin{equation*}
    C(\sigma_n) = N - \Pi(P_{\sigma_n}) = (N/2-\Pi(P_{\sigma_2})) + \mathds 1(\sigma_\theta = 1)(N/2-\Pi(P_{\sigma_1})),
\end{equation*}
then this yields:
\begin{proposition}\label{prop: ns cycle dist}
    If $\sigma_n \sim \Unif(B_N)$, then $C(\sigma_n) \sim Y_n$ where $Y_0 = 1$ and $Y_{n + 1} = Y_n + \eta_n Y_n'$ for $Y_n \sim Y_n'$ independent of $\eta_n \sim \Bern(\frac12)$ for $n \ge 2$.
\end{proposition}
For notational convenience, we will now often use $Y_n$ in place of $C(\sigma_n)$ for $\sigma_n \sim \Unif(B_N)$. We can now compute
    \begin{align}
        \E Y_n &= \left(\frac32\right)^n = N^{\log_2(3/2)}\approx N^{0.5850} \qquad \mbox{and} \label{eq: exp cycle ns}\\
        \Var Y_n &= \frac13\left(\frac32\right)^{2n} - \frac13\left(\frac32\right)^n \approx \frac13N^{1.1699}(1+o(1)). \label{eq: var cycle ns}
    \end{align}
\begin{remark}\label{rmk: exp match}
    The expected number of cycles for a nonsimple binary butterfly permutation match exactly the expected length of the longest increasing subsequence of a simple binary butterfly permutation. We will see this only holds for $p = 2$, as otherwise the expected simple LIS is strictly larger than the expected nonsimple number of cycles.
\end{remark}

To establish \eqref{eq: exp cycle ns}, we have
\begin{equation*}
    \E Y_n = \E Y_{n-1} + \E \eta_{n-1} \E Y_{n-1} = \frac32 \E Y_{n-1} = \left(\frac32\right)^n,
\end{equation*}
using $Y_n,Y_n',\eta_n$ all independent and induction. Additionally, we have
\begin{equation*}
    \Var(\eta_n Y_n') = \frac12\Var Y_n + \frac14(\E Y_n)^2, 
\end{equation*}
by the law of total variance, so that $\Var Y_n$ satisfies the recursion $b_{n+1} = \frac32 b_n + \frac14 (\frac32)^{2n}$ with $b_0 = 0$, which has the solution \eqref{eq: var cycle ns}.

The expected number of cycles $\E C(\sigma) \approx N^{0.5850}$ for $\sigma \sim \Unif(B_N)$ asymptotically lies between each of the expected number of cycles when $\sigma\sim \Unif(B_{s,N})$ and $\sigma\sim \Unif(S_N)$, which take the respective values
\begin{equation*}
    \frac{N+1}2,\qquad H_N^{(1)}= \log N + \gamma + o_n(1).
\end{equation*}
However, the variance $\Var C(\sigma) \approx \frac13 N^{1.1699}(1+o(1))$ asymptotically dominates both the variances, respectively, of
\begin{equation*}
    \frac{N-1}4, \qquad  H_N^{(1)} - H_N^{(2)} = \log N + \gamma - \frac{\pi^2}6 + o_n(1)
\end{equation*}
for $\sigma \sim \Unif(B_{s,N})$ and $\sigma \sim \Unif(S_N)$.

Furthermore, we note that both $\E C(\sigma_n)$ and $\sqrt{\Var C(\sigma_n)}$ are asymptotically of the \textit{exact} same order. Hence, the scaled random variable 
\begin{equation*}
    W_n = Y_n/\E Y_n = Y_n \cdot \left(\frac23\right)^n
\end{equation*}
has mean 1 and variance $\frac13+o_n(1)$, and so can have a random distributional limit $W = \lim_{n \to \infty} W_n$. We will establish such a limit does exists shortly. To better understand the behavior of $W$, we will compute the general moments of $\E Y_n^k$; a similar method and recursion solution yields
\begin{equation*}
    \E Y_n^3 = \frac{32}{15}\left(\frac32\right)^{3n} - \frac{4}3\left(\frac32\right)^{2n} + \frac{1}5\left(\frac32\right)^{n}.
\end{equation*}
This is sufficient now to establish $W$ (assuming it  exists) is not Gaussian: by above, $W$ has mean $\mu = 1$,  variance $s^2=\frac13$ (with second moment then $\frac43$), and third moment $\frac{32}{15}$. Since $N(\mu,s^2)$ has third moment $\mu^3 + 3\mu s^2 = 2 < \frac{32}{15}$, then $W$ is not Gaussian. Of course, this can also be achieved by noting $W$ has positive support  since $C(\sigma_n)$ is nonnegative. We can also exclude any hope that log-normal shows up, since if $Z \sim N(a,b)$, then $\E e^{kZ} = \exp(ka + k^2 b^2/2)$. To match then $\exp(a + b^2/2) = 1$ and $\exp(2a + 2b^2) = \frac43$ then $\E e^{3Z} = (\frac43)^3 = \frac{64}{27}> \frac{32}{15}$, which now also establishes such a $W$ cannot be log-normal.

We will also rule out log-normal for another important reason in terms of the corresponding moments. We will establish such a $W$ exists that is uniquely determined by its moments (see \Cref{prop: unique measure 2}), while the log-normal is often then prototype for a distribution not determined by its moments: if $f(x)$ is the density for a standard log-normal, then one can construct a density $g_{u,m}(x) = f(x) \cdot (1 + u \cdot \sin(2\pi m \ln x))$ for a non-lognormal random variable with matching moments for any $0 < u < 1$ and integer $m$. 

To establish explicit moments for $W$, we will first establish strong moment properties for $Y_n$. We already showed $Y_n$ has first three moments proportional to $\E Y_n^2 = (\E Y_n)^2\cdot(1 + o(1))$. In general, we can show each $k$-th moment of $Y_n$ is exactly proportional to $(\E Y_n)^k = (\E Y_1)^{nk}$. In other words, we have $\E W_n^k = m_k\cdot (1+o_n(1))$ and so $\E W^k = m_k$ (assuming such a $W$ exists) for each integer $k \ge 1$ for constants $m_k \ge 1$. Khintchine's inequality yields this is necessarily the case since $\E W_n^2 = (\E W_n)^2\cdot (1 + o_n(1))$; we will establish this directly by providing a recursive polynomial formula to generate each such moment.

\begin{proposition}\label{prop: poly moments}
    $\E Y_n^k = p_k(\lambda^n)$ for $p_k$ a polynomial of degree $k$ and $\lambda := \frac32$, which satisfy $p_1(x) = x$ and for $k \ge 2$ then $p_{k}(x) = \sum_{j=1}^k a_{kj} x^j$ where $a_{kj} = r_{kj} \cdot \frac{\lambda - 1}{\lambda^j - \lambda}$ for $\sum_{j=2}^k r_{kj} x^j = \sum_{j=1}^{k-1} \binom{k}j p_j(x)p_{k-j}(x)$ for $j \ge 2$ and $a_{k1} = 1-\sum_{j=2}^k a_{kj}$. In particular, $p_k(0) = 0$, $p_k(1) = 1$, and $m_k := a_{kk} > 0$ for all $k$.
\end{proposition}

\begin{proof}
    The last statements follow for all $k$ since $p_k(0) = a_{k0} = 0$, $p_k(1) = \sum_{j=1}^k a_{kj} = 1$ by the definition of $a_{k1}$, and $a_{11} = 1> 0$ and $m_{k} > 0$ inductively for $k > 1$ since then $r_{kk} > 0$. 
    
    We will establish the moment equality $\E Y_n^k = p_k(\lambda^n)$ by induction on $k$. If $k = 1$, we have $\E Y_n = (\frac32)^n = p_1(\lambda^n)$ with $p_1(x) = x$. Now assume the result holds for all $1 \le \ell < k$, i.e., $\E Y_n^\ell = p_\ell(\lambda^n)$ and all $n$. 
    By definition of $p_k$, we have
    \begin{align*}
        (\lambda - 1)\sum_{j=1}^{k-1} \binom{k}j p_j(x) p_{k-j}(x) &= (\lambda - 1)\sum_{j = 2}^k r_j x^j = \sum_{j = 2}^{k} a_{kj}(\lambda^j - \lambda) x^j\\
        &= p_k(\lambda x) - \lambda p_k(x).
    \end{align*}
    In particular, then
    \begin{equation*}
        (\lambda - 1)\sum_{j=1}^{k-1} \binom{k}j p_j(\lambda^n) p_{k-j}(\lambda^n) = p_k(\lambda^{n+1}) - \lambda p_k(\lambda^n).
    \end{equation*}
    By the linearity of expectation, we also have
    \begin{equation*}
    \E Y_{n+1}^k = \E(Y_{n} + \eta_n Y_{n}')^k = \lambda \E Y_{n}^k + (\lambda - 1) \sum_{j=1}^{k-1} \binom{k}{j} \E Y_{n}^j \E Y_{n}^{k-j}.
    \end{equation*}
    Now using the inductive hypothesis, we have $\E Y_{n}^j = p_k(\lambda^n)$ for all $j = 1,\ldots,k-1$, so that
    \begin{equation*}
        \E Y_{n+1}^k = p_k(\lambda^{n+1}) + \lambda(\E Y_n^k -  p_k(\lambda^n)).
    \end{equation*}
    The result for $k$ then follows directly from induction on $n$, where we note $\E Y_0^k = 1 = p_k(\lambda^0) = p_k(1)$.
\end{proof}

An immediate consequence is the following:
\begin{corollary}\label{cor: limit moments}
    $\E W_n^k = m_k \cdot (1 + O_n(\lambda^{-n})) = m_k \cdot (1 + o_n(1))$
\end{corollary}

\begin{table}
    \centering
    \begin{tabular}{c|c}
        $k$ & $p_k(x)$\\
        \hline
        0 & 1\\
        \vspace{-.5pc}\\
        1 & $\displaystyle x$\\
        \vspace{-.5pc}\\
    2 & $\displaystyle \frac43 x^2 - \frac13 x$\\
    \vspace{-.5pc}\\
    3 & $\displaystyle \frac{32}{15} x^3 -\frac43 x^2 + \frac{1}5 x$ \\
    \vspace{-.5pc}\\
    4 & $\displaystyle \frac{3228}{885} x^4 - \frac{64}{15} x^3 + \frac{68}{45} x^2 - \frac{13}{95} x$\\
    \vspace{-.5pc}\\
    5 & $\displaystyle \frac{262144}{33345} x^5 - \frac{6656}{513} x^4 + \frac{352}{45} x^3 - \frac{308}{171}x^2 + \frac{341}{3705} x$
    \end{tabular}
    \caption{First few polynomials $p_k(x)$ for small $k$}
    \label{t: polynomials}
\end{table}
\Cref{t: polynomials} displays the first few values of $p_k$. In particular, we note $\E Y_n = p_1(\lambda^n)$, $\E Y_n^2 = \Var Y_n + (\E Y_n)^2 = p_2(\lambda^n)$, and $\E Y_n^3 = p_3(\lambda^n)$ match the previous computed forms of the first three moments. 

Moreover, from \Cref{prop: poly moments}, we have $\E Y_n^k = m_k \lambda^{nk} \cdot (1 + o_n(1))$, while $m_1 = 1$  and $m_k$ for $k \ge 2$ satisfy the the recurrence relationship
\begin{equation}\label{eq: moment recursion}
    m_k = \frac{\lambda-1}{\lambda^k - \lambda} \sum_{j = 1}^{k-1} \binom{k}j m_j m_{k-j}.
\end{equation}
In \Cref{t: polynomials} above we see $m_2 = \frac43$, $m_3 = \frac{32}{15}$, $m_4 = \frac{3228}{885}$ and $m_5 = \frac{262144}{33345}$, and so we now also have
\begin{align*}
    m_6 &= \frac{\frac12\left[2\binom{6}1 m_1 m_5 + 2\binom{6}2 m_2 m_4 + \binom{6}3 m_3^2\right]}{(\frac32)^6 - \frac32}  = \frac{40435712}{2345265}. 
\end{align*}



Next, we will establish such moments $m_k$ uniquely determine a probability law. 

\begin{proposition}\label{prop: unique measure 2}
    There exists a unique probability measure with support $(0,\infty)$ having positive integer moments $m_k$.
\end{proposition}
\begin{proof}
    It suffices to show Carleman's criterion is satisfied for the Stieljes moment problem, i.e., $\sum_{k=1}^\infty m_k^{-1/2k} = \infty$. Using Stirling's approximation, it then further suffices to show $m_k \le \frac{\lambda}{\lambda^k}k!$ for all $k$. Note first $\frac{\lambda^{k-1} - 1}{\lambda -1} = \sum_{j=0}^{k-2} \lambda^j > k-1$ since $\lambda > 1$. Now we see $m_1 = 1 = \frac{\lambda}{\lambda}\cdot1!$, and if $m_j \le \frac{\lambda}{\lambda^j} j!$ for all $j < k$, then $m_k \le \frac{\lambda}{\lambda^k} k! \cdot \frac{k - 1}{ \frac{\lambda^{k-1} - 1}{\lambda - 1}} \le \frac{\lambda}{\lambda^k} k!$ using \eqref{eq: moment recursion}.
\end{proof}

\begin{figure}
    \centering
    \includegraphics[width=0.75\linewidth]{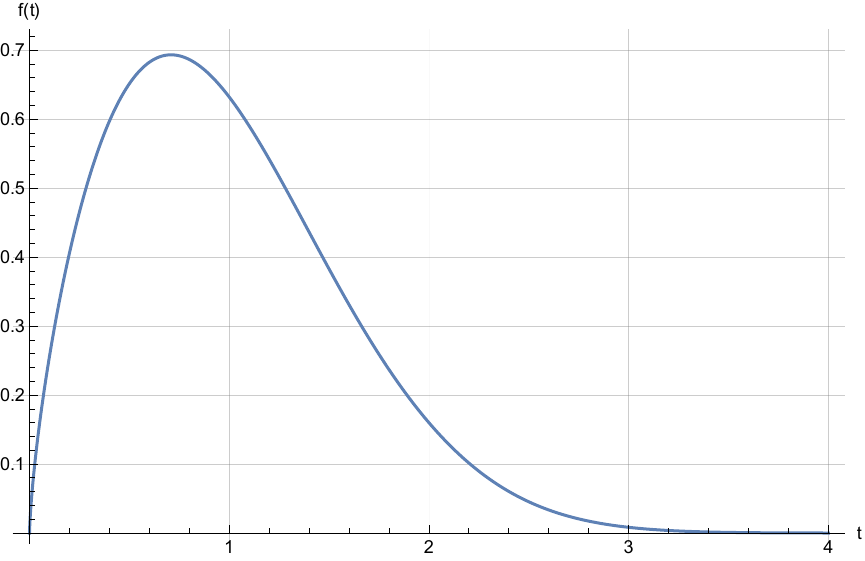}
    \caption{Plot of density function for $W$.}
    \label{fig:cycle limit pdf}
\end{figure}

\Cref{fig:cycle limit pdf} shows a plot of the density function for $W$ that has $k^{th}$ moments given by $m_k$. We can now establish the following CLT result for the number of cycles of a uniform nonsimple (binary) permutation, which follows then from \Cref{prop: unique measure 2}, \Cref{cor: limit moments}, and the moment method. 
Note in particular this then determines the limiting distribution for the number of cycles for \textit{any} uniform 2-Sylow subgroup of $S_{2^n}$, as every Sylow subgroup is conjugate to $B_N$ and the number of cycles of a permutation is invariant under conjugation.

\begin{theorem}\label{thm: cycles 2}
    Let $\sigma_n \sim \Unif(B_N)$. Then $W_n = C(\sigma_n)(2/3)^n$ converges in distribution to $W$ with full support on $(0,\infty)$, where $\E W^k = m_k$ for all positive integers $k$, for $m_1 = 1$ and $m_k$ that satisfy the recursion formula \eqref{eq: moment recursion}.
\end{theorem}

By \Cref{prop: unique measure 2}, we can define the moment generating function for $W$, $M_W(z) = \E e^{zW} = 1 + \sum_{k=1}^\infty \frac{m_k}{k!} z^k$. 
Moreover, noting 
\begin{equation*}
    2M_{Y_{n+1}}(z) = M_{Y_n}(z)^2 + M_{Y_n}(z)
\end{equation*}
and $M_{W_n}(z) = M_{Y_n}((\frac23)^n z)$, then $M_W(z) = \lim_{n \to \infty} M_{W_n}(z)$ is a solution to the functional equation
\begin{equation}\label{eq: functional eqn}
    2 \xi\left(\frac32z\right) = \xi\left(z\right)^2 + \xi\left(z\right)
\end{equation}
where $\xi(0) = \xi'(0) = 1$. Equating the series expansion coefficients after plugging $M_W(z)$ into \eqref{eq: functional eqn} aligns with the previous recursive moment formula \eqref{eq: moment recursion}. However, this does not yield an easy route to finding a closed form for $M_W(z)$ nor for the associated density $f(t)$ of $W$ (via the inverse Laplace transform of $M_W(z)$).

Furthermore, \Cref{prop: unique measure 2}  yields the sub-exponential tail bounds 
\begin{equation}\label{eq: exp upper bd}
    \P(W \ge t) \le M_W(zt)e^{-zt} \le \frac{\lambda^2}{\lambda - z} e^{-zt}
\end{equation}
for all $z \in (0,\lambda)$ and $t > 0$. We also have the standard bounds $\P(W \ge t) \le m_k t^{-k}$ for $t > 0$. In particular, we see then 
\begin{equation*}
    \P(W \ge 4) \le \frac{m_{23}}{4^{23}} \approx 0.0000902136. 
\end{equation*}
This beats the optimal choice of $z^* = \frac54$ in \eqref{eq: exp upper bd}, where then $\frac{\lambda^2}{\lambda - z^*} e^{-4z^*} = \frac9{e^5} \approx 0.060642$. (Note $m_{23} \approx 6.3482 \cdot 10^9$, which minimized the first 100 exact values of $\frac{m_k}{4^k}$.\footnote{$\frac{m_{63}}{4^{63}} \approx 1.1092$ is the first instance when $\frac{m_k}{4^k} \ge 1$. Studying the ratios $\frac{m_k}{4^k}$ for the first 100 moments, these monotonically decreased to $\frac{m_{23}}{4^{23}}$ and then monotonically increased afterwards, growing rapidly, where $\frac{m_{100}}{4^{100}} \approx 1.95675 \cdot 10^8$. We conjecture $m_k \ge 4^k$ for all $k \ge 63$.}) Hence, more than 99.99\% of the total mass of the density for $W$ is concentrated on the interval $[0,4]$. This is apparent in the plot of the corresponding limiting density function in \Cref{fig:cycle limit pdf}. 

Although we do not provide an explicit description of the form of limiting density function $f_W(t)$, we can still recover a very good approximation using the following pointwise limit formula, which was used to generate \Cref{fig:cycle limit pdf} using exact probability values for butterfly permutations of order $N = 2^{20} = 1,048,576$: 
\begin{proposition}
    For $f_W$ the density function for $W$, then
    \begin{equation*}
        f_W(t)  = \lim_{n \to \infty} \P\left(Y_n = \left\lfloor\left(\frac32\right)^nt\right\rfloor\right)\left(\frac32\right)^n.
    \end{equation*}
\end{proposition}
\begin{proof}
We have
    \begin{align*}
    f_W(t) = F_W'(t) &= \lim_{n \to \infty} \left(F_W(t) - F_W\left( t - \left(\frac23\right)^n\right)\right)\left(\frac 32\right)^n\\
    &= \lim_{n \to \infty} \lim_{m\to\infty}\left(F_{W_m}(t) - F_{W_m}\left( t - \left(\frac23\right)^n\right)\right)\left(\frac 32\right)^n \\
    &= \lim_{n \to \infty} \left(F_{W_n}(t) - F_{W_n}\left( t - \left(\frac23\right)^n\right)\right)\left(\frac 32\right)^n \\
    &= \lim_{n \to \infty} \P\left(W_n = \left\lfloor\left(\frac32\right)^nt\right\rfloor\left(\frac23\right)^n\right)\left(\frac32\right)^n \\
    &= \lim_{n \to \infty} \P\left(Y_n = \left\lfloor\left(\frac32\right)^nt\right\rfloor\right)\left(\frac32\right)^n,
\end{align*}
where the uniform convergence of $F_{W_n}$ to $F_W$ is used in the third line to absorb the inside limit. 
\end{proof}
We can compute exact values for $\P(Y_n = k)$ using
\begin{equation*}
    s_{B}(n,k) = \#\{\sigma \in B_N: C(\sigma) = k\} = 2^{2^n-1}\P(Y_n = k),
\end{equation*}
which we can suggestively call the butterfly Stirling numbers of the first kind. From \Cref{prop: ns cycle dist}, these satisfy the recursion
\begin{equation*}
    s_{B}(n+1,k) = 2^{2^n-1}s_{B}(n,k) + \sum_{j=1}^{k-1} s_{B}(n,j)s_{B}(n,k-j).
\end{equation*}
Note $s_{B}(n,2^n) = 1$ and $s_{B}(n,2^n-1) = 2^{n-1}$, while $s_{B}(n,1) = 2^{2^n-n-1}$. \Cref{fig:cycle limit pdf} then was computed using high precision values in Mathematica of $s_{B}(n,k)$ with $n = 20$, for $k = 1$ to $\floor{(\frac32)^{n} \cdot 4}=13301$ (note $|B_N| = 2^{2^n-1} \approx 6.7411 \cdot 10^{315652}$). \Cref{t:snk} shows the values of $s_{B}(n,k)$ for small $n$.

\begin{table}[ht!]
\centering
{
\begin{tabular}{r|ccccccccc}
&\multicolumn{8}{c}{$k$}\\
$n$ & 1 & 2 & 3 & 4 & 5 & 6 &7& 8 & \ldots \\ \hline 
1 & 1& 1\\
2 & 2& 3& 2& 1\\
3 & 16& 28& 28& 25& 16& 10& 4& 1\\
4 & 2048& 3840& 4480& 4880& 4416& 3976& 3128& 2337 &\ldots
\end{tabular}
}
\caption{Triangle of $s_{B}(n,k)$ values for $n = 1$ to 4}
\label{t:snk}
\end{table}

Although we already established the moment problem is determinant for moments $m_k$, we can expand the analytic properties of $M_W(z)$ outside of $|z| < \alpha$.
\begin{proposition}
    $M_W(z)$ is defined for all $z \in \mathbb R$.
\end{proposition}

\begin{proof}
    We want to show $(\limsup_k (m_k/k!)^{1/k})^{-1} = \infty$. This can be accomplished using Stirling's approximation formula then by showing $\frac{m_k^{1/k}}{k} \le \frac1{\ln(k+1)}$ or equivalently $m_k \le g(k) := \left(\frac{k}{\ln(k + 1)}\right)^k$ for all $k$. We will use induction on $k$. This holds for $k = 1$ to 6 by direct computation, comparing the respective (approximate) values of $m_k$ to $g(k)$, where $m_1 = 1 < g(1) \approx 1.44$, $m_2 \approx 1.33 < g(2) \approx 3.31$, $m_3 \approx 2.13 < g(3) \approx 10.13$, $m_4 \approx 3.65 < g(4) \approx 38.15$, $m_5 \approx 7.89 < g(5) \approx 169.22$, and $m_6 \approx 17.24 < g(6) \approx 859.35$. Now assume $m_k \le g(k)$ for all $\ell < k$ where $k \ge 7$. 

    We will use the following technical lemma:
    \begin{lemma}
        Let $k > 0$. For $H_k(x) = \frac{g(x)g(k-x)}{g(k)}$ and $\rho_k =  \left(\frac{\ln(k+1)}{2 \ln(k/2 + 1)}\right)^2$, then $H_k(x) \le \rho_k^x$ for all $x \in [0,k/2]$.
    \end{lemma}
    \begin{proof}
        We have $\rho_k \in (\frac14,1)$ and $H_k(0) = H_k(k) = 1$, while $H_k(x)$ has a unique minimum on $[0,k]$ at $k/2$, where $H_k(k/2) = \frac{g(k/2)^2}{g(k)} = \left(\frac{\ln(k+1)}{2 \ln(k/2 + 1)}\right)^k = \beta_k^{k/2}$. The latter fact follows by symmetry since $H_k(x) = H_k(k-x)$ for $x \in [0,k]$ so necessarily $H_k'(k/2) = 0$, which is the minimum since $H_k(x)$ is strictly convex. Moreover, we have $1 = H_k(0) = \rho_k^0$ (and $H_k(k/2) = \rho_k^{k/2}$). Since $\rho_k^x$ is also strictly convex and $\left.\frac{d}{dx} \rho_k^x \right|_{x = k/2} = -(\ln \frac1{\rho_k}) \rho_k^{k/2} < 0 = H'_k(k/2)$, then necessarily $\rho_k^x - H_k(x) \ge 0$ for all $x \in [0,k/2]$.
    \end{proof}

    It follows then $g(j) g(k - j) \le \rho_k^j g(k)$ for all $j = 1,2,\ldots,\floor{k/2}$ and in particular
    \begin{equation*}
        \sum_{j=1}^{k-1} \binom{k}j g(j) g(k-j) \le 2 \sum_{j = 1}^{\floor{k/2}} \binom{k}j g(j)g(k-j) \le 2(1 + \rho_k)^k g(k).
    \end{equation*}
    Using the inductive hypothesis to bound $m_j m_{k-j} \le g(j)g(k-j)$ for $j=1,\ldots, k-1$, we have $m_k \le D(k) g(k)$, where
    \begin{align*}
        D(x) = \frac{(1 + \rho_x)^x}{\lambda^x - \lambda} = \frac1{1 - \frac32(\frac23)^x} \left[\frac23\left( 1 + \left(\frac{\ln(x + 1)}{2\ln(x/2 + 1)}\right)^2\right) \right]^x.
    \end{align*}
    Since $D(x)$ is decreasing for $x \ge 2$ (noting $D(x)$ is strictly convex while also $D(x)$ is asymptotic to $(\frac56)^x$) and $D(7) \approx 0.987832$, then $D(k) < 1$ for all $k \ge 7$. It follows then $m_k \le g(k)$, as desired.
\end{proof}
Future work can explore establishing the exact asymptotics of $m_k$. Numerically, it appears $\frac{{m_k}^{1/k}}{k} = O(\frac1{\ln^2(k+1)})$.

\subsubsection{Numerical experiments: \texorpdfstring{{\boldmath$W_N$}}{n} empirical moments}

\begin{figure}
    \centering
    \includegraphics[width=0.75\linewidth]{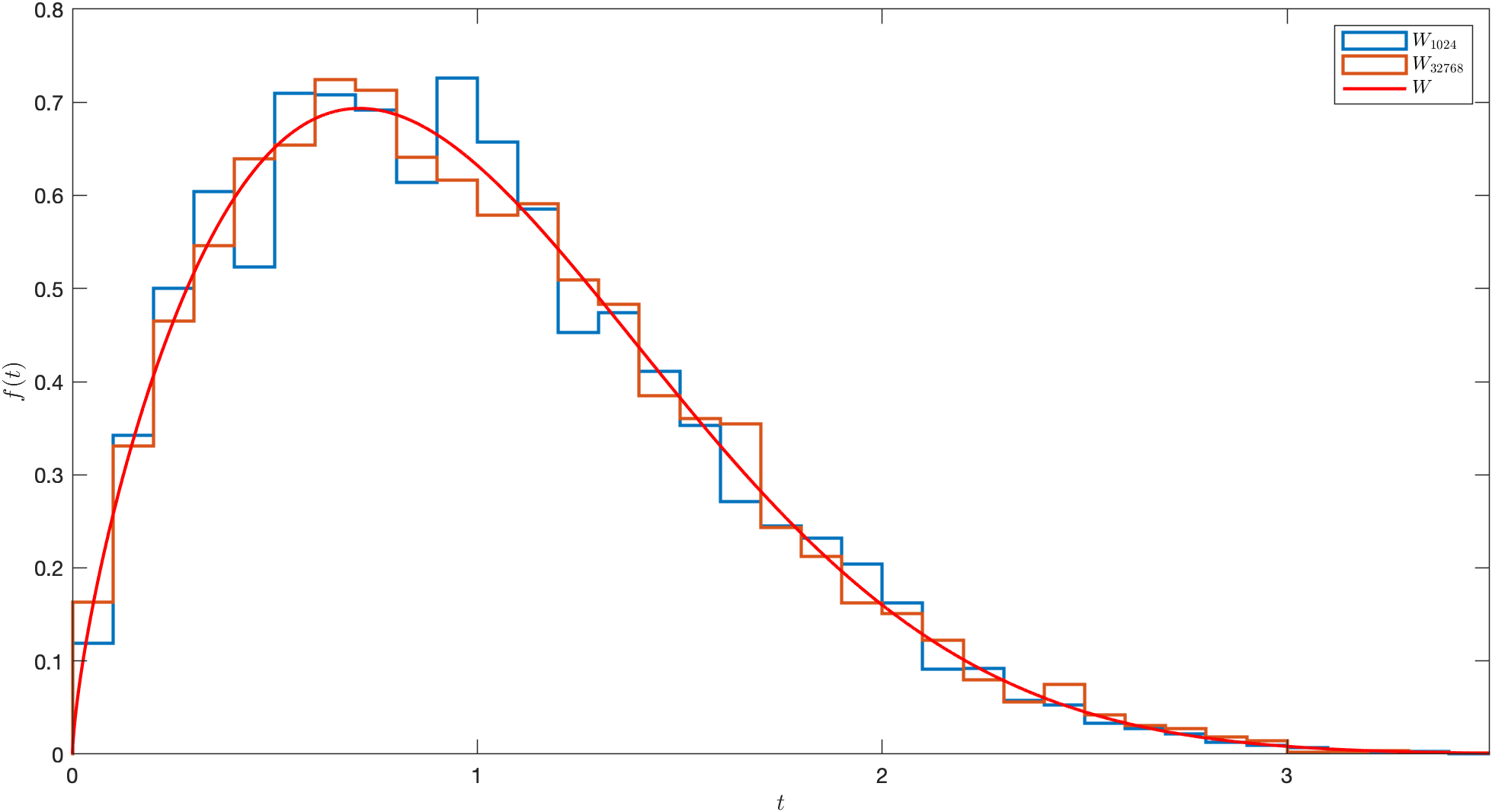}
    \caption{Plot of $W_N = C(\sigma_n)(2/3)^n$ for $\sigma_n \sim \Unif(B_N)$, for $n = 10,15$, using $10^4$ trials, compared to the limiting density of $W$.}
    \label{fig:cycle experiments}
\end{figure}

To support our results in \Cref{sec: ns cycles 2}, we run a series of simple experiments, using set number of trials of $W_N = C(\sigma_n)(2/3)^n$ for $\sigma_n \sim \Unif(B_N)$. We then compare the first 6 empirical sample moments $\hat m_k = \frac1M \sum_{j = 1}^M x_j^k$ for iid $x_j \sim W_N$ to the limiting moments $m_k$ for $W$. For simplicity, we limit our focus to using $M = 10^5$ trials using $n = 5$ and $M = 10^4$ trials each for $n = 10$ and $n = 15$. \Cref{tab:moments 2} shows the empirical sample moments for each set of trials, compared to the exact limiting moments. \Cref{fig:cycle experiments} shows an overlay of the associated histograms for the $n = 10$ and $n = 15$ trials, compared to the limiting density function of $W$, as approximated using exact values for the associated probability mass function for $W_{20}$.

\begin{table}[ht!]
\centering
\begin{tabular}{c|c|ccc|}
&&\multicolumn{3}{c}{$\hat{m}_k$}\\
$k$ & $m_k$ & $n=5$& $n=10$ & $n=15$ \\
\hline
1 & 1.00000 & 1.00184 & 0.99553 & 1.00189 \\
2 & 1.33333 & 1.29503 & 1.31766 & 1.34246 \\
3 & 2.13333 & 1.97621 & 2.09036 & 2.15959 \\
4 & 3.64746 & 3.39555 & 3.77517 & 3.95717 \\
5 & 7.86157 & 6.38829 & 7.52958 & 8.00596 \\
6 & 17.24143 & 12.92663 & 16.26121 & 17.51766 \\
\end{tabular}
\caption{Empirical moments $\hat{m}_k = \frac1M\sum_{j=1}^M x_j^k$ for $M$ iid samples $x_j \sim W_{N} = C(\sigma_n)(2/3)^n$ with $\sigma_n \sim \Unif(B_{N})$ for $n=5$ (using $M = 10^5$ trials) and $n = 10,15$ (using $M = 10^4$ trials), compared to the exact moments $m_k$ of the limiting distribution $W$.}
\label{tab:moments 2}
\end{table}

Note that only the first empirical moment is an unbiased estimator for the corresponding population moment. Since each $W_N$ are formed by scaling $C(\sigma_n)$ to have unit mean, then the first empirical moment aligns closely to 1 for each set of samples. Although higher order moments are biased (skewing to be slightly over biased), our trials  produce empirical moments close to the final limiting moments.

\subsubsection{\texorpdfstring{{\boldmath$p$}}{n}-nary butterfly permutations}

Many of the methods used in the prior section carry over to the $p$-nary butterfly permutations when $p$ is a prime. Hence, these again classify the number of cycles question for uniform permutations from any $p$-Sylow subgroup of $S_{p^n}$. While these methods can be expanded to handle composite $m$, we will keep the current focus on the case when $m=p$ is prime. 

We can again approach the plan to determine the distribution of $C(\sigma_n)$ for $\sigma_n \in B_n^{p}$ by considering the equivalent question of determining the number of GEPP pivot movements needed on $P_{\sigma_n}$, i.e., $\Pi(P_{\sigma_n})$. We can write $\sigma_n = (\sigma_\theta \otimes 1_{N/p})(\bigoplus_{j = 1}^p \sigma_j)$ where $\sigma_j \in B_{n-1}^{(p)}$ and $\sigma_\theta \in B_1^{(p)} = \langle (1 \ 2 \ \cdots \ p)\rangle$. If $\sigma_\theta = 1$, then $\Pi(\sigma_n) = \sum_{j = 1}^p \Pi(\sigma_j)$.

Now consider the case $\sigma_\theta \ne 1$. Since $p$ is prime, then $(1 \ 2 \ \cdots \ p)^k$ is also a $p$-cycle for all $k = 1,\ldots,p-1$ (note here's one step where the composite $m$ case would need adjustment); in particular, then $\sigma_\theta$ is a $p$-cycle, so that $\sigma_\theta = (p \ i_p) \cdots (2 \ i_2)(1 \ i_1)$ must satisfy $i_k > k$ for all $1 \le k < p$ (since $i_p = p$ must be maximal for this cycle). Hence, the intermediate GEPP form of $\sigma_\theta$ after $j-1$ GEPP steps is necessarily of the form $(p \ i_p) \cdots (j \ i_j)$, and  is thus a $p-j$ cycle. Now necessarily the first $p^{n-1}$ steps of GEPP applied to $P_{\sigma_n}$ only act on the first block in $P_{\sigma_n}$ that consists of $P_{\sigma_1}$ in the block location $\sigma_\theta(1)$; in particular, $j \le p^{n-1}$ is in the same $\sigma_n$-orbit as $\sigma_1(j) + (\sigma_\theta(1) - 1)p^{n-1} > j$, so a GEPP pivot is required at each of the first $p^{n-1}$ GEPP steps. By \Cref{l: int gepp}, then the $(p+1)^{st}$ intermediate GEPP form of $P_{\sigma_n}$ is $\V I \oplus [(\sigma_\theta^{(1)} \otimes 1)(\bigoplus_{j = 2}^p P_{\sigma_j})]$, where now $\sigma_\theta^{(1)}$ is a $p-1$ cycle. It then inductively follows that a pivot movement is needed at each of the first $(p-1)p^{n-1}$ steps, while the remaining untriangularized block is $P_{\sigma_p}$. To summarize, we have
\begin{equation*}
    \Pi(P_{\sigma_n}) = \Pi(P_{\sigma_p}) + \left\{ \begin{array}{cc}
    \sum_{j = 1}^{p-1} \Pi(P_{\sigma_j}), & \sigma_\theta = 1\\ (p-1)p^{n-1}, & \sigma_\theta \ne 1,
    \end{array}\right.
\end{equation*}
which can be equivalently reframed as:
\begin{proposition}
    If $\sigma_n \sim \Unif(B_n^{(p)})$ for prime $p$, then $C(\sigma_n) \sim Y^{(p)}_n$ where $Y^{(p)}_0 = 1$ and $Y^{(p)}_{n+1} = Y^{[1]}_n + \eta_n \sum_{j=2}^p Y_n^{[j]}$ for $Y_n^{[j]} \sim Y_n^{(p)}$ iid and independent of $\eta_n \sim \Bern(\frac1p)$ for $n \ge 2$.
\end{proposition}
We have $Y_1^{(p)} = 1 + (p-1) \eta_1$, and so let
\begin{equation*}
    \lambda_p = \E Y_1^{(p)} = 1 + (p-1) \cdot \frac1p = 2 - \frac1p.
\end{equation*}
\begin{remark}
    Note $\frac1p = \frac{\lambda_p - 1}{p-1}$, while $\lambda_2 = \frac32$ is consistent with the previous section. Moreover, as previously mentioned in \Cref{rmk: exp match}, we have $2^{\alpha_2} = \frac32 = \lambda_2$ while $p^{\alpha_p} > 2 >  \lambda_p$ for $p > 2$.
\end{remark}
We can now similarly directly adapt the proof (which we will abstain from including here) used for $p = 2$ to work for general prime $p$ to derive:
\begin{proposition}
    $\E [{Y_n^{(p)}}]^k = p_k^{(p)}(\lambda_p^n)$ where $p_k^{(p)}$ is a polynomial of degree $k$ and $\lambda_p = 2-\frac1p$, which satisfy $p_1^{(p)}(x) = x$ and for $k \ge 2$ then $p_k^{(p)}(x) = \sum_{j = 1}^k a_{kj}^{(p)}x^j$ where $a_{kj}^{(p)} = r_{kj}^{(p)} \cdot  \frac1{p-1}\frac{\lambda_p - 1}{\lambda_p^k - \lambda_p}$ for \[\sum_{j = 2}^k r_{kj}^{(p)} x^j = \sum_{\text{\scriptsize{$\begin{array}{c}i_1 + \cdots + i_p = k\\ i_1,\ldots,i_p < k\end{array}$}}} \binom{k}{i_1,\ldots,i_p} \prod_{\ell = 1}^p p_{i_\ell}^{(p)}(x)\] for $j \ge 2$ and $a_{k1}^{(p)} = 1 - \sum_{j = 2}^k a_{kj}^{(p)}$. In particular, $p_k^{(p)}(0) = 0$, $p_k^{(p)}(1) = 1$, and $m_k^{(p)} := a_{kk}^{(p)} >0$ for all $k$.
\end{proposition}
This now yields the similar recursion formula for $m_k^{(p)}$ where $m_0^{(p)}=m_1^{(p)} = 1$ and
\begin{equation}\label{eq: moment rec p}
    m_k^{(p)} = \frac1{p-1} \frac{\lambda_p - 1}{\lambda_p^k - \lambda_p} \sum_{\text{\scriptsize{$\begin{array}{c}i_1 + \cdots + i_p = k\\ i_1,\ldots,i_p < k\end{array}$}}} \binom{k}{i_1,\ldots,i_p} \prod_{\ell = 1}^p m_{i_\ell}^{(p)}
\end{equation}
for $k \ge 2$. Moreover, we also have:
\begin{corollary}
    Let $W_n^{(p)} = Y_n^{(p)} \cdot \lambda_p^{-n}$. Then $\E [{W_n^{(p)}}]^k = m_k^{(p)} \cdot (1 + O(\lambda_p^{-n}))$.
\end{corollary}

We now aim to similarly derive the moment problem is determinant for positive integer moments $m_k^{(p)}$.
\begin{proposition}\label{prop: moment unique p}
    There exists a unique probability distribution with support $(0,\infty)$ having positive integer moments $m_k^{(p)}$.
\end{proposition}
\begin{proof}
    We first define the auxiliary function
    \begin{equation*}
        H(x) = \frac1{p-1} \frac{\lambda_p - 1}{\lambda_p^x - \lambda_p} \left[\frac{\Gamma(p + x)}{\Gamma(x+1)\Gamma(p)} - p \right].
    \end{equation*}
    Using Stirling's approximation, we note then $H(x) = O_x\left({x^{p-1}}/{\lambda_p^x}\right) = o_x(1)$. Let $M$ be a sufficiently large integer such that $|H(x)| = H(x) \le 1$ for all $x \ge M$. 

    Now we can show $m_k^{(p)} \le (m_M^{(p)})^{k/M} \cdot k!$ for all $k$, which again yields that $m_k^{(p)}$ satisfy Carleman's criterion for the Stieljes moment problem to thus finish the proof. We will use induction on $k$. Note for $r \ge s$, then 
    \begin{align*}
        [m_{r}^{(p)}]^{1/r}\cdot (1 + O(\lambda_p^{-n/r})) &= (\E[W_n^{(p)}]^{r})^{1/r} \ge (\E[W_n^{(p)}]^{s})^{1/s} \\
        &= [m_s^{(p)}]^{1/s}\cdot (1 + O(\lambda_p^{-n/s})), 
    \end{align*}
    and so $[m_{r}^{(p)}]^{1/r} \ge [m_s^{(p)}]^{1/s}$ after taking limits with $n \to\infty$. It follows then for all $k \le M$,  we have $m_k^{(p)} \le [m_M^{(p)}]^{k/M} \le [m_M^{(p)}]^{k/M} \cdot k!$. Now assume the result holds for all $\ell < k$ where  $k \ge M$. It follows then
    \begin{align*}
        m_k^{(p)} &\le [m_M^{(p)}]^{k/M}\cdot  k! \cdot H(k) \le [m_M^{(p)}]^{k/M} \cdot k!,
    \end{align*}
    using the inductive hypothesis for the first inequality and now $H(k) \le 1$ since $k \ge M$ for the last inequality.
\end{proof}

We can now establish an analogous CLT result for the number of cycles of a uniform permutation from any $p$-Sylow subgroups of $S_{p^n}$:
\begin{theorem}\label{thm: cycles p}
    Let $\sigma_n \sim \Unif(B_n^{(p)})$ where $p$ is prime. Then $W_n^{(p)} = C(\sigma_n)\cdot \lambda_p^{-n}$ converges in distribution to $W^{(p)}$ with support on $(0,\infty)$, where $\E [W^{(p)}]^k = m_k^{(p)}$ for all positive integers $k$, for $m_0^{(p)} = m_1^{(p)} = 1$ and $m_k^{(p)}$ satisfying the recursion formula \eqref{eq: moment rec p}.
\end{theorem}

Similarly, then the moment generating function $M_{W^p}(z)$ satisfies the functional equation
\begin{equation*}
    p \xi ( \lambda_p z) = \xi(z)^p + (p-1)\xi(z)
\end{equation*}
with $\xi(0) = \xi'(0) = 1$. The argument in \Cref{prop: moment unique p} yields $M_{W^{(p)}}(z)$ is well-defined on the neighborhood of the origin $|z| < 1/[m_M^{(p)}]^{1/M}$\footnote{For $p = 2$, we saw $M = 7$ sufficed for the argument used in \Cref{prop: moment unique p}. Then this would only establish $M_{W^{(2)}}(z)$ is defined on a neighborhood of 0 of radius $1/[m_7^{(2)}]^{1/7} \approx 0.589308$, which is smaller than $\lambda_2 = 1.5$ guaranteed using the argument in \Cref{prop: unique measure 2}.}; this can be similarly be shown to be defined on all of $\mathbb R$. Moreover, we can now also define the $p$-butterfly Stirling numbers 
\begin{equation*}
    s_B^{(p)}(n,k) = \#\{\sigma \in B_n^{(p)}: C(\sigma) = k\} = |B_n^{(p)}| \cdot \P(Y_n^{(p)} = k), 
\end{equation*}
with $\sum_{k = 1}^{p^n} s_B^{(p)}(n,k) = |B_n^{(p)}| = p^{(p^n-1)/(p-1)}$, which now satisfy $s_B^{(p)}(n,k) = 0$ if $k \not\equiv 1 \pmod{(p-1)}$ and otherwise satisfy the recursion
\begin{equation*}
    s^{(p)}_B(n+1,k) = (p-1) p^{p^n - 1} s^{(p)}_B(n,k) + \sum_{\text{\scriptsize{$\begin{array}{c}i_1 + \cdots + i_p = k\\ i_j \equiv 1 \pmod{(p-1)}\end{array}$}}} \prod_{\ell = 1}^p s^{(p)}_B(n,i_\ell).
\end{equation*}
The condition that the support of $Y_n^{(p)}$ is contained on $k \equiv 1\pmod{(p-1)}$ inductively follows from the initial requirement $s_B^{(p)}(1,1) = p-1$, $s_B^{(p)}(1,p) = 1$, and $s_B^{(p)}(1,k) = 0$ if $k \ne 1,p$. In particular, $s_B^{(p)}(n,k) > 0$ only if $k = j(p-1) + 1$ for $j = 0,1,\ldots,\frac{p^n-1}{p-1}$. \Cref{t:snk 3,t:snk 5} show exact values of $p$-butterfly Stirling numbers for $p = 3,5$ and small $n$.

\begin{table}[ht!]
\centering
{
\begin{tabular}{r|ccccccc}
&\multicolumn{6}{c}{$k$}\\
$n$ & 1 & 3 & 5 & 7 & 9 & 11 &\ldots \\ \hline 
1 & 2& 1\\
2 & 36& 26& 12& 6 & 1\\
3 & 472392& 387828& 258552& 198396& 121418& 77472& \ldots\\
\end{tabular}
}
\caption{Triangle of $s_{B}^{(3)}(n,k)$ values for $n = 1$ to 3}
\label{t:snk 3}
\end{table}

\begin{table}[ht!]
\centering
{
\begin{tabular}{r|ccccccc}
&\multicolumn{7}{c}{$k$}\\
$n$ & 1 & 5 & 9 & 13 & 17 & 21 & 25  \\ \hline 
1 & 4& 1\\
2 & 10000& 3524& 1280& 640 & 160&20&1\\
3 & 2384185791015625000000 & \ldots 
\end{tabular}
}
\caption{Triangle of $s_{B}^{(5)}(n,k)$ values for $n = 1$ to 3}
\label{t:snk 5}
\end{table}

Moreover, note, for instance, $s_B^{(p)}(n,p^n) = 1$ necessarily, while now in general $s_B^{(p)}(n,1) = (p-1)^n p^{\frac{p^n-1}{p-1} - (n-1)}$ so that $\P(Y_n^{(p)} = 1) = p \left(1-\frac1p\right)^n$. 

\begin{figure}[t]
    \centering
    \includegraphics[width=0.75\linewidth]{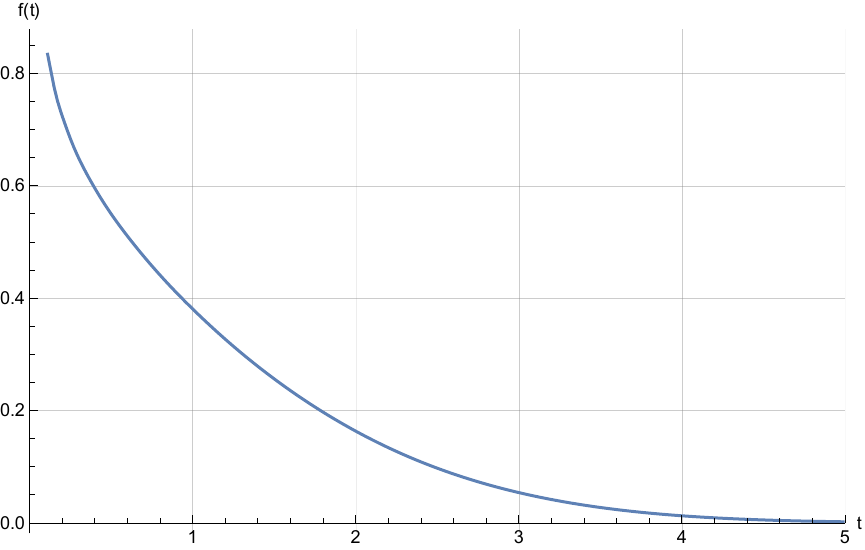}
    \caption{Plot of density function for $W^{(3)}$.}
    \label{fig:cycle limit pdf-3}
\end{figure}

We can similarly determine a pointwise limit formula for the limiting density function:
\begin{proposition}\label{prop: limit ft p}
    Let $f_p(t)$ be the density for $W^{(p)}$. Then
\begin{equation*}
    f_p(t) = \lim_{n\to \infty} \P \left( Y_n^{(p)} = \left\lfloor \frac{ t \lambda_p^n - 1}{p-1} \right\rfloor(p-1) + 1 \right) \cdot \frac{\lambda_p^n}{p-1}.
\end{equation*}
\end{proposition}
In particular, note then
\begin{equation*}
    \lim_{t \to 0+} f_p(t) = \lim_{n\to\infty} \P(Y_n^{(p)} = 1) \cdot \frac{\lambda_p^n}{p-1} = \frac{p}{p-1} \cdot \lim_{n \to\infty} \left(2 - \frac{3p - 1}{p^2} \right)^n.
\end{equation*}
Let $n_p^* = \argmax_k s_B^{(p)}(n,k)$ denote the mode of $Y_n^{(p)}$. When $p = 2$, then $n_2^*$ (empirically) grew logarithmically (cf. \Cref{t:snk}). However, for $p \ge 3$ then it clearly appears $n_p^* = 1$. This is consistent with the following discussion regarding the behavior near zero: 
When $p = 2$, then $2 - \frac{3p-1}{p^2} = \frac34$, so $\lim_{t \to 0+} f_2(t) = 0$, as was evident in \Cref{fig:cycle limit pdf}. Now for $p \ge 3$, then $2 - \frac{3p-1}{p^2} \ge \frac{11}9$, so that $\lim_{t \to 0+} f_p(t) = +\infty$. Hence, the limiting density function has an asymptote at the origin (while still with unit integral over $(0,\infty)$). As $p$ increases, then the mass at the origin shifts to form fatter tails. \Cref{fig:cycle limit pdf-3} shows an approximate plot for density of $W^{(3)}$, using \Cref{prop: limit ft p} to approximate $f_3(t)$ using ternary butterfly permutations of order $3^{9} = 19683$ (which aligned with the $3^8$ approximation of $f_{W^{(3)}}(t)$, with no corresponding point differing by more than 0.001 for $t \ge 0.1$; note $|B_9^{(3)}| = 3^{(3^9-1)/(3-1)} \approx 2.2401 \cdot 10^{4695}$). Furthermore, we see $\P(W^{(3)} \ge 5) \le {m_{10}^{(3)}}{5^{-10}} \approx 0.00634537$, showing most of the mass for the density of $W^{(3)}$ is now contained in $[0,5]$; this again aligns with what is apparent from \Cref{fig:cycle limit pdf-3}.

\subsubsection{Number of fixed points}

We will discuss briefly the question of the number of fixed points, $C_1(\sigma)$, in the general $m$-nary nonsimple butterfly permutation case. We will abstain from exploring other fixed cycle length counts as well as the length of the longest or shortest cycle counts (these are necessarily supported on powers of $p$ for $p$-Sylow subgroups), which can be explored further in future work. For reference, the order of a uniform element of $B_{n}^{(p)}$ is already known, as established by Ab\/ert and Vir\'ag in \cite[Theorem 1]{Abert_Virag_2005} in answer to a question of Tur\'an from the 1980s (cf. \cite{PS83}):
\begin{theorem}[\cite{Abert_Virag_2005}]
    If $\sigma_n \sim \Unif(B_n^{(p)})$ and $\ord(\sigma_n) = p^{K_p(\sigma_n)}$, then $K_p(\sigma_n)/n$ converges in probability to $z_p$, the solution of the equation $z(1 - z)^{1/z - 1} = 1 - 1/p$ in the open unit interval.
\end{theorem}
\noindent Their proof used a novel measure-preserving bijection between conjugacy classes of random elements from $B_{n}^{(p)}$, realized as a group acting on the $p$-nary tree of depth $n$, and Galton-Watson trees. For comparison, this question was significantly easier to answer for simple butterfly permutations, as seen in \Cref{cor: s order}.

Returning to the question of the number of fixed points, recall the known result for $\sigma_n \sim \Unif(S_n)$, where $C_1(\sigma_n)$ satisfies $\E C_1(\sigma_n) = 1$ and converges in distribution to $\operatorname{Poisson}(1)$, and so also $\P(C_1(\sigma_n) = 0)$ converges to $e^{-1} \approx 0.367879$.

Now we consider $\sigma \in B_n^{(m)}$, where $\sigma = (\sigma_\theta \otimes 1)(\bigoplus_{j=1}^m \sigma_j)$ for $\sigma_\theta \in B_1^{(m)}, \sigma_j \in B_{n-1}^{(m)}$. In order for there to be any fixed points in $\sigma$, necessarily $\sigma_\theta = 1$. It follows then $C_1(\sigma) = \mathds 1(\sigma_\theta = 1) \cdot \sum_{j =1}^m C_1(\sigma_j)$. Hence, we have the following characterization:
\begin{proposition}
    If $\sigma_n \sim \Unif(B_n^{(m)})$, then $C_1(\sigma_n) \sim T_n$, for $T_0 = 1$ and $T_{n + 1} \sim \eta_n \cdot \sum_{j = 1}^m T_n^{(j)}$ where $T_n^{(j)} \sim T_n$ iid are independent of $\eta_n \sim \Bern(\frac1m)$.
\end{proposition}
It now follows directly that $\E T_{n+1} = \frac1m \sum_{j = 1}^m \E T_n = \E T_n$ and so $\E T_n = 1$ since $T_0 = 1$; this matches the expected number of fixed points for $\sigma_n \sim \Unif(S_n)$ as well as $\sigma_n \sim \Unif(B_{s,n}^{(p)})$. Next, we have $\E T_{n + 1}^2 = \frac1m \E (\sum_{j = 1}^m T_n^{(j)})^2 = \E T_n^2 + m-1$, so that $\E T_n^2 = n(m-1) + 1$. Similarly, using discrete finite difference sums, we derive also $\E T_n^k = O(n^{k-1})$ for all positive integer $k$. Hence, $\E (T_n n^{-\theta})^k = O(n^{k(1 - \theta) - 1})$, so that $\E (T_n n^{-\theta})^k = O(1)$ only when $\theta \ge 1 - \frac1k$; it follows then the only scaling for $\theta \in [0,1]$ that controls all moments is taking $\theta = 1$ (i.e., scaling $T_n$ by $n = \log_p N$), in which case $T_n/n$ converges to 0 in probability as all moments vanish asymptotically.

A more interesting question can explore $p_n^{(m)} = \P(T_n = 0)$ that determines when $\sigma_n \sim \Unif(B_n^{(m)})$ has no fixed points. Using independence and inclusion-exclusion, we have directly $p_{n + 1}^{(m)} = h_m(p_n^{(m)})$ where $h_m(x) = \frac1m + (1 - \frac1m) x^m$. Note $h_m(1) = 1$, and there is exactly one fixed point $x^*_m \in (0,1)$ of $h_m(x)$ when $m \ge 3$. Note $x^*_m$ is the unique real root of 
\begin{equation*}
    q_m(x) = -1 + (m-1) \sum_{j=1}^{m-1} x^j = -1 + (m-1) \cdot x \cdot \frac{1 - x^{m-1}}{1 - x}, 
\end{equation*}
where $m(h(x) - x) = (x - 1)q_m(x)$ (with $q_2(x) = x-1$ so that $x_2^* = 1$); existence and uniqueness follow directly since there is at most one real root of $q_m(x)$ by Descartes' rule of signs, and at least one real root since $q_m(\frac1m) = - \frac{m}{m^m} < 0$ and $q_m((m-1)^{-1/(m-1)}) = -1 + \frac{m-2}{(m-1)^{1/(m-1)} - 1} > 0$, so $x_m^* \in (\frac1m, (m-1)^{-1/(m-1)})$. Since $h_m(x)$ is strictly convex, $h_m'(1) = m-1 > 1$, $h_m'(0) = 0$, and $h_m(0) = \frac1m > 0$, then $h_m(x) > x$ on $[0,x_m^*)$; it follows then $p_{n} < p_{n+1} \le x_m^*$ for all $n$, and so necessarily $\lim_{n\to\infty} p_n^{(m)} = x_m^*$. Moreover, only $x_m^*$ is a stable fixed point of $h_m(x)$ when $m > 2$ since $h_m'(x) < 1$ when $x < (m-1)^{-1/(m-1)}$ while $h_m'(1) = m-1 > 1$. For $m = 2$, then $h_2'(x_2^*) = h_2'(1) = 1$; it follows then $\P(T_n = 0)$ converges to 1, which yields then $T_n$ converges to 0 in probability. 

For $m \ge 3$, then $x_m^* = \frac1m + O(m^{-m})$; this can be established using the Banach Fixed Point Theorem, as $h_m'(x) < \frac12$ for $|x| < \frac12$, so $|x^*_m - \frac1m| < \frac{\frac12}{1 - \frac12} |h_m(x^*) - \frac1m| = |h_m(x^*) - \frac1m| = (1 - \frac1m)m^{-m}$. Hence, the sequence $x_m^*$ concentrates on the sequence $\frac1m$ as $m$ increases. For instance, we have $x^*_3 = \frac12(-1 + \sqrt 3) \approx 0.366025$ (which is close to $e^{-1}$), $x_5^* \approx 0.200257$, and $x_7^* \approx 0.142858$. In double precision using MATLAB, the computed values of $x_{15}^*$ and $\frac1{15}$ match exactly within the 52-bit precision rounding error. We summarize this as:
\begin{proposition}
    Let $\sigma_n \sim \Unif(B_n^{(m)})$. For $m = 2$, then $C_1(\sigma_n)$ converges in probability to 0. For $m \ge 3$, then $C_1(\sigma_n)/n$ converges in probability to 0 while $C_1(\sigma_n)/n^{1 - \varepsilon}$ does not converge to 0 in probability for any $\varepsilon > 0$, while $\lim_{n \to \infty} \P(C_1(\sigma_n) = 0) = x^*_m$ where $x^*_2 = 1$ and $x^*_m = \frac1m + O(m^{-m})$ is the unique fixed point in the open unit interval of $h_m(x) = \frac1m + (1 - \frac1m) x^m$ for $m \ge 3$.
\end{proposition}

\bibliographystyle{plain} 
\bibliography{references}

\begin{thebibliography}{10}

\bibitem{Abert_Virag_2005}
Miklós Abért and Bálint Virág.
\newblock Dimension and randomness in groups acting on rooted trees.
\newblock {\em Journal of the American Mathematical Society}, 18(1):157–192, 2005.

\bibitem{aldous1995hammersley}
David Aldous and Persi Diaconis.
\newblock Hammersley's interacting particle process and longest increasing subsequences.
\newblock {\em Probability Theory and Related Fields}, 103:199--213, 1995.

\bibitem{aldous1999longest}
David Aldous and Persi Diaconis.
\newblock Longest increasing subsequences: from patience sorting to the {B}aik-{D}eift-{J}ohansson theorem.
\newblock {\em Bulletin of the American Mathematical Society}, 36(4):413--432, 1999.

\bibitem{arratia2003logarithmic}
Richard Arratia, Andrew~D Barbour, and Simon Tavar{\'e}.
\newblock {\em Logarithmic combinatorial structures: a probabilistic approach}, volume~1.
\newblock European Mathematical Society, 2003.

\bibitem{baboulin}
M.~Baboulin, X.S. Li, and F.~Rouet.
\newblock Using random butterfly transformations to avoid pivoting in sparse direct methods.
\newblock {\em In: Proc. of Int. Con. on Vector and Par. Proc.}, 2014.

\bibitem{BDJ99}
Jinho Baik, Percy Deift, and Kurt Johansson.
\newblock On the distribution of the length of the longest increasing subsequence of random permutations.
\newblock {\em J. Amer. Math. Soc.}, 12(4):1119–1178, 1999.

\bibitem{bassino2022linear}
Fr{\'e}d{\'e}rique Bassino, Mathilde Bouvel, Michael Drmota, Valentin Feray, Lucas Gerin, Micka{\"e}l Maazoun, and Adeline Pierrot.
\newblock Linear-sized independent sets in random cographs and increasing subsequences in separable permutations.
\newblock {\em Combinatorial Theory}, 2(3):https--escholarship, 2022.

\bibitem{bassino2018brownian}
Fr{\'e}d{\'e}rique Bassino, Mathilde Bouvel, Valentin F{\'e}ray, Lucas Gerin, and Adeline Pierrot.
\newblock The {Brownian} limit of separable permutations.
\newblock {\em The Annals of Probability}, 46(4):2134--2189, 2018.

\bibitem{Basu_Bhatnagar_2017}
Riddhipratim Basu and Nayantara Bhatnagar.
\newblock Limit theorems for longest monotone subsequences in random {M}allows permutations.
\newblock {\em Ann. Inst. Henri Poincar\'e{} Probab. Stat.}, 53(4):1934--1951, 2017.

\bibitem{bhatnagar2015lengths}
Nayantara Bhatnagar and Ron Peled.
\newblock Lengths of monotone subsequences in a {Mallows} permutation.
\newblock {\em Probability Theory and Related Fields}, 161(3):719--780, 2015.

\bibitem{Borga21}
Jacopo Borga.
\newblock {\em Random Permutations--A geometric point of view}.
\newblock {PhD} thesis, Universität Zürich, 2021.
\newblock arXiv preprint arXiv:2107.09699.

\bibitem{BDG24}
Jacopo Borga, William Da~Silva, and Ewain Gwynne.
\newblock Power-law bounds for increasing subsequences in {B}rownian separable permutons and homogeneous sets in {B}rownian cographons.
\newblock {\em Advances in Mathematics}, 439:109480, 2024.

\bibitem{Borodin_1999}
Alexei Borodin.
\newblock Longest increasing subsequences of random colored permutations.
\newblock {\em The Electronic Journal of Combinatorics}, 6(1):Research paper R13, 12 p., 1999.

\bibitem{BOO}
Alexei Borodin, Andrei Okounkov, and Grigori Olshanski.
\newblock Asymptotics of {P}lancherel measures for symmetric groups.
\newblock {\em J. Amer. Math. Soc.}, 13(3):481--515, 2000.

\bibitem{chatterjee2024vershik}
Sourav Chatterjee and Persi Diaconis.
\newblock A {V}ershik-{K}erov theorem for wreath products.
\newblock {\em arXiv preprint arXiv:2408.04364}, 2024.

\bibitem{corwin2018comments}
Ivan Corwin.
\newblock Commentary on ``{L}ongest increasing subsequences: from patience sorting to the {B}aik-{D}eift-{J}ohansson theorem'' by {D}avid {A}ldous and {P}ersi {D}iaconis.
\newblock {\em Bulletin of the American Mathematical Society}, 55(3):363--374, July 2018.

\bibitem{Crane16}
Harry Crane.
\newblock The ubiquitous {E}wens sampling formula.
\newblock {\em Statist. Sci.}, 31(1):1--19, 2016.

\bibitem{dauvergne2021scaling}
Duncan Dauvergne and B{\'a}lint Vir{\'a}g.
\newblock The scaling limit of the longest increasing subsequence.
\newblock {\em arXiv preprint arXiv:2104.08210}, 2021.

\bibitem{deuschel1995limiting}
Jean-Dominique Deuschel and Ofer Zeitouni.
\newblock Limiting curves for iid records.
\newblock {\em The Annals of Probability}, pages 852--878, 1995.

\bibitem{deuschel1999increasing}
Jean-Dominique Deuschel and Ofer Zeitouni.
\newblock On increasing subsequences of {IID} samples.
\newblock {\em Combinatorics, Probability and Computing}, 8(3):247--263, 1999.

\bibitem{deutsch2003longest}
Emeric Deutsch, AJ~Hildebrand, and Herbert~S Wilf.
\newblock Longest increasing subsequences in pattern-restricted permutations.
\newblock {\em The Electronic Journal of Combinatorics}, 9(2):R12, 2003.

\bibitem{DiSh87}
P.~Diaconis and M.~Shahshahani.
\newblock The subgroup algorithm for generating uniform random variables.
\newblock {\em Prob. in the Eng. and Info. Sci.}, 1:15--32, 1987.

\bibitem{diaconis1988group}
Persi Diaconis.
\newblock {\em Group Representations in Probability and Statistics}.
\newblock Lecture Notes-Monograph Series. Institute of Mathematical Statistics, Hayward, CA, 1988.

\bibitem{diaconis2022statistical}
Persi Diaconis and Mackenzie Simper.
\newblock Statistical enumeration of groups by double cosets.
\newblock {\em Journal of Algebra}, 607:214--246, 2022.

\bibitem{Diaconis_Tung_2024}
Persi Diaconis and Nathan Tung.
\newblock Poisson approximation for large permutation groups.
\newblock {\em arXiv preprint arXiv:2408.06611}, 2024.

\bibitem{dubach2024increasing}
Victor Dubach.
\newblock Increasing subsequences of linear size in random permutations and the {R}obinson--{S}chensted tableaux of permutons.
\newblock {\em Random Structures \& Algorithms}, 65(3):488--534, 2024.

\bibitem{Durrett_2010}
Rick Durrett.
\newblock {\em Probability: Theory and Examples}, volume~49.
\newblock Cambridge University Press, Cambridge, fourth edition, 2010.

\bibitem{feng1}
Shui Feng.
\newblock {\em The {P}oisson-{D}irichlet distribution and related topics}.
\newblock Probability and its Applications (New York). Springer, Heidelberg, 2010.
\newblock Models and asymptotic behaviors.

\bibitem{FY38}
R.~A. Fisher and F.~Yates.
\newblock {\em Statistical tables for biological, agricultural and medical research}.
\newblock Oliver \& Boyd, 1938.

\bibitem{gladkich2018cycle}
Alexey Gladkich and Ron Peled.
\newblock On the cycle structure of {M}allows permutations.
\newblock {\em The Annals of Probability}, 46(2):1114–1169, 2018.

\bibitem{Goncharov}
V.~Gontcharoff.
\newblock Sur la distribution des cycles dans les permutations.
\newblock {\em C. R. (Doklady) Acad. Sci. URSS (N.S.)}, 35:267–--269, 1942.

\bibitem{GIMPS2024}
{Great Internet Mersenne Prime Search (GIMPS)} and L.~Durant.
\newblock Largest known prime number.
\newblock \url{https://www.mersenne.org/}, 2024.
\newblock Discovered prime: \(2^{136279841}-1\).

\bibitem{he2023cycles}
Jimmy He, Tobias M{\"u}ller, and Teun~W Verstraaten.
\newblock Cycles in {Mallows} random permutations.
\newblock {\em Random Structures \& Algorithms}, 63(4):1054--1099, 2023.

\bibitem{hoffman2017pattern}
Christopher Hoffman, Douglas Rizzolo, and Erik Slivken.
\newblock Pattern-avoiding permutations and {B}rownian excursion, part {II}: fixed points.
\newblock {\em Probability Theory and Related Fields}, 169:377--424, 2017.

\bibitem{HT23}
H.~Huang and K.~Tikhomirov.
\newblock Average-case analysis of the {G}aussian elimination with partial pivoting.
\newblock {\em Probability Theory and Related Fields}, 189:501--567, 2024.

\bibitem{Kaloujnine_1948}
Léo Kaloujnine.
\newblock La structure des $p$-groupes de {Sylow} des groupes symétriques finis.
\newblock {\em Annales scientifiques de l'École normale supérieure}, 65:239–276, 1948.

\bibitem{kammoun2}
Mohamed~Slim Kammoun.
\newblock Universality for random permutations and some other groups.
\newblock {\em Stochastic Process. Appl.}, 147:76--106, 2022.

\bibitem{kingman1978representation}
John~FC Kingman.
\newblock The representation of partition structures.
\newblock {\em Journal of the London Mathematical Society}, 2(2):374--380, 1978.

\bibitem{logan1977variational}
Benjamin~F Logan and Larry~A Shepp.
\newblock A variational problem for random {Young} tableaux.
\newblock {\em Advances in Mathematics}, 26(2):206--222, 1977.

\bibitem{madras2017longest}
Neal Madras and G{\"o}khan Y{\i}ld{\i}r{\i}m.
\newblock Longest monotone subsequences and rare regions of pattern-avoiding permutations.
\newblock {\em The Electronic Journal of Combinatorics}, pages P4--13, 2017.

\bibitem{mallows1957non}
Colin~L Mallows.
\newblock Non-null ranking models. i.
\newblock {\em Biometrika}, 44(1/2):114--130, 1957.

\bibitem{mansour2020permutations}
Toufik Mansour and G{\"o}khan Y{\i}ld{\i}r{\i}m.
\newblock Permutations avoiding 312 and another pattern, {Chebyshev} polynomials and longest increasing subsequences.
\newblock {\em Advances in Applied Mathematics}, 116:102002, 2020.

\bibitem{Mezz}
Francesco Mezzadri.
\newblock How to generate random matrices from the classical compact groups.
\newblock {\em Notices of the American Mathematical Society}, 54(5):592 -- 604, May 2007.

\bibitem{Mueller_Starr_2013}
Carl Mueller and Shannon Starr.
\newblock The length of the longest increasing subsequence of a random {M}allows permutation.
\newblock {\em J. Theoret. Probab.}, 26(2):514--540, 2013.

\bibitem{mukherjee}
Sumit Mukherjee.
\newblock Fixed points and cycle structure of random permutations.
\newblock {\em Electron. J. Probab.}, 21:Paper No. 40, 18, 2016.

\bibitem{Okounkov}
Andrei Okounkov.
\newblock Random matrices and random permutations.
\newblock {\em Internat. Math. Res. Notices}, 2000(20):1043--1095, 2000.

\bibitem{PS83}
{P. P.} P\'alfy and M.~Szalay.
\newblock On a problem of {P}. {T}ur\'an concerning {S}ylow subgroups.
\newblock {\em In: {S}tudies in pure mathematics}, pages 531--542, 1983.

\bibitem{Pa95}
D.~Stott Parker.
\newblock Random butterfly transformations with applications in computational linear algebra.
\newblock {\em Tech. rep., UCLA}, 1995.

\bibitem{phd}
John Peca-Medlin.
\newblock {\em Numerical, spectral, and group properties of random butterfly matrices}.
\newblock {PhD} thesis, University of California, Irvine, 2021.
\newblock ProQuest ID: PecaMedlin\_uci\_0030D\_17221. Merritt ID: ark:\/13030\/m5ck4tnm.

\bibitem{P24_gecp}
John Peca-Medlin.
\newblock Complete pivoting growth of butterfly matrices and butterfly {H}adamard matrices.
\newblock {\em arXiv preprint arXiv:2410.06477}, 2024.

\bibitem{P24}
John Peca-Medlin.
\newblock Distribution of the number of pivots needed using {G}aussian elimination with partial pivoting on random matrices.
\newblock {\em Ann. Appl. Probab.}, 34(2):2294--2325, 2024.

\bibitem{PT23}
John Peca-Medlin and Thomas Trogdon.
\newblock Growth factors of random butterfly matrices and the stability of avoiding pivoting.
\newblock {\em SIAM J. Matrix Anal. Appl.}, 44(3):945--970, 2023.

\bibitem{Romik_2015}
Dan Romik.
\newblock {\em The Surprising Mathematics of Longest Increasing Subsequences}, volume Series Number 4.
\newblock Cambridge University Press, West Nyack, 1 edition, 2015.

\bibitem{Sachkov}
V.~N. Sachkov.
\newblock {\em Probabilistic methods in combinatorial analysis, {V}olume 56 of {E}ncyclopedia of {M}athematics and its {A}pplications}.
\newblock Cambridge University Press, Cambridge, 1997.

\bibitem{sjostrand2023monotone}
Jonas Sj{\"o}strand.
\newblock Monotone subsequences in locally uniform random permutations.
\newblock {\em The Annals of Probability}, 51(4):1502--1547, 2023.

\bibitem{kammoun}
Mohamed Slim~Kammoun.
\newblock Monotonous subsequences and the descent process of invariant random permutations.
\newblock {\em Electron. J. Probab.}, 23:Paper no. 118, 31, 2018.

\bibitem{slim2024small}
Mohamed Slim~Kammoun and Myl{\`e}ne Ma{\"\i}da.
\newblock Small cycle structure for words in conjugation invariant random permutations.
\newblock {\em Random Structures \& Algorithms}, 64(4):918--939, 2024.

\bibitem{Tracy_Widom_1993}
Craig~A Tracy and Harold Widom.
\newblock Level-spacing distributions and the {A}iry kernel.
\newblock {\em Physics Letters B}, 305(1):115–118, 1993.

\bibitem{Tracy_Widom_1994}
Craig~A Tracy and Harold Widom.
\newblock Level-spacing distributions and the {A}iry kernel.
\newblock {\em Communications in Mathematical Physics}, 159(1):151–174, 1994.

\bibitem{Tracy_Widom_1996}
Craig~A Tracy and Harold Widom.
\newblock On orthogonal and symplectic matrix ensembles.
\newblock {\em Communications in Mathematical Physics}, 177(3):727–754, 1996.

\bibitem{Tr19}
Thomas Trogdon.
\newblock On spectral and numerical properties of random butterfly matrices.
\newblock {\em Applied Math. Letters}, 95(4):48--58, September 2019.

\bibitem{Ulam61}
{S. M.} Ulam.
\newblock {\em Monte Carlo Calculations in Problems of Mathematical Physics}, pages 261--281.
\newblock McGraw-Hill Book Co., Inc., New York-Toronto-London, 1961.

\bibitem{vershik1977asymptotics}
Anatolii~Moiseevich Vershik and Sergei~Vasil'evich Kerov.
\newblock Asymptotics of the {Plancherel} measure of the symmetric group and the limiting form of {Young} tableaux.
\newblock In {\em Doklady akademii nauk}, volume 233, pages 1024--1027. Russian Academy of Sciences, 1977.

\bibitem{We40}
And\'e Weil.
\newblock {\em L'int\'egration dans les groupes topologiques et ses applications, \emph{Actualit\'es Scientifiques et Industrielles}}, volume 869.
\newblock Paris: Hermann, 1940.

\bibitem{zhong2023cycle}
Chenyang Zhong.
\newblock Cycle structure of {M}allows permutation model with the {$L^1$} distance.
\newblock {\em arXiv preprint arXiv:2312.15833}, 2023.

\bibitem{Zhong_2023}
Chenyang Zhong.
\newblock The length of the longest increasing subsequence of {Mallows} permutation models with {$L^1$} and {$L^2$} distances.
\newblock {\em arXiv preprint arXiv:2303.09688}, 2023.

\end{thebibliography}

\end{document}